\documentclass[11pt, francais]{smfart}
\usepackage{amsfonts, amssymb, amsmath, latexsym, enumerate, epsfig, color, mathrsfs, fancyhdr, supertabular, stmaryrd, smfthm, yhmath, mathtools, pifont, chngcntr, pinlabel, ifthen}
\usepackage[all]{xy}
\usepackage{tikz-cd}
\usepackage[french,refpage]{nomencl}

\usepackage[french,english]{babel}

\usepackage[T1]{fontenc}
\usepackage[mac]{inputenc}

\usepackage{hyperref}

\theoremstyle{plain} 
\newtheorem{theointro}{Th\'eor\`eme}

\newtheorem{corointro}[theointro]{Corollaire}
\newtheorem{propintro}[theointro]{Proposition}

\theoremstyle{definition} 
\newenvironment{nota}{\begin{enonce}[definition]{Notation}}{\end{enonce}}

\DeclareFontFamily{U}{mathc}{}
\DeclareFontShape{U}{mathc}{m}{it}%
{<->s*[1.03] mathc10}{}
\DeclareMathAlphabet{\mathcal}{U}{mathc}{m}{it}

\newcommand{\N}{\ensuremath{\mathbf{N}}}
\newcommand{\Z}{\ensuremath{\mathbf{Z}}}

\newcommand{\Q}{\ensuremath{\mathbf{Q}}}

\newcommand{\R}{\ensuremath{\mathbf{R}}}
\newcommand{\C}{\ensuremath{\mathbf{C}}}

\renewcommand{\P}{\ensuremath{\mathbf{P}}}
\newcommand{\F}{\ensuremath{\mathbf{F}}}

\newcommand{\kF}{\mathfrak{F}}

\newcommand{\E}[2]{\ensuremath{\mathbf{A}^{#1,\mathrm{an}}_{#2}}}
\newcommand{\EP}[2]{\ensuremath{\mathbf{P}^{#1,\mathrm{an}}_{#2}}}

\newcommand{\oD}{\overline{D}}

\newcommand{\too}{\longrightarrow}
\newcommand{\simtoo}{\overset{\sim}{\longrightarrow}}
\newcommand{\mapstoo}{\longmapsto}

\newcommand{\ho}[1]{\mathbin{\hat{\otimes}_{#1}}}

\DeclareMathOperator{\rang}{rang}

\DeclareMathOperator{\Frac}{Frac}

\DeclareMathOperator{\Gal}{Gal}

\DeclareMathOperator{\Met}{M\acute{e}t}

\newcommand{\id}{\mathrm{id}}

\newcommand{\st}{\mathrm{st}}
\renewcommand{\div}{\mathrm{div}}

\newcommand{\hyb}{\mathrm{hyb}}

\newcommand{\simto}{\xrightarrow[]{\sim}}

\newcommand*\diff{\mathop{}\!\mathrm{d}}

\DeclareMathOperator{\Supp}{Supp}
\newcommand{\pr}{\mathrm{pr}}
\DeclareMathOperator{\Mes}{Mes}

\DeclareMathOperator{\SH}{SH}
\DeclareMathOperator{\DSH}{DSH}

\DeclareMathOperator{\Exc}{Exc}

\DeclareMathOperator{\MCL}{MCL}

\newcommand{\m}{\ensuremath{\mathfrak{m}}}

\newcommand{\cA}{\mathcal{A}}
\newcommand{\cB}{\mathcal{B}}
\newcommand{\cC}{\mathcal{C}}

\newcommand{\cF}{\mathcal{F}}

\newcommand{\cH}{\mathcal{H}}
\newcommand{\cI}{\mathcal{I}}

\newcommand{\cK}{\mathcal{K}}
\newcommand{\cM}{\mathcal{M}}

\newcommand{\cO}{\mathcal{O}}

\newcommand{\cn}[2]{\ensuremath{\llbracket{#1},{#2}\rrbracket}}

\DeclarePairedDelimiter{\intoo}{]}{[}
\DeclarePairedDelimiter{\intof}{]}{]}
\DeclarePairedDelimiter{\intfo}{[}{[}
\DeclarePairedDelimiter{\intff}{[}{]}

\newcommand{\la}{\ensuremath{\langle}}
\newcommand{\ra}{\ensuremath{\rangle}}

\newcommand{\eps}{\ensuremath{\varepsilon}}

\newcommand{\disp}{\displaystyle}

\DeclarePairedDelimiter\abs{\lvert}{\rvert}

\DeclarePairedDelimiter\Bigabs{\Big\lvert}{\Big\rvert}

\DeclarePairedDelimiter\norm{\lVert}{\rVert}

\newcommand\wc{{\mkern 2mu\cdot\mkern 2mu}}
\newcommand\va{\abs{\wc}}
\newcommand\nm{\norm{\wc}}
\newcommand\cf{\textit{cf}.}

\newcommand{\fonction}[5]{\begin{array}{cccc}
#1 \colon &#2&\too&#3\\
&#4&\mapstoo&#5
\end{array}}

\countdef \commentaires = 1
\newcommand{\JP}[1]{\ifnum\commentaires = 1{\color{magenta}{#1}}\fi}

\def \spectriv(#1,#2,#3,#4,#5,#6){
\foreach \x in {0,...,#6}
\draw[line width=.001pt] ({#1},{#2}) -- ({#1+#5*cos((\x)/(#6)*(360/#4)+#3)},{#2+#5*sin((\x)/(#6)*(360/#4)+#3)}) ;}

\def \spectrivcentre(#1,#2,#3,#4,#5,#6){
\foreach \x in {-#6,...,#6}
\draw[line width=.001pt] ({#1},{#2}) -- ({#1+#5*cos((\x)/(#6)*(360/#4)+#3)},{#2+#5*sin((\x)/(#6)*(360/#4)+#3)}) ;}

\def \spectrivcentreenplus(#1,#2,#3,#4,#5,#6,#7,#8,#9){
\foreach \y in {-#5,...,#5}
\spectrivcentre(#1+#6*cos((\y)/(#5)*(360/#4)+#3),#2+#6*sin((\y)/(#5)*(360/#4)+#3),(\y)/(#5)*(360/#4)+#3,#7,#8,#9) 
;}
\begin{document}

\frontmatter

\title[Dynamique analytique sur~$\Z$. I]{Dynamique analytique sur~$\Z$. I~: Mesures d'\'equilibre sur une droite projective relative.}
\alttitle{Analytic dynamics over~$\Z$. I: Equilibrium measures on a relative projective line.}

\author{J\'er\^ome Poineau}
\address{Normandie Univ., UNICAEN, CNRS, Laboratoire de math\'ematiques Nicolas Oresme, 14000 Caen, France}
\email{\href{mailto:jerome.poineau@unicaen.fr}{jerome.poineau@unicaen.fr}}
\urladdr{\url{https://poineau.users.lmno.cnrs.fr/}}

\date{\today}

\makeatletter
\@namedef{subjclassname}{Classification math\'ematique par sujets \textup{(2020)}}
\makeatother

\altkeywords{Berkovich spaces over~$\Z$, hybrid spaces, equilibrium measures, potential theory}

\begin{abstract}
Consid\'erons un espace de Berkovich sur un bon anneau de Banach et la droite projective relative sur celui-ci. (C'est un espace dont les fibres sont des droites projectives sur diff\'erents corps valu\'es complets.) Pour tout endomorphisme polaris\'e de cette droite, nous montrons que la famille des mesures d'\'equilibre associ\'ees aux restrictions de l'endomorphisme aux fibres est continue. Le r\'esultat vaut, par exemple, lorsque l'anneau de Banach est un corps valu\'e complet, un corps hybride, un anneau de valuation discr\`ete complet ou un anneau d'entiers de corps de nombres.
%Consid\'erons une droite projective relative sur un espace de Berkovich sur~$\Z$ (dont les fibres sont des droites projectives sur diff\'erents corps valu\'es) et un endomorphisme polaris\'e de cette droite. Nous d\'emontrons que la famille des mesures d'\'equilibre associ\'ees aux restrictions de l'endomorphisme aux fibres est continue. Le r\'esultat reste valable en rempla\c cant~$\Z$ par un corps valu\'e complet, un corps hybride, un anneau de valuation discr\`ete ou un anneau d'entiers de corps de nombres.
\end{abstract}

\begin{altabstract}
Consider a Berkovich space over a good Banach ring and the relative projective line over it. (It is a space whose fibers are projective lines over different complete valued fields.) For each polarized endomorphism of this line, we prove that the family of equilibrium measures associated to the restrictions of the endomorphism to the fibers is continuous. The result holds, in particular, when the Banach ring is a complete valued field, a hybrid field, a complete discrete valuation ring, or the ring of integers of a number field.

%Consider a relative projective line over a Berkovich space over~$\Z$ (whose fibers are projective lines over different valued fields) and a polarized endomorphism of this line. We prove that the family of equilibrium measures associated to the restrictions of the endomorphism to the fibers is continuous. The result still holds if $\Z$ is replaced by a complete valued field, a hybrid field, a discrete valuation ring or the ring of integers of a number field.
\end{altabstract}

\maketitle

\tableofcontents

\mainmatter

\section{Introduction}

Dans ce texte, nous initions l'\'etude des syst\`emes dynamiques dans le cadre des espaces analytiques sur~$\Z$, au sens de Vladimir G. Berkovich (\cf~\cite{rouge}), ou sur d'autres anneaux de Banach aux propri\'et\'es similaires. Ces espaces m\^elent naturellement fibres archim\'ediennes (espaces analytiques complexes usuels, \'eventuellement \`a conjugaison pr\`es) et fibres ultram\'etriques (espaces de Berkovich classiques sur les diff\'erents~$\Q_{p}$ et d'autres corps valu\'es). Aussi les syst\`emes dynamiques que nous consid\'erons se pr\'esentent-ils comme des familles de syst\`emes dynamiques sur des espaces analytiques de diff\'erente nature. L'objectif de ce texte est de d\'emontrer que, sous des hypoth\`eses tr\`es g\'en\'erales, les mesures d'\'equilibre associ\'ees forment des familles continues de mesures.

\subsection{\'Enonc\'e des r\'esultats} 

Commen\c cons par quelques rappels sur la notion de mesure d'\'equilibre, tout d'abord dans le cas de la droite projective complexe~$\P^1(\C)$. Soit $\varphi$ une fraction rationnelle complexe de degr\'e~$d$. Pour $a \in \P^1(\C)$, notons $\delta_{a}$ la mesure de Dirac support\'ee au point~$a$. Pour $n\in \N$, notons $a_{1},\dotsc,a_{d^n}$ les solutions de l'\'equation $\varphi^n = a$, compt\'ees avec multiplicit\'e, et posons
\[ [ \varphi^{-n}(a) ] :=\sum_{i=1}^{d^n} \delta_{a_{i}}.\]
Rappelons que l'on d\'efinit l'ensemble exceptionnel de~$\varphi$ par
\[ \Exc(\varphi) := \{ a \in \P^1(\C) : \varphi^{-1}(a) = \{a\} \} \]
et qu'il contient au plus 2 \'el\'ements. Les travaux de H.~Brolin \cite{BrolinInvariantSets}, dans le cas polynomial, puis de M.~Lyubich \cite{LyubichEntropy} et A.~Freire-A.~Lopez-R.~Ma\~n\'e \cite{FreireLopezMane}, dans le cas g\'en\'eral, montrent qu'il existe une mesure de probabilit\'e~$\mu_{\varphi}$ sur~$\P^1(\C)$ telle que l'on ait
\[ \lim_{n\to +\infty}  \frac1{d^n}  [ \varphi^{-n}(a) ] = \mu_{\varphi},\]
pour tout $a\in \P^1(\C)$ hors de l'ensemble exceptionnel~$\Exc(\varphi)$.
% (celui des points dont les ant\'ec\'edents par l'ensemble des it\'er\'ees de~$\varphi$ sont en nombre fini) contenant au plus 2 \'el\'ements. 
La mesure~$\mu_{\varphi}$ ainsi d\'efinie, dite \emph{mesure d'\'equilibre} associ\'ee \`a~$\varphi$, satisfait deux propri\'et\'es fondamentales, qui, prises ensemble, la caract\'erisent (\cf~\cite[d\'ebut de II]{ManeHausdorffDimension}) :
\begin{enumerate}[i)]
\item $\varphi^*\mu_{\varphi} = d \, \mu_{\varphi}$ ;
\item le support de~$\mu_{\varphi}$ est disjoint de~$\Exc(\varphi)$.
%\item $\mu_{\varphi}$ ne charge aucun point de~$\P^1(\C)$ ;
\end{enumerate}
On v\'erifie, en outre, que le support de~$\mu_{\varphi}$ n'est autre que l'ensemble de Julia~$J(\varphi)$ associ\'ee \`a~$\varphi$.

Pla\c cons-nous maintenant sur un corps valu\'e ultram\'etrique complet~$k$, et rempla\c cons la droite projective complexe~$\P^1(\C)$ par la droite projective~$\EP{1}{k}$ au sens de Berkovich. Nonobstant le caract\`ere totalement discontinu de~$k$, l'espace~$\EP{1}{k}$ jouit de bonnes propri\'et\'es topologiques, telles la compacit\'e et la connexit\'e par arcs locales, et propose un cadre pertinent o\`u adapter les th\'eories complexes classiques.
%qui permettent d'adapter de nombreux raisonnement classiques.
%qui permettent d'y appliquer la th\'eorie des mesures de Radon. 
Ch.~Favre et J.~Rivera--Letelier l'ont ainsi utilis\'e pour associer \`a toute fraction rationnelle~$\varphi$ \`a coefficients dans~$k$ une mesure d'\'equilibre~$\mu_{\varphi}$, poss\'edant les m\^emes propri\'et\'es que son analogue classique,
%\footnote{Pr\'ecisons un point subtil. Dans l'\'enonc\'e de la propri\'et\'e (i) ci-dessus, il faut demander que $\mu_{\varphi}$ ne charge aucun point rigide de~$\EP{1}{k}$, c'est-\`a-dire aucun point \`a valeurs dans une extension finie de~$k$. Il peut arriver que d'autres points de~$\EP{1}{k}$ soient affect\'es d'une masse non nulle.}, 
\cf~\cite{FRLBrolin}. 
%(Pr\'ecisons qu'il faut, dans l'\'enonc\'e de la propri\'et\'e (i) ci-dessus, demander que $\mu_{\varphi}$ ne charge aucun point rigide de~$\EP{1}{k}$, c'est-\`a-dire aucun point \`a valeurs dans une extension finie de~$k$. Il peut arriver que d'autres points de~$\EP{1}{k}$ soient affect\'es d'une masse non nulle.) 
%\begin{enumerate}[i)]
%\item $\mu_{\varphi}$ ne charge aucun point rigide de~$\EP{1}{k}$ ;
%\item $\varphi^*\mu_{\varphi} = d \, \mu_{\varphi}$. 
%\end{enumerate}

L'objet de ce texte est d'\'etudier la variation de ces mesures d'\'equilibre. 
%lorsque les fractions rationnelles dont elles sont issues varient. 
Pla\c cons-nous tout d'abord dans le cadre complexe et consid\'erons une suite de fractions rationnelles complexes $(f_{n})_{n\in \N}$ qui converge vers une fraction rationnelle complexe~$f$, au sens de la convergence uniforme pour les endomorphismes de~$\P^1(\C)$ associ\'es. D'apr\`es \cite[theorem~B]{ManeHausdorffDimension}, la suite de mesures $(\mu_{f_{n}})_{n\in \N}$ sur~$\P^1(\C)$ converge alors vers la mesure~$\mu_{f}$ pour la topologie faible. L'analogue ultram\'etrique de ce r\'esultat est \'egalement valable.

Nous souhaitons d\'emontrer un r\'esultat analogue, mais avec des corps de base variables. 
%en laissant \'egalement varier le corps de d\'efinition des fractions rationnelles.
% pour des familles de fractions rationnelles dont les corps de d\'efinition varient \'egalement. 
Pour le formuler, nous consid\'erons une famille de fractions rationnelles param\`etr\'ee par un espace de Berkovich~$Y$ sur un anneau de Banach~$\cA$. Remettant \`a plus tard les d\'efinitions pr\'ecises (\cf~section~\ref{sec:Berkovichdefinition}), nous nous contentons de rappeler qu'\`a tout point~$y$ de~$Y$ est associ\'e un corps r\'esiduel compl\'et\'e~$\cH(y)$, qui peut \^etre aussi bien ultram\'etrique~: $\Q_{p}$, $\C(\!(t)\!)$, $\F_{p}(\!(t)\!)$, etc. qu'archim\'edien~: $\R$ ou~$\C$. Signalons que la th\'eorie de Berkovich garde un sens dans ce dernier cadre et que la droite projective $\EP{1}{\cH(y)}$ s'identifie alors \`a un espace familier~: la droite projective complexe usuelle~$\P^1(\C)$, dans le cas de~$\C$, et son quotient par la conjugaison complexe, dans le cas de~$\R$. La th\'eorie des mesures d'\'equilibre s'\'etend sans peine \`a ces espaces.

Le r\'esultat principal de ce texte s'\'enonce ainsi. Il est valable sous certaines conditions techniques sur l'anneau de Banach~$\cA$ consid\'er\'e, \cf~section~\ref{sec:bonsanneaux}. Elles sont satisfaites pour les anneaux classiques de la th\'eorie tels que $\Z$ et les anneaux d'entiers de corps de nombres, les corps hybrides, ou encore les anneaux de valuation discr\`ete complets (et bien s\^ur les corps valu\'es complets).

\begin{theointro}[\textmd{\emph{infra} th\'eor\`eme~\ref{th:muphicontinue}}]\label{th:continuiteintro}
Soit $Y$ un espace de Berkovich sur un bon anneau de Banach~$\cA$. Notons $X := Y \times_{\cA} \EP{1}{\cA}$ la droite projective relative au-dessus de~$Y$. Soit~$\varphi$ un endomorphisme de~$X$ au-dessus de~$Y$, fini et plat de degr\'e $d\ge 2$, et polaris\'e\footnote{Cette derni\`ere condition se traduit en l'existence d'un isomorphisme $\varphi^* \cO_{X}(1) \simeq \cO_{X}(d)$, les fibr\'es $\cO_{X}(1)$ et $\cO_{X}(d)$ \'etant les tir\'es en arri\`ere sur~$X$ des fibr\'es $\cO(1)$ et $\cO_{X}(d)$ sur $\E{1}{\cA}$.}.

%Soit~$\varphi$ un endomorphisme polaris\'e de~$X$ au-dessus de~$Y$ de degr\'e sup\'erieur \`a~2. 

Pour tout point~$y$ de~$Y$, l'endomorphisme~$\varphi$ induit un endomorphisme~$\varphi_{y}$ de la fibre $X_{y} \simeq \EP{1}{\cH(y)}$. Notons $\mu_{\varphi_{y}}$ la mesure d'\'equilibre associ\'ee, identifi\'ee \`a son image sur~$X$. 

Alors la famille de mesures $(\mu_{\varphi_{y}})_{y\in Y}$ sur~$X$ est continue.
\end{theointro}

La continuit\'e de la famille de mesures est \`a prendre au sens de la topologie faible et signifie pr\'ecis\'ement que, pour toute fonction $f\colon X\to \R$ continue \`a support compact, la fonction
\[ y\in Y \mapstoo \int f_{\vert X_{y}} \diff \mu_{\varphi_{y}} \in \R\]
est continue.

\medbreak

Le cas le plus simple du th\'eor\`eme~\ref{th:continuiteintro}, celui o\`u l'endomorphisme~$\varphi$ est l'application d'\'el\'evation au carr\'e $z\mapsto z^2$ est d\'ej\`a int\'eressant. Dans ce cas, la mesure d'\'equilibre est compl\`etement explicite. Sur $\EP{1}{\C} = \P^1(\C)$, il s'agit de la mesure de Haar $\mu_{\mathrm H}$ de masse totale~1 sur le cercle unit\'e. Sur $\EP{1}{\R}$ (quotient de $\P^1(\C)$ par la conjugaison complexe), c'est l'image de la pr\'ec\'edente, et nous la noterons identiquement. Sur un corps valu\'e ultram\'etrique complet~$k$, en revanche, on obtient la mesure de Dirac~$\delta_{\mathrm G}$ support\'ee au point de Gau\ss, un point particulier de~$\EP{1}{k}$ qui joue le r\^ole de point g\'en\'erique du cercle unit\'e. L'analogie entre ces mesures se situe au c\oe ur du travail \cite{ACLMesures} d'Antoine Chambert--Loir, qui l'a remarqu\'ee et largement diffus\'ee. Revenant au probl\`eme qui nous int\'eresse, nous montrons que ces diff\'erentes mesures forment une famille continue. 

%La version du th\'eor\`eme~\ref{th:continuiteintro} dans le cas de $z\mapsto z^2$ constitue une nouvelle incarnation de cette id\'ee.

\begin{theointro}[\textmd{\emph{infra} th\'eor\`eme~\ref{th:chicontinue}}]\label{th:ACLintro}
Soit $Y$ un espace de Berkovich sur un bon anneau de Banach~$\cA$. Pour tout point~$y$ de~$Y$, consid\'erons la mesure $\chi_{y}$ sur $X_{y}\simeq \EP{1}{\cH(y)}$ d\'efinie par
\[ \chi_{y} := 
\begin{cases}
\mu_{\mathrm H} & \textrm{si } \cH(y) \textrm{ est archim\'edien ;}\\
\delta_{\mathrm G} & \textrm{si } \cH(y) \textrm{ est ultram\'etrique,}
\end{cases}\]
identifi\'ee \`a son image sur~$X$.

Alors la famille de mesures $(\chi_{y})_{y\in Y}$ sur $X$ est continue.
\end{theointro}

%Cet \'enonc\'e peut s'interpr\'eter comme une incarnation de l'analogie, propos\'ee par A.~Chambert--Loir, entre mesure de Haar sur le cercle unit\'e du c\^ot\'e complexe et masse de Dirac au point de Gau\ss{} du c\^ot\'e ultram\'etrique. 

Ajoutons que nous d\'emontrons, en m\^eme temps que le th\'eor\`eme~\ref{th:continuiteintro}, un r\'esultat de continuit\'e des potentiels des mesures~$\mu_{\varphi_{y}}$. Rappelons que, sur tout corps valu\'e ultram\'etrique complet, on peut, comme sur~$\C$, d\'efinir une th\'eorie du potentiel sur la droite projective au sens de Berkovich, et, en particulier, pour une certaine classe de fonctions, un op\'erateur laplacien~$\Delta$ qui prend ses valeurs dans l'ensemble des mesures de Radon. 

%\JP{C'est moche en dessous.}

\begin{propintro}[\textmd{\emph{infra} lemme~\ref{lem:lambdaphi} et proposition \ref{prop:convergenceXy}}]\label{prop:potentielintro}%
Avec les hypoth\`eses et notations des th\'eor\`emes~\ref{th:continuiteintro} et~\ref{th:ACLintro}, il existe une fonction continue $\lambda_{\varphi} \colon X \to \R$ telle que, pour tout $y\in Y$, le laplacien de $\lambda_{\varphi,\vert X_{y}}$ est bien d\'efini et satisfait l'\'egalit\'e
\[\Delta (\lambda_{\varphi,\vert X_{y}}) = \chi_{y} - \mu_{\varphi_{y}}.\]
\end{propintro}

La pr\'ecision sur les potentiels est particuli\`erement int\'eressante lorsqu'on cherche \`a \'etudier des accouplements de mesures au sens de C.~Favre et J.~Rivera--Letelier. Dans l'article~\cite{FRLEquidistribution}, ces auteurs d\'efinissent, pour des mesures de probabilit\'es~$\mu_{1}$ et~$\mu_{2}$ suffisamment r\'eguli\`eres, un accouplement $\la \mu_{1}, \mu_{2} \ra$ qui se comporte comme une distance entre ces deux mesures (ou plut\^ot son carr\'e). Il peut \^etre obtenu en int\'egrant un potentiel de $\mu_{1}-\mu_{2}$ par rapport \`a cette m\^eme mesure. 
%Dans le cas qui nous int\'eresse, celui de deux fractions rationnelles~$\varphi$ et~$\psi$, les mesures d'\'equilibre $\mu_{\varphi}$ et $\mu_{\psi}$ satisfont toujours les propri\'et\'es de r\'egularit\'e demand\'ees et on dispose d'un potentiel continu puisque 
%\[ \Delta(\lambda_{\psi} - \lambda_{\varphi}) = (\chi - \mu_{\psi}) - (\chi - \mu_{\varphi}) = \mu_{\varphi} - \mu_{\psi}.\] 

\begin{corointro}\label{cor:energieintro}
Soit $Y$ un espace de Berkovich sur un bon anneau de Banach~$\cA$. Notons $X := Y \times_{\cA} \EP{1}{\cA}$ la droite projective relative au-dessus de~$Y$. Soient~$\varphi$ et~$\psi$ des endomorphismes polaris\'es de~$X$ au-dessus de~$Y$ de degr\'e sup\'erieur \`a~2. 

Pour tout point~$y$ de~$Y$, notons $\mu_{\varphi_{y}}$ (resp. $\mu_{\psi_{y}}$) la mesure d'\'equilibre associ\'ee \`a l'endomorphisme~$\varphi_{y}$ (resp. $\psi_{y}$) induit par~$\varphi$ (resp. $\psi$) sur la fibre $X_{y} \simeq \EP{1}{\cH(y)}$. 

Alors, la fonction $y\in Y \mapsto \la \mu_{\varphi_{y}}, \mu_{\psi_{y}} \ra$ est continue.
\end{corointro}
\begin{proof}
Consid\'erons des fonctions continues~$\lambda_{\varphi}$ et~$\lambda_{\psi}$ comme dans la proposition~\ref{prop:potentielintro}. Pour tout $y\in Y$, on a 
\[ \Delta(\lambda_{\psi,\vert X_{y}} - \lambda_{\varphi,\vert X_{y}}) = (\chi_{y} - \mu_{\psi_{y}}) - (\chi_{y} - \mu_{\varphi_{y}}) = \mu_{\varphi_{y}} - \mu_{\psi_{y}},\]
 d'o\`u
 \[ \la \mu_{\varphi_{y}}, \mu_{\psi_{y}} \ra = \int (\lambda_{\psi,\vert X_{y}} - \lambda_{\varphi,\vert X_{y}}) \diff (\mu_{\varphi_{y}} - \mu_{\psi_{y}}).\]
L'\'enonc\'e d\'ecoule alors des r\'esultats de continuit\'e pr\'ec\'edents (th\'eor\`eme~\ref{th:continuiteintro} et proposition~\ref{prop:potentielintro}).
\end{proof}

%Commen\c cons par quelques rappels sur la notion de mesure d'\'equilibre, tout d'abord dans le cas de la droite projective complexe~$\P^1(\C)$. Notons~$\mu_{\mathrm H}$ la mesure de Haar de masse totale~1 sur le cercle unit\'e de~$\P^1(\C)$. Soit~$\varphi$ un endormorphisme de~$\P^1(\C)$ de degr\'e~$d$. Alors, la suite de mesures $\frac1{d^n} (\varphi^\ast)^n \mu_{\mathrm H}$ converge vers une mesure~$\mu_{\varphi}$ sur~$\P^1(\C)$, dite mesure d'\'equilibre associ\'ee \`a~$\varphi$.
%
%Pla\c cons-nous maintenant sur un corps valu\'e ultram\'etrique complet~$k$, et rempla\c cons la droite projective complexe~$\P^1(\C)$ par la droite projective~$\EP{1}{k}$ au sens de Berkovich. Nonobstant le caract\`ere totalement discontinu de~$k$, l'espace~$\EP{1}{k}$ jouit de bonnes propri\'et\'es topologiques, telles la compacit\'e et la connexit\'e par arcs locales, qui permettent d'y appliquer la th\'eorie des mesures de Radon. Reste \`a trouver l'analogue de la mesure~$\mu_{\mathrm H}$. Dans son texte~\cite{ACLMesures}, Antoine Chambert--Loir propose de lui substituer la mesure de Dirac~$\delta_{\mathrm G}$ support\'ee au point de Gau\ss, un point particulier de~$\EP{1}{k}$ qui joue le r\^ole de point g\'en\'erique du cercle unit\'e. On peut alors d\'evelopper la th\'eorie de fa\c con parall\`ele au cas complexe~: \'etant donn\'e un endomorphisme de~$\EP{1}{k}$ de degr\'e~$d$, la suite de mesures $\frac1{d^n} (\varphi^\ast)^n \delta_{\mathrm G}$ converge vers une mesure~$\mu_{\varphi}$ sur~$\EP{1}{k}$, dite mesure d'\'equilibre associ\'ee \`a~$\varphi$. 

\medbreak

Le th\'eor\`eme~\ref{th:continuiteintro} s'inspire directement du r\'esultat principal de l'article~\cite{FavreEndomorphisms}. Charles Favre y d\'emontre un r\'esultat de continuit\'e de mesures d'\'equilibre pour un espace projectif de dimension quelconque au-dessus d'une base~$Y$ fix\'ee~: le disque hybride (un avatar du disque unit\'e complexe de dimension~1 muni d'un point ultram\'etrique en son centre, r\'ealis\'e comme un espace de Berkovich sur un anneau de Banach dit hybride). Notre th\'eor\`eme autorise une plus grande vari\'et\'e dans la direction horizontale. %Il s'agit, \`a notre connaissance, du premier r\'esultat de ce genre valable sur une base de dimension strictement sup\'erieure \`a~1.

Dans la direction verticale, en revanche, nous nous cantonnons \`a la dimension~1, faute de disposer des outils n\'ecessaires en th\'eorie du potentiel. Les d\'emonstrations de~\cite{FavreEndomorphisms} font en effet intervenir de fa\c con cruciale des r\'esultats analytiques subtils~: th\'eorie de Bedford-Taylor, in\'egalit\'es de Chern-Levine-Nirenberg, etc., dont les analogues ultram\'etriques sont largement conjecturaux \`a l'heure actuelle. Le travail fondateur~\cite{ChambertLoirDucros} d'A.~Chambert--Loir et A.~Ducros propose une d\'efinition de formes diff\'erentielles r\'eelles et de courants dans ce cadre, mais les propri\'et\'es plus fines requises ici ne sont pas encore disponibles.

Indiquons que d'autres r\'esultats de convergence de mesures depuis l'archim\'edien vers l'ultram\'etrique existent dans la litt\'erature~: \cite{BJ} (par des techniques tropicales), \cite{DHL} (par la th\'eorie des mod\`eles), etc. Ils ont tous pour objet des familles \`a un param\`etre et le th\'eor\`eme~\ref{th:continuiteintro} appara\^it donc, \`a notre connaissance, comme le premier de ce genre \`a valoir sur une base de dimension sup\'erieure.  %citer aussi Amini-Nicolussi}

\medbreak

Dans le texte compagnon \cite{DynamiqueII}, nous appliquons les r\'esultats de continuit\'e obtenus ici \`a un espace de modules (de dimension~5) de paires de courbes elliptiques. Nous en tirons une d\'emonstration de l'\'enonc\'e suivant, initialement conjectur\'e par F.~Bogomolov, H.~Fu et Yu.~Tschinkel (\cf~\cite[conjectures~2 et~12]{BFT}).

\begin{theointro}\label{th:BFT}
Il existe $M \in \R_{>0}$ telle que, pour toutes courbes elliptiques~$E_{a}$ et~$E_{b}$ sur~$\C$ et tous rev\^etements doubles $\pi_{a} \colon E_{a}\to \P^1_{\C}$ et $\pi_{b} \colon E_{b}\to \P^1_{\C}$ tels que $\pi_{a}(E_{a}[2]) \ne \pi_{b}(E_{b}[2])$, on ait
\[ \sharp \big( \pi_{a}(E_{a}[\infty]) \cap \pi_{b}(E_{b}[\infty])\big) \le M.  \]
\end{theointro}

\subsection{Strat\'egie de la preuve du th\'eor\`eme~\ref{th:continuiteintro}}

Soit $\cA$ un bon anneau de Banach. 
%un corps valu\'e complet, un corps hybride, un anneau de valuation discr\`ete ou un anneau d'entiers de corps de nombres. 
Soit~$Y$ un espace de Berkovich sur~$\cA$ et notons $X := Y \times_{\cA} \EP{1}{\cA}$ la droite projective relative au-dessus de~$Y$. 

Notre premi\`ere \'etape consiste \`a d\'emontrer le th\'eor\`eme~\ref{th:ACLintro}, incarnation de l'analogie d'A.~Chambert--Loir, qui assure la continuit\'e de la famille constitu\'ee de mesures de Haar sur le cercle, pour la partie archim\'edienne, et de mesures de Dirac au point de Gau\ss, pour la partie ultram\'etrique.

%Soit~$\varphi$ un endomorphisme de~$X$ au-dessus de~$Y$ polaris\'e de degr\'e~$d \ge 2$.

%Nous commen\c cons par d\'emontrer un r\'esultat de continuit\'e pour la famille de mesures la plus simple, celle constitu\'ee uniquement de la mesure~$\mu_{\mathrm H}$ sur~$\P^1(\C)$, de son image sur $\P^1(\C)$ quotient\'e par la conjugaison complexe, et de la mesure de Dirac~$\delta_{\mathrm G}$ sur la droite projective de Berkovich sur un corps ultram\'etrique.
%
%\begin{theointro}[\textmd{\emph{infra} th\'eor\`eme~\ref{th:chicontinue}}]\label{th:ACLintro}
%Pour tout point~$y$ de~$Y$, consid\'erons la mesure $\chi_{y}$ sur $X_{y}\simeq \EP{1}{\cH(y)}$ d\'efinie par
%\[ \chi_{y} := 
%\begin{cases}
%\mu_{\mathrm H} & \textrm{si } \cH(y) \textrm{ est archim\'edien ;}\\
%\delta_{\mathrm G} & \textrm{si } \cH(y) \textrm{ est ultram\'etrique,}
%\end{cases}\]
%identifi\'ee \`a son image sur~$X$.
%
%Alors la famille de mesures $(\chi_{y})_{y\in Y}$ sur $X$ est continue.
%\end{theointro}
%
%Cet \'enonc\'e peut s'interpr\'eter comme une incarnation de l'analogie, propos\'ee par A.~Chambert--Loir, entre mesure de Haar sur le cercle unit\'e du c\^ot\'e complexe et masse de Dirac au point de Gau\ss{} du c\^ot\'e ultram\'etrique. 

Pour la preuve, on consid\'ere une suite (ou une suite g\'en\'eralis\'ee) convergente $(y_{n})_{n\in\N}$ de points de~$Y$ de limite~$y$. Il faut montrer que la suite de mesures~$(\chi_{y_{n}})_{n\in\N}$ converge vers~$\chi_{y}$. Le cas crucial est celui o\`u les points~$y_{n}$ sont archim\'ediens et le point~$y$ ultram\'etrique. Rappelons que les espaces de Berkovich sont localement compacts et que l'espace des mesures de probabilit\'e sur un compact est lui-m\^eme compact. On en d\'eduit qu'il suffit de montrer que toute valeur d'adh\'erence de la suite~$(\chi_{y_{n}})_{n\in\N}$ co\"incide avec $\chi_{y} = \delta_{G}$, autrement dit qu'elle est support\'ee au point de Gau\ss, ce qu'on peut prouver directement.

%Pour la preuve, on consid\'ere une suite convergente $(y_{n})_{n\in\N}$ de points de~$Y$ de limite~$y$. Il faut montrer que la suite de mesures~$(\chi_{y_{n}})_{n\in\N}$ converge vers~$\chi_{y}$. Si les~$y_{n}$ et leur limite~$y$ sont tous situ\'es dans la partie archim\'edienne de~$Y$, ou tous situ\'es dans sa partie ultram\'etrique, la suite de mesures est essentiellement constante, et le r\'esultat se d\'emontre sans difficult\'es. 
%
%Le cas crucial est celui de l'interface entre les deux domaines. La partie archim\'edienne \'etant ouverte, cela ne peut se produire que si la limite~$y$ est ultram\'etrique. Rappelons que les espaces de Berkovich sont localement compacts et que l'espace des mesures de probabilit\'e sur un compact est lui-m\^eme compact. On en d\'eduit qu'il suffit de montrer que toute valeur d'adh\'erence de la suite~$(\chi_{y_{n}})_{n\in\N}$ co\"incide avec $\chi_{y} = \delta_{G}$, autrement dit qu'elle est support\'ee au point de Gau\ss, ce qu'on peut prouver directement. 

Ajoutons que nous d\'emontrons une version raffin\'ee du r\'esultat, o\`u le centre et le rayon des cercles sont autoris\'es \`a varier contin\^ument sur~$Y$.

\medbreak

Passons maintenant au th\'eor\`eme~\ref{th:continuiteintro}. Soit~$\varphi$ un endomorphisme de~$X$ au-dessus de~$Y$ polaris\'e de degr\'e~$d \ge 2$. Rappelons, tout d'abord, que, pour tout $y\in Y$, la mesure d'\'equilibre associ\'ee \`a~$\varphi_{y}$ sur $X_{y} \simeq \EP{1}{\cH(y)}$ peut \^etre obtenue en tirant en arri\`ere la mesure~$\chi_{y}$~:
\[ \mu_{\varphi_{y}} = \lim_{n\to +\infty} \frac1{d^n} (\varphi_{y}^*)^n\chi_{y}.\]
%\JP{explication ? r\'ef\'erence ?}
Cette remarque sugg\`ere une strat\'egie de preuve. \`A partir du th\'eor\`eme~\ref{th:ACLintro}, en tirant en arri\`ere par~$\varphi$ de fa\c con r\'ep\'et\'ee, on montre que, pour tout $n\in \N$, la famille de mesures $(\mu_{n,y} := \frac1{d^n}(\varphi_{y}^*)^n\chi_{y})_{y\in Y}$ est encore continue. Il reste \`a passer \`a la limite. 

%Pour ce faire, nous avons recours \`a des arguments de th\'eorie de potentiel. 
%
%Rappelons que, \'etant donn\'e un corps valu\'e complet~$k$, archim\'edien ou ultram\'etrique, on peut d\'efinir un laplacien~$\Delta$, sur une certaine classe de fonctions sur~$\EP{1}{k}$, qui prend ses valeurs dans l'ensemble des mesures de Radon sur~$\EP{1}{k}$. Cet op\'erateur peut \^etre utilis\'e pour d\'efinir de fa\c con uniforme les mesures~$\chi_{y}$. Fixons une coordonn\'ee $T$ sur $\E{1}{\Z} \subset \EP{1}{\Z}$. Pour tout $y\in Y$, on a alors
%\[ \Delta \big(\max(-\log(\abs{T})_{\vert X_{y}},0)\big) = \chi_{y} - \delta_{y,0} \textrm{ sur } X_{y} \simeq \EP{1}{\cH(y)},\] 
%o\`u~$\delta_{y,0}$ d\'esigne la mesure de Dirac en~0 sur~$X_{y}$. 

Pour ce dernier point, nous avons recours \`a des arguments de th\'eorie de potentiel. En utilisant un raisonnement classique d\^u \`a S.~W.~Zhang (\cf~\cite{ZhangSmallPoints}), on construit une suite de fonctions $(u_{n})_{n\in \N}$ sur~$X$ convergeant uniform\'ement vers une fonction~$u_{\varphi}$ et dont les potentiels sur la fibre~$X_{y}$ sont respectivement la suite $(\mu_{n,y})_{n\in \N}$ et~$\mu_{\varphi_{y}}$. (Ces fonctions sont obtenues \`a partir de la norme d'une section du fibr\'e~$\cO(1)$ convenablement m\'etris\'e.)

Soit $f \colon X \to \R$ continue \`a support compact. Sous des hypoth\`eses convenables, pour tout $y\in Y$ et tout $n\in \N$, on peut \'ecrire
\begin{align*}
\Bigabs{ \int f_{\vert X_{y}} \diff\mu_{\varphi_{y}} -  \int f_{\vert X_{y}} \diff \mu_{n,y}} & = \Bigabs{\int f_{\vert X_{y}} \diff\Delta (u_{\varphi,_{\vert X_{y}}} - u_{n,_{\vert X_{y}}})} \\
& = \Bigabs{\int (u_{\varphi,_{\vert X_{y}}} - u_{n,_{\vert X_{y}}}) \diff \Delta f_{\vert X_{y}}} \\
& \le \norm{u_{\varphi} - u_{n}}_{X_{y}} \int  \diff\abs{\Delta f_{\vert X_{y}}}.
\end{align*}
Si la fonction $y\mapsto \int  \diff\abs{\Delta f_{\vert X_{y}}}$ est born\'ee sur tout compact, alors la suite de fonctions $(y\mapsto \int f_{\vert X_{y}} \diff \mu_{n,y})_{n\in \N}$ converge uniform\'ement vers $y \mapsto \int f_{\vert X_{y}} \diff\mu_{\varphi_{y}}$ sur tout compact. La continuit\'e de cette derni\`ere en d\'ecoule.

La strat\'egie expos\'ee ci-dessus ne s'applique que sous certaines hypoth\`eses sur~$f$ et une partie importante de notre travail consiste \`a identifier un sous-ensemble dense de fonctions continues \`a support compact poss\'edant les propri\'et\'es requises. Nous introduisons \`a cet effet la notion de fonction \emph{macologue}, inspir\'ee de la notion de fonction mod\`ele d\'efinie par Ch.~Favre dans~\cite{FavreEndomorphisms}. Il s'agit de fonctions pouvant s'\'ecrire localement comme diff\'erences de fonctions de la forme 
\[\max(p_{1} + q_{1}\log(\abs{g_{1}}), \dotsc, p_{n} + q_{n}\log(\abs{g_{n}})),\] 
o\`u $p_{1},\dotsc,p_{n}$ sont des nombres rationnels, $q_{1},\dotsc,q_{n}$ des nombres rationnels positifs et $g_{1},\dotsc,g_{n}$ des fonctions analytiques (\cf~d\'efinitions \ref{def:affablebasique} et \ref{def:affable} pour les d\'etails).

\subsection{Organisation du texte}

Le texte est d\'ecoup\'e en 5 sections. Les trois premi\`eres sont essentiellement compos\'ees de rappels. Dans la section~\ref{sec:Berkovich}, nous abordons la th\'eorie, encore exotique, des espaces de Berkovich sur un anneau de Banach, en rappelant d\'efinitions et propri\'et\'es de base. Nous insistons sur la notion de flot, action du mono\"ide~$\intof{0,1}$ correspondant \`a l'\'elevation d'une valeur absolue \`a une puissance. C'est par son interm\'ediaire que l'on peut faire d\'eg\'en\'erer des familles d'espaces sur des espaces au-dessus de corps trivialement valu\'es, ce qui jouera un r\^ole crucial pour des applications ult\'erieures.

Dans la section~\ref{sec:Radon}, nous consid\'erons des mesures de Radon sur des espaces de Berkovich et \'etudions notamment les op\'erations d'image directe et r\'eciproque dans ce cadre. Nous nous appuyons de fa\c con essentielle sur les bonnes propri\'et\'es topologiques des espaces de Berkovich sur certains anneaux de Banach d\'emontr\'ees dans~\cite{CTC}.

Dans la section~\ref{sec:fibresmetrises}, nous adaptons la notion de fibr\'e m\'etris\'e issue de la th\'eorie d'Arakelov au cadre des espaces analytiques sur un anneau de Banach au sens de Berkovich. Nous montrons que l'espace des fibr\'es m\'etris\'es peut \^etre muni d'une structure uniforme pour laquelle il est complet et rappelons, suivant~\cite{ZhangSmallPoints}, comment construire une m\'etrique invariante par un syst\`eme dynamique polaris\'e donn\'e.

La section~\ref{sec:Laplacien} contient des rappels de th\'eorie du potentiel sur la droite projective au-dessus d'un corps valu\'e complet, archim\'edien ou non. L'un des ingr\'edients essentiels, d'int\'er\^et ind\'ependant, est l'existence, pour une fonction sous-harmonique~$u$ sur un disque~$D$, d'une majoration de la forme 
\[\int_{D} \diff \Delta u \le C_{D,D'}\, \norm{u}_{D'},\]
o\`u~$D'$ est un disque strictement plus grand que~$D$ et $C_{D,D'}$ un nombre r\'eel qui ne d\'epend que des rayons des disques et pas du corps de base. C'est cette derni\`ere propri\'et\'e qui permet de borner la masse totale du laplacien d'une fonction macologue, ingr\'edient sur lequel repose la convergence uniforme de la famille de mesures~$(\mu_{n,y})_{n\in \N}$, comme expliqu\'e plus haut.

Dans la section finale~\ref{sec:droiteprojectiverelative}, nous d\'emontrons les r\'esultats annonc\'es, en commen\c cant par le th\'eor\`eme~\ref{th:ACLintro}. Nous introduisons ensuite les fonctions macologues et concluons par la d\'emonstration du th\'eor\`eme~\ref{th:continuiteintro}.

\subsection{Remerciements}

Je remercie Charles Favre et Marco Maculan pour de fructueux \'echanges autour de la th\'eorie du potentiel, ainsi que Dorian Berger pour ses commentaires. Merci \'egalement aux rapporteurs et rapporteuses des diff\'erentes versions. Ce texte est impr\'egn\'e de l'influence d'Antoine Chambert--Loir. Je saisis l'occasion ainsi offerte de lui exprimer ma gratitude sinc\`ere pour son soutien et la g\'en\'erosit\'e avec laquelle il partage ses id\'ees. 

Une partie de ce texte a \'et\'e r\'edig\'ee 
%La derni\`ere version de ce texte a \'et\'e finalis\'ee 
pendant un s\'ejour de l'auteur \`a Francfort. Il remercie l'universit\'e Goethe et ses membres pour les excellentes conditions d'accueil dont il a b\'en\'efici\'e, ainsi que la Deutsche Forschungsgemeinschaft (TRR 326 Geometry and Arithmetic of Uniformized Structures, project number 444845124) pour son soutien financier. 

%\bigbreak

\pagebreak

\begin{center} \textbf{Conventions} \end{center}

Soient $X,Y$ des ensembles. On note $\cF(X,Y)$ l'ensemble des applications de~$X$ dans~$Y$. Pour $f \in \cF(X,\R)$, on pose
\[\norm{f}_{X} := \sup_{x\in X}(\abs{f(x)}).\]

Soient $X,Y$ des espaces topologiques. On note~$\cC(X,Y)$ (resp. $\cC_{c}(X,Y)$) l'ensemble des applications continues (resp. \`a support compact) de~$X$ dans~$Y$.

Soit $T$ une partie de~$X$. On note $\overline T$ son adh\'erence, $\mathring T$ son int\'erieur et $\partial T := \overline T \setminus \mathring T$ son bord.

\section{Espaces de Berkovich sur un anneau de Banach}\label{sec:Berkovich}

Cette section est consacr\'ee aux fondements de la th\'eorie de Berkovich sur les anneaux de Banach. Nous commen\c cons par rappeler les d\'efinitions (section~\ref{sec:Berkovichdefinition} puis donnons des exemples d'anneaux de Banach sur lesquels appliquer la th\'eorie (section~\ref{sec:Berkovichexemples}). Nous terminons en consid\'erant le flot, une op\'eration naturelle sur les espaces de Berkovich utile dans l'\'etude des d\'eg\'en\'erescences (section~\ref{sec:flot}).

\subsection{D\'efinitions}\label{sec:Berkovichdefinition}

Dans cette section, nous rappelons la d\'efinition d'espace analytique sur un anneau de Banach au sens de Berkovich (\cf~\cite[Section~1.5]{rouge}). 

\medbreak

Soit $(\cA,\nm)$ un anneau de Banach. Int\'eressons-nous tout d'abord \`a l'\emph{espace affine analytique de dimension~$n$ sur~$\cA$}, not\'e~$\E{n}{\cA}$. C'est un espace localement annel\'e que nous d\'efinirons en trois temps~: ensemble sous-jacent, topologie, faisceau structural.

$\bullet$ \textit{Ensemble~:} L'ensemble sous-jacent \`a $\E{n}{\cA}$ est l'ensemble des semi-normes multiplicatives sur $\cA[T_{1},\dotsc,T_{n}]$ qui sont born\'ees sur~$\cA$, c'est-\`a-dire l'ensemble des applications 
\[\va \colon \cA[T_{1},\dotsc,T_{n}] \too \R_{\ge 0}\]
qui satisfont les propri\'et\'es suivantes~:
\begin{enumerate}[i)]
\item $|0|=0$ et $|1|=1$;
\item $\forall P,Q \in \cA[T_{1},\dotsc,T_{n}]$, $|P+Q| \le |P| + |Q|$;
\item $\forall P,Q \in \cA[T_{1},\dotsc,T_{n}]$, $|PQ| = |P|\, |Q|$;
\item $\forall a \in \cA$, $|a| \le \|a\|$.
\end{enumerate}

On pose $\cM(\cA) := \E{0}{\cA}$ et on l'appelle \emph{spectre} de~$\cA$. Lorsque $\cA$ n'est pas nul, $\cM(\cA)$ et $\E{n}{\cA}$ ne sont pas vides.

Pour tout $m\in\cn{0}{n}$, le morphisme d'inclusion $\cA[T_{1},\dotsc,T_{m}] \to \cA[T_{1},\dotsc,T_{n}]$ induit une application $\pr_{n,m} \colon \E{n}{\cA} \to \E{m}{\cA}$, dite de \emph{projection sur les $m$ premi\`eres coordonn\'ees}.

Soit $x\in \E{n}{\cA}$. Notons~$\va_{x}$ la semi-norme multiplicative associ\'ee \`a~$x$. L'anneau $\cA[T_{1},\dotsc,T_{n}]/\ker(\va_{x})$ est int\`egre. La semi-norme~$\va_{x}$ induit une valeur absolue sur $\Frac(\cA[T_{1},\dotsc,T_{n}]/\ker(\va_{x}))$ et on peut consid\'erer le compl\'et\'e de ce corps, not\'e~$\cH(x)$. Nous noterons simplement~$\va$ la valeur absolue sur~$\cH(x)$ induite par~$\va_{x}$.

On a un morphisme naturel $\chi_{x}\colon \cA[T_{1},\dotsc,T_{n}] \to \cH(x)$, dit d'\emph{\'evaluation}. Pour tout $P \in \cA[T_{1},\dotsc,T_{n}]$, on pose $P(x) := \chi_{x}(P)$. On a alors $|P(x)| = |P|_{x}$, par d\'efinition.

$\bullet$ \textit{Topologie~:} On munit l'ensemble $\E{n}{\cA}$ de la topologie la plus grossi\`ere telle que, pour tout $P \in \cA[T_{1},\dotsc,T_{n}]$, l'application
\[x \in \E{n}{\cA} \mapstoo |P(x)| \in \R_{\ge0}\] 
soit continue. L'espace topologique ainsi obtenu est s\'epar\'e et localement compact. Le spectre~$\cM(\cA)$ est compact. 

L'application de projection~$\pr_{\cA}$ est continue et surjective et, pour tout $b\in \cM(\cA)$, on a un hom\'eomorphisme canonique
\[\E{n}{\cH(b)} \simto \pr_{\cA}^{-1}(b).\]

$\bullet$ \textit{Faisceau structural~:} Pour tout ouvert~$V$ de~$\E{n}{\cA}$, on note~$S_{V}$ l'ensemble des \'el\'ements de $\cA[T_{1},\dotsc,T_{n}]$ qui ne s'annulent pas sur~$V$ et on pose $\cK(V) := S_{V}^{-1} \cA[T_{1},\dotsc,T_{n}]$. 

Soit~$U$ un ouvert de~$\E{n}{\cA}$. On d\'efinit~$\cO(U)$ comme l'ensemble des applications
\[f \colon U \too \bigsqcup_{x\in U} \cH(x)\]
telles que
\begin{enumerate}[i)]
\item pour tout $x\in U$, $f(x)\in \cH(x)$ ;
\item tout $x\in U$ poss\`ede un voisinage ouvert~$V$ sur lequel~$f$ est limite uniforme d'\'el\'ements de~$\cK(V)$. 
\end{enumerate}

\medbreak

Nous pouvons maintenant \'enoncer la d\'efinition g\'en\'erale d'espace analytique sur~$\cA$. Un \emph{mod\`ele local d'espace $\cA$-analytique} et un espace localement annel\'e de la forme $(V(\cI),\cO_{U}/\cI)$, o\`u~$U$ est un ouvert d'un espace affine analytique~$\E{N}{\cA}$ et $\cI$~est un faisceau d'id\'eaux coh\'erent de~$\cO_{U}$. Un \emph{espace $\cA$-analytique} est un espace localement annel\'e qui est localement isomorphe \`a un mod\`ele local d'espace $\cA$-analytique. Un tel espace~$X$ poss\`ede un morphisme naturel $\pr_{\cA} \colon X \to \cM(\cA)$ et, pour tout $b\in \cM(\cA)$, la fibre $\pr_{\cA}^{-1}(b)$ peut naturellement \^etre munie d'une structure d'espace analytique sur~$\cH(b)$.

On dispose de peu de r\'esultats dans cette g\'en\'eralit\'e. Signalons tout de m\^eme que tout espace $\cA$-analytique est localement compact.

\medbreak

On peut r\'ep\'eter de nombreuses constructions classiques dans ce contexte. On construit, par exemple, un \emph{espace projectif analytique de dimension~$n$ sur~$\cA$}, not\'e $\EP{n}{\cA}$, en recollant $n+1$ copies de~$\E{n}{\cA}$ selon le proc\'ed\'e habituel.

Consid\'erons un mod\`ele local d'espace $\cA$-analytique $Z = (V(\cI),\cO_{U}/\cI)$, avec les m\^emes notations que pr\'ec\'edemment. Consid\'erons l'application $\pr_{N+n,N} \colon \E{N+n}{\cA} \to \E{N}{\cA}$ de projection sur les $n$ derni\`eres coordonn\'ees. Posons $U' := \pr_{N+n,N}^{-1}(U)$ et $\cI' := \pr_{N+n,N}^*\cI$, qui est un faisceau d'id\'eaux coh\'erent de~$\cO_{U'}$. L'espace $(V(\cI'),\cO_{U'}/\cI')$ est appel\'e \emph{espace affine relatif de dimension~$n$ sur~$Z$} et not\'e~$\E{n}{Z}$. 

La construction pr\'ec\'edente se recolle et permet de d\'efinir, pour tout espace $\cA$-analytique~$X$, un espace affine relatif de dimension~$n$ sur~$X$ not\'e~$\E{n}{X}$. On d\'efinit de m\^eme un espace projectif relatif de dimension~$n$ sur~$X$ not\'e~$\EP{n}{X}$.

\medbreak

Le lemme technique suivant sera utile dans la suite du texte.

\begin{lemm}\label{lem:localisationBanach}
Soit~$X$ un espace $\cA$-analytique. Soient $x\in X$ et $f_{1},\dotsc,f_{d} \in \cO(X)$. Alors, il existe un voisinage compact~$V$ de~$x$ dans~$X$, une $\cA$-alg\`ebre de Banach~$\cB$, une application $\varphi \colon \cM(\cB) \to V$ au-dessus de~$\cM(\cA)$ et des \'el\'ements $F_{1},\dotsc,F_{d}$ de~$\cB$ tels que
\begin{enumerate}[i)]
\item $\varphi$ soit un hom\'eomorphisme ;
\item pour tout $y \in V$, le morphisme $\psi_{y} \colon \cH(\varphi(y)) \to \cH(y)$ induit par $\varphi$ soit un isomorphisme isom\'etrique ;
\item pour tout $i\in \cn{1}{d}$ et tout $y \in V$, on ait $\psi_{y}(f_{i}(\varphi(y))) = F_{i}(y)$.
\end{enumerate} 
Pour tout $n\in \N$, notons $\E{n}{V}$ (resp. $\EP{n}{V}$) l'image r\'eciproque de~$V$ par la projection $\E{n}{X} \to X$ (resp. $\EP{n}{X} \to X$). La propri\'et\'e ii) entra\^ine que, pour tout $n\in \N$, l'application~$\varphi$ induit des applications $\E{n}{\cB} \to \E{n}{V}$ et $\EP{n}{\cB} \to \EP{n}{V}$. 

On peut construire~$\varphi$ de fa\c con \`a avoir \'egalement
\begin{enumerate}[i)]
\item[iv)] pour tout $n\in \N$, les applications $\E{n}{\cB} \to \E{n}{V}$ et $\EP{n}{\cB} \to \EP{n}{V}$ induites par~$\varphi$ sont des hom\'eomorphismes.
\end{enumerate} 
\end{lemm}
\begin{proof}
Par d\'efinition, il existe un voisinage ouvert~$U_{x}$ de~$x$ isomorphe \`a un ferm\'e de Zariski~$Z$ d'un ouvert~$U$ d'un espace affine~$\E{N}{\cA}$. Quitte \`a restreindre, on peut supposer qu'il existe $f'_{1},\dotsc,f'_{d} \in \cO(U)$ dont les images dans~$\cO(U_{x})$ sont $f_{1},\dotsc,f_{d}$. On peut \'egalement supposer qu'il existe $g_{1},\dotsc,g_{m} \in \cO(U)$ tels que le faisceau d'id\'eaux de~$\cO_{U}$ d\'efinissant~$Z$ soit engendr\'e par $g_{1},\dotsc,g_{m}$. 

Notons~$x'$ l'image de~$x$ dans~$U$. Soit~$W$ un voisinage compact rationnel de~$x'$ dans~$\E{n}{\cA}$ contenu dans~$U$ (\cf~\cite[d\'efinition~1.2.8]{A1Z}). Notons~$\cB(W)$ le compl\'et\'e de~$\cK(W)$ pour la norme uniforme sur~$W$ (\cf~\cite[d\'efinition~1.2.1]{A1Z}). Quitte \`a restreindre~$W$, on peut supposer que les~$f'_{i}$ et les~$g_{j}$ appartiennent \`a~$\cB(W)$. D'apr\`es \cite[th\'eor\`eme~1.2.11]{A1Z}, le morphisme $\cA[T_{1},\dotsc,T_{N}] \to \cB(W)$ induit un hom\'eomorphisme $\varphi_{0} \colon \cM(\cB(W)) \simto W$. Notons~$I$ l'id\'eal de~$\cB(W)$ engendr\'e par $g_{1},\dotsc,g_{m}$. Le morphisme~$\varphi_{0}$ induit alors un hom\'eomorphisme entre les ferm\'es de Zariski de $\cM(\cB(W))$ et~$W$ d\'efinis par l'id\'eal~$I$. En outre, en notant~$\cB$ le s\'epar\'e compl\'et\'e de~$\cB(W)/I$ muni de sa semi-norme r\'esiduelle, le morphisme canonique $\cB(W) \to \cB$ induit un hom\'eomorphisme entre~$\cM(\cB)$ et le ferm\'e de Zariski de $\cM(\cB(W))$ d\'efini par l'id\'eal~$I$. Les propri\'et\'es i), ii) et~iii) de l'\'enonc\'e sont alors v\'erifi\'ees en d\'efinissant~$V$ comme l'image de~$W$ dans~$U_{x}$ et, pour tout $i\in\cn{1}{d}$, $F_{i}$ comme l'image de $f_{i}$ dans~$\cB$.

La propri\'et\'e~iv) pour~$\E{n}{}$ d\'ecoule de \cite[proposition~1.2.15]{A1Z}. L'hom\'eomorphisme est encore valable pour les ouverts de~$\E{n}{}$ compl\'ementaires des hyperplans de coordonn\'ees. Le cas de~$\EP{n}{}$ s'en d\'eduit par recollement.
\end{proof}

\subsection{Exemples}\label{sec:Berkovichexemples}

Nous regroupons ici quelques exemples d'anneaux de Banach. Nous renvoyons \`a \cite[Section~1.1]{CTC} pour un traitement plus d\'etaill\'e, ainsi que d'autres exemples.

\subsubsection{Corps archim\'ediens}

Le corps~$\C$ muni de la valeur absolue usuelle~$\va_{\infty}$ est un anneau de Banach. 

Soit $n\in \N$. \`A tout $z\in \C^n$, on associe la semi-norme multiplicative
\[\fonction{\va_{z}}{\C[T_{1},\dotsc,T_{n}]}{\R_{\ge 0}}{P}{\abs{P(z)}_{\infty}}.\]
L'application
\[ z \in \C^n \mapstoo \va_{z} \in \E{n}{\C}\]
ainsi d\'efinie est une bijection. C'est m\^eme un isomorphisme d'espaces localement annel\'es, lorsqu'on munit~$\C^n$ de la topologie transcendante et du faisceau des fonctions analytiques.

Soit $\eps\in \intof{0,1}$. Notons $\C_{\eps}$ le corps~$\C$ muni de la valeur absolue~$\va_{\infty}^\eps$. C'est encore un anneau de Banach et, pour tout $n\in \N$, les espaces localement annel\'es~$\E{n}{\C_{\eps}}$ et $\E{n}{\C}$ sont isomorphes (\cf~lemme~\ref{lem:Phiepsiso}).

\medbreak

Consid\'erons maintenant le corps~$\R$ muni de la valeur absolue usuelle~$\va_{\infty}$. C'est un anneau de Banach et la construction pr\'ec\'edente induit un hom\'eomorphisme
\[ \C^n/\Gal(\C/\R) \simto \E{n}{\R},\]
o\`u~$\C^n$ est muni de la topologie transcendante et $\C^n/\Gal(\C/\R)$ de la topologie quotient. 

Soit $\eps\in \intof{0,1}$. Notons $\R_{\eps}$ le corps~$\R$ muni de la valeur absolue~$\va_{\infty}^\eps$. C'est encore un anneau de Banach et, comme dans le cas complexe, les espaces localement annel\'es~$\E{n}{\R_{\eps}}$ et $\E{n}{\R}$ sont isomorphes (\cf~lemme~\ref{lem:Phiepsiso}).

Remarquons que l'injection isom\'etrique $\R_{\eps} \hookrightarrow \C_{\eps}$ induit un morphisme $\pr_{\C,\R} \colon \E{n}{\C_{\eps}} \to \E{n}{\R_{\eps}}$.

\begin{rema}\label{rem:calculepsilonarc}
Tout corps valu\'e archim\'edien complet $(k,\va)$ est isom\'etriquement isomorphe \`a~$\C_{\eps}$ ou~$\R_{\eps}$ pour un certain $\eps \in \intof{0,1}$. Le nombre r\'eel~$\eps$ peut \^etre obtenu par la formule
\[\eps = \frac{\log(\abs{2})}{\log(2)}.\]
\end{rema}

\subsubsection{Corps ultram\'etriques}\label{sec:corpsum}

Tout corps valu\'e ultram\'etrique complet fournit un exemple d'anneau de Banach.

Soit $(k,\va)$ un corps valu\'e ultram\'etrique complet. Soit $n\in\N$. \`A tout $z\in k^n$, on associe la semi-norme multiplicative
\[\fonction{\va_{z}}{k[T_{1},\dotsc,T_{n}]}{\R_{\ge 0}}{P}{\abs{P(z)}}.\]
L'application
\[ z \in k^n \mapstoo \va_{z} \in \E{n}{k}\]
ainsi d\'efinie est injective, mais elle est loin d'\^etre surjective, d\`es que $n\ge 1$.

Donnons un exemple. Pour $z\in k$ et $s\in \R_{> 0}$, l'application
\[\fonction{\va_{z,s}}{k[T]}{\R_{\ge 0}}{\disp\sum_{i\ge 0} a_{i} (T-z)^i}{\disp\max_{i\ge 0}(\abs{a_{i}}s^i)}\]
d\'efinit une valeur absolue sur~$k[T]$. Le point de~$\E{1}{k}$ correspondant, que nous noterons~$\eta_{z,s}$, n'appartient pas \`a l'image de~$k$ dans~$\E{1}{k}$.

\subsubsection{Autres anneaux de Banach}\label{sec:autresanneaux}

Nous pr\'esentons ici trois classes d'anneaux de Banach pour lesquels les r\'esultats du texte s'appliquent. 

\medbreak

\noindent $\bullet$ Corps hybrides

Soit $(k,\va)$ un corps muni d'une valeur absolue non triviale, archim\'edienne ou non. Notons $\hat k$ le compl\'et\'e de~$k$. Consid\'erons la norme sur~$k$ d\'efinie par 
\[ \nm_{\hyb} := \max(\va,\va_{0}),\]
o\`u $\va_{0}$ d\'esigne la valeur absolue triviale (d\'efinie par $\abs{0}_{0}=0$ et $\abs{\alpha}_{0}=1$ pour $\alpha\in k^\ast$). Notons~$k_{\hyb}$ le corps~$k$ muni de la norme~$\nm_{\hyb}$. C'est un anneau de Banach dont on peut d\'ecrire explicitement le spectre~:
\[ \cM(k_{\hyb}) = \{\va_{0}\} \cup \{\va^\eps : 0< \eps \le 1\}.\]  
Il est hom\'eomorphe au segment~$\intff{0,1}$.

Soit $X$ un espace analytique sur~$k_{\hyb}$ et consid\'erons le morphisme structural $\pr \colon X \to \cM(k_{\hyb})$. Alors, pour tout $\eps \in \intof{0,1}$, $\pr^{-1}(\va^\eps)$ est un espace analytique sur~$(\hat k,\va^\eps)$, donc, \`a peu de choses pr\`es, un espace analytique sur $(\hat k,\va)$, et $\pr^{-1}(\va_{0})$ est un espace analytique sur~$(k,\va_{0})$, de nature fondamentalement diff\'erente, auquel on peut penser comme \`a la limite des pr\'ecedents. Lorsque $(k,\va)$ est le corps~$\C$ muni de la valeur absolue usuelle, on obtient ainsi une famille d'espaces analytique complexes d\'eg\'en\'erant vers un espace ultram\'etrique. C'est l'exemple consid\'er\'e dans~\cite{FavreEndomorphisms}.

\medbreak

\noindent $\bullet$ Anneaux de valuation discr\`ete

Soit $R$ un anneau de valuation discr\`ete complet et~$v$ la valuation associ\'ee. Notons~$\m$ l'id\'eal maximal de~$R$. Soit $r \in \intoo{0,1}$ et posons $\va := r^{v(\wc)}$. Alors $(R,\va)$ est un anneau de Banach.

Notons~$\va^{+\infty}$ la semi-norme sur~$R$ d\'efinie comme la composition de l'application quotient $R \to R/\m$ et de la valeur absolue triviale sur~$R/\m$. On peut alors d\'ecrire explicitement le spectre de~$(R,\va)$~:
\[ \cM(R) = \{\va^\eps : \eps \ge 1\}\cup \{\va^{+\infty}\}.\]  
Il est hom\'eomorphe au segment~$\intff{1,+\infty}$.

Cette situation est assez semblable \`a celle des corps hybrides. Soit $X$ un espace analytique sur~$R$ et consid\'erons le morphisme structural $\pr \colon X \to \cM(R)$. Alors, pour tout $\eps \in \intfo{1,+\infty}$, $\pr^{-1}(\va^\eps)$ est, \`a normalisation pr\`es, un espace analytique sur~$(K,\va)$, o\`u $K$ est le corps des fractions de~$R$, et $\pr^{-1}(\va^{+\infty})$ est un espace analytique sur~$(R/\m,\va_{0})$, sorte de version d\'eg\'en\'er\'ee des pr\'ec\'edents.

\medbreak

\noindent $\bullet$ Anneaux d'entiers de corps de nombres

L'anneau des entiers relatifs~$\Z$ muni de la valeur absolue usuelle~$\va_{\infty}$ est un anneau de Banach. La description de son spectre analytique d\'ecoule du th\'eor\`eme d'Ostrowski. 
%et est tr\`es explicite, \cf~\cite[section~2]{DynamiqueII}. 
%Consid\'erons l'anneau des entiers relatifs~$\Z$ muni de la valeur absolue usuelle~$\va_{\infty}$. Il s'agit d'un anneau de Banach et l'on peut donc consid\'erer son spectre analytique~$\cM(\Z)$ au sens de V.~Berkovich. Rappelons que ce dernier est d\'efini, ensemblistement, comme l'ensemble des semi-normes multiplicatives sur~$\Z$. Le th\'eor\`eme d'Ostrowski permet de le d\'ecrire explicitement. 
Il contient~: 
\begin{itemize}
\item la valeur absolue triviale~$\va_{0}$ sur~$\Z$ ;
\item les valeurs absolues archim\'ediennes~$\va_{\infty}^\eps$ avec $\eps \in \intof{0,1}$ ;
\item pour tout nombre premier~$p$, les valeurs absolues $p$-adiques~$\va_{\infty}^\eps$ avec $\eps \in \intoo{0,+\infty}$ ;
\item pour tout nombre premier~$p$, la semi-norme~$\va_{p}^{+\infty}$ d\'efinie comme la composition de l'application quotient $\Z \to \Z/p\Z$ et de la valeur absolue triviale sur~$\Z/p\Z$. 
\end{itemize}
%Nous noterons $a_{0}$ le point associ\'e \`a~$\va_{0}$, $a_{\infty}^\eps$ le point associ\'e \`a~$\va_{\infty}^\eps$, pour $\eps \in \intof{0,1}$, et $a_{p}^\eps$ le point associ\'e \`a~$\va_{p}^\eps$, pour $p\in \cP$ et $\eps \in \intof{0,+\infty}$. 
%La figure~\ref{fig:MZ} contient une repr\'esentation d'un plongement (non canonique) de~$\cM(\Z)$ dans~$\R^2$ respectant la topologie.

%\begin{figure}[!h]
%\centering
%\begin{tikzpicture}
%\foreach \x [count=\xi] in {-2,-1,...,17}
%\draw (0,0) -- ({10*cos(\x*pi/10 r)/\xi},{10*sin(\x*pi/10 r)/\xi}) ;
%\foreach \x [count=\xi] in {-2,-1,...,17}
%\fill ({10*cos(\x*pi/10 r)/\xi},{10*sin(\x*pi/10 r)/\xi}) circle ({0.07/(sqrt(\xi)}) ;
%
%\draw ({10.5*cos(-pi/5 r)},{10.5*sin(-pi/5 r)}) node{$\va_{\infty}$} ;
%\fill ({5.5*cos(-pi/5 r)},{5.5*sin(-pi/5 r)}) circle (0.07) ;
%\draw ({5.5*cos(-pi/5 r)},{5.5*sin(-pi/5 r)-.1}) node[below]{$\va_{\infty}^\eps$} ;
%
%\draw ({11*cos(-pi/10 r)/2+.1},{11*sin(-pi/10 r)/2}) node{$\va_2^{+\infty}$} ;
%\fill ({2.75*cos(-pi/10 r)},{2.75*sin(-pi/10 r)}) circle ({0.07/(sqrt(2)}) ;
%\draw ({2.75*cos(-pi/10 r)},{2.75*sin(-pi/10 r)-.05}) node[below]{$\va_2^\eps$} ;
%
%\draw ({12*cos(pi/5 r)/5+.1},{12*sin(pi/5 r)/5+.1}) node{$\va_p^{+\infty}$} ;
%\end{tikzpicture}
%\caption{Le spectre de Berkovich $\cM(\Z)$.}\label{fig:MZ}
%\end{figure}

Cette th\'eorie se g\'en\'eralise aux anneaux d'entiers de corps de nombres. Nous renvoyons \`a \cite[section~3.1]{A1Z} pour plus de d\'etails.

\subsubsection{Bons anneaux de Banach}\label{sec:bonsanneaux}

Nous introduisons ici une notion de bon anneau de Banach, pour lesquels les r\'esultats de ce texte sont valables. Elle est abstraite, utilise de fa\c con essentielle la terminologie de~\cite{CTC}, et probablement sp\'ecifique \`a ce papier. Nous sugg\'erons au lecteur de ne pas trop y accorder trop d'attention, mais plut\^ot de garder en t\^ete les exemples de la section pr\'ec\'edente.

\begin{defi}
Un anneau de Banach $(\cA,\nm)$ est dit \emph{bon} si 
\begin{enumerate}[i)]
\item $(\cA,\nm)$ est un anneau de base g\'eom\'etrique (\cf~\cite[d\'efinition~3.3.8]{CTC}) ;
\item $\cM(\cA,\nm)$ est peu mixte (\cf~\cite[d\'efinition~7.2.5]{CTC}) ;
\item $\cM(\cA,\nm)$ et sa partie ultram\'etrique sont localement connexes par arcs. 
\end{enumerate}
\end{defi}

Les exemples pr\'ec\'edents (corps valu\'es, corps hybrides, anneaux de valuation discr\`ete, anneaux d'entiers de corps de nombres) sont tous de bons anneaux de Banach (\cf~\cite[exemples~3.3.9 et~7.2.7]{CTC}).

Les espaces analytiques sur de tels anneaux de Banach ont re\c cu un traitement d\'etaill\'e dans la litt\'erature, \cf~\cite{A1Z,EtudeLocale,CTC}. On sait qu'ils jouissent de bonnes propri\'et\'es, dans la m\^eme veine que les espaces analytiques classiques~: hens\'elianit\'e, noeth\'erianit\'e et excellence des anneaux locaux (\cf~\cite[corollaire~2.2.7]{A1Z} et \cite[corollaire~9.19 et~10.3]{EtudeLocale}), coh\'erence du faisceau structural (\cf~\cite[corollaire~11.10]{EtudeLocale}), connexit\'e par arcs locale (\cf~\cite[th\'eor\`eme~7.2.17]{CTC}), etc. On peut \'egalement d\'efinir la cat\'egorie des espaces analytiques sur un tel anneau (\cf~\cite[section~2.1]{CTC}) et l'on dispose d'un foncteur d'analytification depuis la cat\'egorie des sch\'emas localement de pr\'esentation finie, d'un foncteur de changement de base et de produits fibr\'es (\cf~\cite[chapitre~4]{CTC}). Nous utiliserons librement ces propri\'et\'es dans la suite du texte.

\subsection{Flot}\label{sec:flot} 

Nous nous int\'eressons ici \`a la fa\c con dont sont modifi\'es les espaces de Berkovich lorsqu'on remplace les semi-normes multiplicatives qui le constituent par des puissances.

\subsubsection{Cas d'un corps valu\'e}\label{sec:flotcorpsvalue}

Soient $A$ un anneau et~$\va$ une valeur absolue sur~$A$. 

Rappelons que le fait que~$\va$ soit archim\'edienne ou ultram\'etrique ne d\'epend que de sa restriction \`a l'image de~$\Z$ dans~$A$. Plus pr\'ecis\'ement, $\va$ est ultram\'etrique si, et seulement si, on a $\abs{2}\le 1$.

\begin{nota}
Si~$\va$ est ultram\'etrique, on pose $I(\va) := \R_{>0}$.

Si~$\va$ est archim\'edienne, on pose $I(\va) := \mathopen{\big]} 0,\frac{\log(2)}{\log(\abs{2})} \mathclose{\big]}$ (\cf~remarque~\ref{rem:calculepsilonarc}).
\end{nota}

L'int\'er\^et de la notation appara\^it dans le r\'esultat classique suivant. 

\begin{lemm}\label{lem:vaeps}
Pour tout $\eps \in I(\va)$, $\va^\eps$ d\'efinit une valeur absolue sur~$A$.
\qed
\end{lemm}

Le r\'esultat vaut encore pour les semi-normes multiplicatives. (On se ram\`ene au cas d'une valeur absolue en quotientant par le noyau de la semi-norme.)

\medbreak

Soit $(k,\va)$ un corps valu\'e complet. 

\begin{nota}
Pour tout $\eps\in I(\va)$, on note~$k_{\eps}$ le corps~$k$ muni de la valeur absolue~$\va^\eps$. C'est encore un corps valu\'e complet.
\end{nota}

Soit $x\in \E{n}{k}$. Par d\'efinition, il est associ\'e \`a une semi-norme multiplicative $\va_{x} \colon k[T_{1},\dotsc,T_{n}] \to \R_{\ge0}$ qui induit~$\va$ sur~$k$. Pour tout $\eps \in I(\va)$, il d\'ecoule du lemme~\ref{lem:vaeps} que l'application $\va_{x}^\eps \colon k[T_{1},\dotsc,T_{n}] \to \R_{\ge0}$ d\'efinit une semi-norme multiplicative qui induit~$\va^\eps$ sur~$k$. 

\begin{nota}
Pour tout $x\in \E{n}{\cA}$ et tout $\eps \in I(k)$, on note~$x^\eps$ le point de~$\E{n}{k_\eps}$ associ\'e \`a~$\va_{x}^\eps$.
\end{nota}

Remarquons que les corps abstraits $\cH(x)$ et $\cH(x^\eps)$ sont isomorphes, puisqu'ils sont obtenus en compl\'etant le m\^eme corps pour deux valeurs absolues \'equivalentes.

\begin{rema}
L'application $x\mapsto x^\eps$ agit trivialement sur les points $k$-rationnels~: pour tout $a \in k^n$, le point de~$\E{n}{k}$ associ\'e \`a~$a$ est envoy\'e  sur le point de~$\E{n}{k_{\eps}}$ associ\'e \`a~$a$. On le v\'erifie imm\'ediatement sur les d\'efinitions.
\end{rema}

\begin{lemm}\label{lem:Phiepsiso}
L'application 
\[\renewcommand{\arraystretch}{1.2} \fonction{\Phi_{\eps}}{\E{n}{k}}{\E{n}{k_{\eps}}}{x}{x^\eps}\]
r\'ealise un isomorphisme d'espaces localement annel\'es.

Pour tout ouvert $U$ de~$\E{n}{k}$, toute $f\in \cO(U)$ et tout $x\in U$, on a
\[ \abs{(\Phi_{\eps})_{\ast }f(x^\eps)} = \abs{f(x)}^\eps.\]
\end{lemm}
\begin{proof}
L'application $\Phi_{\eps^{-1}} \colon \E{n}{k_{\eps}} \to \E{n}{k}$ satisfait $\Phi_{\eps^{-1}} \circ \Phi_{\eps} = \Phi_{\eps}\circ\Phi_{\eps^{-1}}=\id$, donc $\Phi_{\eps}$ est une bijection.

Il d\'ecoule directement de la d\'efinition de la topologie que les applications~$\Phi_{\eps}$ et~$\Phi_{\eps^{-1}}$ sont continues, donc $\Phi_{\eps}$ est un hom\'eomorphisme.

Rappelons que le faisceau structural est construit \`a partir des corps~$\cH(x)$. L'isomorphisme d'espaces localement annel\'es suit donc des identifications $\cH(x^\eps) = \cH(x)$. 

La propri\'et\'e finale d\'ecoule des d\'efinitions.
\end{proof}

Rappelons que les espaces $k$-analytiques que nous consid\'erons ici sont construits en recollant des mod\`eles locaux, qui sont des ferm\'es analytiques d'ouverts d'espaces affines. Il d\'ecoule du lemme~\ref{lem:Phiepsiso} que la restriction de~$\Phi_{\eps}$ \`a un tel espace ne d\'epend pas de sa pr\'esentation comme mod\`ele local. Par cons\'equent, on peut recoller les images des mod\`eles locaux et les applications~$\Phi_\eps$.

\begin{nota}
Soit $X$ un espace $k$-analytique. On note~$X_{\eps}$ l'espace $k_{\eps}$-analytique et 
\[\fonction{\Phi_{\eps}}{X}{X_{\eps}}{x}{x^\eps}\]
l'isomorphisme d'espaces localement annel\'es obtenus par la construction ci-dessus.
\end{nota}

\subsubsection{Cas d'un anneau de Banach}\label{sec:flotanneau}

Soit $(\cA,\nm)$ un anneau de Banach. Sous certaines conditions, l'application $x\mapsto x^\eps$ peut \^etre d\'efinie comme application d'un espace $\cA$-analytique dans lui-m\^eme.

\begin{defi}\label{def:Aflottant}
On dit que $(\cA,\nm)$ est \emph{flottant} si, pour tout $a\in \cA$, on a $\norm{a} \ge 1$.
\end{defi}

Pour toute semi-norme multiplicative~$\va$ sur $\cA[T_{1},\dotsc,T_{n}]$ et tout $\eps \in \intof{0,1}$, $\va^\eps$ est encore une semi-norme multiplicative sur $\cA[T_{1},\dotsc,T_{n}]$. Lorsque $(\cA,\nm)$ est flottant, le caract\`ere born\'e par rapport \`a~$\nm$ est pr\'eserv\'e.

\begin{nota}\label{nota:xeps}
Supposons que $(\cA,\nm)$ est flottant. Pour tout $x\in \E{n}{\cA}$ et tout $\eps \in \intof{0,1}$, on note~$x^\eps$ le point de~$\E{n}{\cA}$ associ\'e \`a la semi-norme multiplicative~$\va_{x}^\eps$.

On pose 
\[ \fonction{\Phi_{\eps}}{\E{n}{\cA}}{\E{n}{\cA}}{x}{x^\eps}.\]

On \'etend les d\'efinitions de~$x^\eps$ et~$\Phi_{\eps}$ \`a~$\EP{n}{\cA}$ en utilisant des cartes.
\end{nota}

\begin{rema}
Supposons que $(\cA,\nm)$ est flottant. Notons $\pr \colon \E{n}{\cA} \to \cM(\cA)$ la projection. Soit $b\in \cM(\cA)$. Alors, l'application $\Phi_{\eps} \colon \E{n}{\cA} \to \E{n}{\cA}$ de la notation~\ref{nota:xeps} induit une application 
\[\pr^{-1}(b) \simeq \E{n}{\cH(b)} \too \pr^{-1}(b^\eps) \simeq \E{n}{\cH(b^\eps)}\simeq \E{n}{\cH(b)_{\eps}}\] 
qui n'est autre que celle de la section~\ref{sec:flotcorpsvalue}. 
\end{rema}

\begin{lemm}
Supposons que $(\cA,\nm)$ est flottant. Alors, pour tout $\eps \in \intof{0,1}$, l'application $\Phi_{\eps} \colon \EP{n}{\cA} \to \EP{n}{\cA}$ r\'ealise un hom\'eomorphisme sur son image.
\end{lemm}
\begin{proof}
Soit $\eps \in \intof{0,1}$. On se ram\`ene imm\'ediatement au cas de $\Phi_{\eps} \colon \E{n}{\cA} \to \E{n}{\cA}$. Il d\'ecoule des d\'efinitions que cette application est continue (\cf~\cite[proposition~1.3.4]{A1Z}).

Notons~$\cA_{\eps}$ l'anneau~$\cA$ muni de la norme~$\nm^\eps$. C'est encore un anneau de Banach. On a une inclusion naturelle $\E{n}{\cA_{\eps}} \subset \E{n}{\cA}$ par laquelle $\E{n}{\cA_{\eps}}$ s'identifie \`a l'image de~$\Phi_{\eps}$. 

Sur l'espace $\E{n}{\cA_{\eps}}$, on peut d\'efinir une application~$\Phi_{\eps^{-1}}$, qui est continue comme pr\'ec\'edemment. Ceci conclut la preuve.
\end{proof}

\begin{defi}\label{def:flottant}
Supposons que $(\cA,\nm)$ est flottant. Un espace $\cA$-analytique~$X$ est dit \emph{flottant} s'il existe une immersion de~$X$ dans~$\EP{n}{\cA}$ telle que, pour tout $\eps \in \intof{0,1}$, on ait $\Phi_{\eps}(X) \subset X$.
\end{defi}

Si $X$ est flottant, pour tout ouvert $U$ de~$X$, toute $f\in \cO(U)$, tout $x\in U$ et tout $\eps\in \intof{0,1}$, on a $\abs{f(x^\eps)} = \abs{f(x)}^\eps$. On en d\'eduit que la propri\'et\'e d'\^etre fottant et l'application $\Phi_{\eps} \colon x\in X \mapsto x^\eps \in X$ ne d\'ependent pas du choix de l'immersion.

\begin{defi}
Supposons que $(\cA,\nm)$ est flottant. Soit $X$ un espace $\cA$-analytique flottant.

Une fonction $f \colon X \to \R_{\ge0} \cup \{+\infty\}$ (resp. $f \colon X \to \R\cup\{\pm\infty\}$) est dite \emph{flottante} (resp. \emph{$\log$-flottante}) si, pour tout $x\in X$ et tout $\eps\in \intof{0,1}$, on a $f(x^\eps) = f(x)^\eps$ (resp. $f(x^\eps) = \eps f(x)$).\footnote{On utilise ici les conventions $(+\infty)^\eps = \eps (+\infty) = +\infty $ et $\eps (-\infty) = -\infty$, pour $\eps>0$.}

Soit $Y$ un espace $\cA$-analytique flottant. Un morphisme $\varphi\colon X\to Y$ est dit \emph{flottant} si, pour tout $\eps\in \intof{0,1}$, on a $\Phi_{\eps} \circ \varphi =\varphi \circ \Phi_{\eps}$.
\end{defi}

\section{Mesures de Radon}\label{sec:Radon}

Cette section est consacr\'ee aux mesures de Radon. Nous y consid\'erons les op\'erations d'image directe (section~\ref{sec:Radonimage}) et image r\'eciproque (section~\ref{sec:Radonimagereciproque}). Dans la section~\ref{sec:Radonfamille}, nous introduisons la notion de famille continue de mesures, objet d'\'etude de ce texte.

\subsection{Images de mesures de Radon}\label{sec:Radonimage}

Commen\c cons par rappeler quelques notions topologiques.

\begin{defi}\label{def:applicationpropre}
Soit $\varphi \colon X \to X'$ une application continue entre espaces topologiques. 

On dit que $\varphi$ est \emph{propre} si, pour tout espace topologique~$X''$, l'application $\varphi\times \id_{X''} \colon X\times X'' \to X'\times X''$ est ferm\'ee.

On dit que $\varphi$ est \emph{finie} si elle est s\'epar\'ee, ferm\'ee et \`a fibres finies.
\end{defi}

Le r\'esultat suivant est une cons\'equence directe des d\'efinitions.

\begin{lemm}\label{lem:finivoisinagefibre}
Soit $\varphi \colon X \to X'$ une application finie entre espaces topologiques. Soit~$x'\in X'$. Pour tout voisinage~$U$ de~$\varphi^{-1}(x')$ dans~$X$, il existe un voisinage~$U'$ de~$x'$ dans~$X'$ tel que $\varphi^{-1}(U') \subset U$.
\qed
\end{lemm}

Ajoutons d'autres r\'esultats classiques.

\begin{prop}[\protect{\cite[I, \S 10, \no 2, prop.~6]{BourbakiTG14}}]
Soit $\varphi \colon X \to X'$ une application propre entre espaces topologiques. Alors, pour toute partie quasi-compacte~$K'$ de~$X'$, $\varphi^{-1}(K')$ est quasi-compacte.
\qed 
\end{prop}

\begin{theo}[\protect{\cite[I, \S 10, \no 2, th.~1]{BourbakiTG14}}]\label{thm:finipropre}
Soit $\varphi \colon X \to X'$ une application finie entre espaces topologiques. Alors $\varphi$ est propre.
\qed
\end{theo}

Nous pouvons maintenant d\'ecrire la proc\'edure pour pousser en avant des mesures de Radon.

\begin{defi}\label{def:imagemesure}
Soit $\varphi \colon X \to X'$ une application propre entre espaces topologiques localement compacts. Supposons que~$X$ est topologiquement s\'epar\'e. Soit~$\mu$ une mesure de Radon sur~$X$. L'application
\[\fonction{\varphi_{*}\mu}{\cC(X',\R)}{\R}{f}{\mu(f\circ \varphi)}\]
est une mesure de Radon sur~$X'$ appel\'ee \emph{image} de~$\mu$ par~$\varphi$.
\end{defi}

\begin{defi}
Soit $\cA$ un anneau de Banach flottant. Soit $X$ un espace $\cA$-analytique flottant et topologiquement s\'epar\'e. Une mesure de Radon~$\mu$ sur~$X$ est dite \emph{flottante} si, pour tout $\eps\in \intof{0,1}$, on a $(\Phi_{\eps})_{\ast} \mu=\mu$.
\end{defi}

\subsection{Images r\'eciproques de mesures de Radon}\label{sec:Radonimagereciproque}

Pour tirer en arri\`ere des mesures, nous avons besoin d'hypoth\`eses suppl\'ementaires sur l'application. Pour les \'enoncer, nous nous placerons dans le cadre des espaces analytiques. Soit~$\cA$ un bon anneau de Banach. 

\begin{defi}\label{def:morphismepropre}
Un morphisme entre espaces $\cA$-analytiques est dit \emph{propre} (resp. \emph{fini}) si l'application induite entre les espaces topologiques sous-jacents est propre (resp. finie) au sens de la d\'efinition~\ref{def:applicationpropre}.
\end{defi}

Rappelons tout d'abord le r\'esultat fondamental suivant. 

\begin{theo}[\protect{\cite[th\'eor\`eme~5.2.1]{CTC}}]\label{th:finicoherent}
Soit $\varphi \colon X \to X'$ un morphisme fini d'espaces $\cA$-analytiques. Alors, pour tout faisceau coh\'erent~$\cF$ sur~$X$, $\varphi_{*}\cF$ est un faisceau coh\'erent sur~$X'$.

En particulier, pour tout $x\in X$, $\cO_{x}$ est un module de type fini sur~$\cO_{\varphi(x)}$.
\qed
\end{theo}

Sous une hypoth\`ese suppl\'ementaire de platitude, nous pouvons maintenant d\'efinir des degr\'es locaux, en utilisant le fait qu'un module plat et de type fini sur un anneau local est libre.

\begin{defi}\label{def:degre}
Soit $\varphi \colon X \to X'$ un morphisme fini et plat d'espaces $\cA$-analytiques. Alors, pour tout $x\in X'$, $\cO_{x}$ est un module libre de type fini sur~$\cO_{\varphi(x)}$. Son rang est appel\'e \emph{degr\'e} de~$\varphi$ en~$x$ et not\'e~$\deg_x(\varphi)$.
\end{defi}

Pla\c cons-nous dans le cadre de la d\'efinition~\ref{def:degre}. Pour tout~$x'\in X'$, on a
\[\sum_{x\in \varphi^{-1}(x')} \deg_{x}(\varphi) = \sum_{x\in \varphi^{-1}(x')} \rang_{\cO_{x'}}(\cO_{x}) = \rang_{\cO_{x'}}((\varphi_{*}\cO)_{x'}).\]
Les hypoth\`eses assurent que $\varphi_{*}\cO$ est un $\cO$-module localement libre. En particulier, si $X'$~est connexe, il est de rang constant, et la quantit\'e pr\'ec\'edente ne d\'epend pas de~$x'$.

\begin{defi}
Soit $\varphi \colon X \to X'$ un morphisme fini et plat d'espaces $\cA$-analytiques. Supposons que $X'$~est connexe. On appelle \emph{degr\'e} de~$\varphi$ la quantit\'e
\[\deg(\varphi) := \sum_{x\in \varphi^{-1}(x')} \deg_{x}(\varphi),\]
pour $x'\in X'$.
\end{defi}

Afin de tirer en arri\`ere les mesures, nous aurons besoin de pousser en avant les fonctions.

\begin{defi}
Soit $\varphi \colon X \to X'$ un morphisme fini et plat d'espaces $\cA$-analytiques. Pour toute $f\in \cC(X,\R)$ et tout $x'\in X'$, on pose
\[ (\varphi_{*} f)(x') := \sum_{x\in \varphi^{-1}(x')} \deg_{x}(\varphi)\, f(x).\]
\end{defi}

\begin{prop}\label{prop:varphi_{*}f}
Soit $\varphi \colon X \to X'$ un morphisme fini et plat d'espaces $\cA$-analytiques. Alors, pour toute $f\in \cC(X,\R)$ (resp. $f\in \cC_{c}(X,\R)$), on a $\varphi_{*} f \in \cC(X',\R)$ (resp. $\varphi_{*} f \in \cC_{c}(X',\R)$).

En outre, si~$X'$ est connexe, pour toute partie~$E'$ de~$X'$, on a
\[ \norm{\varphi_{*}f}_{E'} \le \deg(\varphi)\,  \norm{f}_{\varphi^{-1}(E')}.\] 
\end{prop}
\begin{proof}
Nous suivons fid\`element la preuve de~\cite[proposition~2.4]{FRLergodique}. Soit $f\in \cC(X,\R)$. Soient $x'\in X'$ et $\eps\in \R_{>0}$. Pour tout $x\in \varphi^{-1}(x')$, il existe un voisinage~$U_{x}$ de~$x$ dans~$X$ tel que, pour tout $y\in U_{x}$, on ait $\abs{f(y)-f(x)}\le \eps$. On peut supposer que les~$U_{x}$, pour $x\in \varphi^{-1}(x')$, sont disjoints.

D'apr\`es le lemme~\ref{lem:finivoisinagefibre}, il existe un voisinage~$U'$ de~$x'$ dans~$X'$ tel que $\varphi^{-1}(U')\subset \bigcup_{x\in \varphi^{-1}(x')}U_{x}$. Quitte \`a restreindre~$U'$, on peut supposer qu'il est connexe. Pour $x\in \varphi^{-1}(x')$, posons $V_{x} := U'_{x}\cap \varphi^{-1}(U')$. Le morphisme $\psi_{x} \colon V_{x} \to U'$ induit par~$\varphi$ est fini et de degr\'e~$\deg_{x}(\varphi)$. Pour tout $y'\in U'$, on a donc
\begin{align*}
\Bigabs{\Big(\sum_{y\in V_{x}\cap \varphi^{-1}(y')} \deg_{y}(\varphi) f(y)\Big) - \deg_{x}(\varphi) f(x)} &= \Bigabs{\sum_{y\in V_{x}\cap \varphi^{-1}(y')} \deg_{y}(\varphi) (f(y) - f(x))}\\
& \le \deg_{x}(\varphi)\, \eps.
\end{align*}
On en d\'eduit que, pour tout $y'\in U'$, on a $\abs{(\varphi_{*}f)(y') - (\varphi_{*}f)(x')}\le \deg(\varphi)\, \eps$. Par cons\'equent, la fonction~$\varphi_{*}f$ est continue. Les autres r\'esultats d\'ecoulent directement des d\'efinitions.
\end{proof}

Nous disposons maintenant de tous les ingr\'edients permettant de tirer en arri\`ere les mesures.

\begin{defi}\label{def:imagereciproquemesure}
Soit $\varphi \colon X \to X'$ un morphisme fini et plat d'espaces $\cA$-analytiques. Soit~$\mu'$ une mesure de Radon sur~$X'$. L'application
\[\fonction{\varphi^{*}\mu'}{\cC(X,\R)}{\R}{f}{\mu'(\varphi_{*}f)}\]
est une mesure de Radon sur~$X$ appel\'ee \emph{image r\'eciproque} de~$\mu'$ par~$\varphi$.
\end{defi}

\begin{rema}\label{rem:pullbackproba}
Pla\c cons-nous dans le cadre de la d\'efinition~\ref{def:imagereciproquemesure} et supposons, en outre, que $X'$ est compact et connexe. On v\'erifie alors que, si~$\mu'$ est une mesure de probabilit\'e sur~$X'$, alors $\frac1{\deg(\varphi)}\, \varphi^*\mu'$ est une mesure de probabilit\'e sur~$X$.
\end{rema}

\begin{rema}
Nous avons choisi de travailler au-dessus d'un bon anneau de Banach. Cette hypoth\`ese utilis\'ee \`a deux endroits~: pour disposer d'une part du th\'eor\`eme~\ref{th:finicoherent}, qui permet de d\'efinir le degr\'e local d'un morphisme fini (\cf~d\'efinition~\ref{def:degre}), et d'autre part de la connexit\'e locale, qui est utilis\'ee dans la preuve de la proposition~\ref{prop:varphi_{*}f}.
\end{rema}

\subsection{Familles continues de mesures}\label{sec:Radonfamille}

Soit~$\cA$ un bon anneau de Banach. Soient $X$ et~$Y$ des espaces $\cA$-analytiques et $\pi \colon X \to Y$ un morphisme. 

\begin{defi}
Une \emph{$\pi$-famille de mesures} est une famille $\mu = (\mu_{y})_{y\in Y}$ o\`u, pour tout $y\in Y$, $\mu_{y}$ est une mesure de Radon sur~$\pi^{-1}(y)$. On identifiera $\mu_{y}$ \`a sa mesure image sur~$X$. 

La famille de mesures~$\mu$ est dite \emph{continue} si, pour toute $f\in \cC_{c}(X,\R)$, l'application
\[ y \in Y \mapstoo \mu_{y}(f) = \int f  \diff \mu_{y} \in \R\]
est continue.
\end{defi}

\begin{defi}
Supposons que~$\cA$ est flottant et que $X$, $Y$ et~$\pi$ sont flottants. Une  $\pi$-famille de mesures $\mu = (\mu_{y})_{y\in Y}$ est dite \emph{flottante} si, pour tout $y\in Y$ et tout $\eps\in \intof{0,1}$, on a $(\Phi_{\eps})_{\ast} \mu_{y} = \mu_{y^\eps}$.
\end{defi}

Soient $X'$ un espace $\cA$-analytique et $\pi' \colon X' \to Y$ un morphisme. Soit $\varphi \colon X \to X'$ un morphisme au-dessus de~$Y$. Nous expliquons maintenant comme tirer et pousser des familles de mesures par~$\varphi$.

\begin{defi}
Supposons que~$\varphi$ est propre et que $X$ est topologiquement s\'epar\'e. Pour toute $\pi$-famille de mesures~$\mu$, on d\'efinit une $\pi'$-famille de mesures~$\varphi_{*}\mu$ en posant, pour tout $y\in Y$,
\[ (\varphi_{*}\mu)_{y} := \varphi_*\mu_{y}.\]

Supposons que~$\varphi$ est fini et plat. Alors, pour toute $\pi'$-famille de mesures~$\mu'$, on d\'efinit une $\pi$-famille de mesures~$\varphi^{*}\mu'$ en posant, pour tout $y\in Y$,
\[ (\varphi^{*}\mu')_{y} := \varphi^{*}\mu'_{y}.\]
\end{defi}

\begin{lemm}\label{lem:imagemesurecontinue}
Supposons que~$\varphi$ est propre et que $X$ est topologiquement s\'epar\'e. Alors, pour toute $\pi$-famille de mesures continue~$\mu$, la famille~$\varphi_*\mu$ est continue.

Supposons que~$\varphi$ est fini et plat. Alors, pour toute $\pi'$-famille de mesures continue~$\mu'$, la famille~$\varphi^{*}\mu'$ est continue.
\end{lemm}
\begin{proof}
Supposons que~$\varphi$ est propre. Soit $\mu$ une $\pi$-famille de mesures continue. Soit $f \in \cC_{c}(X',\R)$. La fonction $f \circ \varphi$ est alors continue sur~$X$, car~$\varphi$ est continue, et \`a support compact, car~$\varphi$ est propre. Puisque la famille~$\mu$ est continue, la fonction 
\[y \in Y \mapstoo \int f  \diff (\varphi_{*}\mu_{y}) = \int (f \circ \varphi) \diff \mu_{y}\]
est continue. Le r\'esultat s'ensuit.

\medbreak

Supposons que~$\varphi$ est fini et plat. Le r\'esultat se d\'emontre de la m\^eme mani\`ere, en faisant appel \`a la proposition~\ref{prop:varphi_{*}f}.
\end{proof}

\section{M\'etriques sur les fibr\'es en droites}\label{sec:fibresmetrises}

Dans cette section, nous \'etudions les fibr\'es m\'etris\'es sur un espace de Berkovich au-dessus d'un anneau de Banach. La notion de fibr\'e m\'etris\'e, qui appara\^it de fa\c con essentielle en th\'eorie d'Arakelov, \cf~\cite{Arakelov} pour la version complexe ou \cite{ZhangSmallPoints} pour une version sur tout corps, s'adapte sans difficult\'es \`a notre cadre. 

Dans la section~\ref{sec:generalitesfibres}, nous \'enon\c cons la d\'efinition, ainsi que quelques propri\'et\'es simples. Dans la section~\ref{sec:fibrestopologie}, nous munissons l'ensemble des m\'etriques sur un fibr\'e donn\'e d'une topologie. Dans la section~\ref{sec:systemedynamiquepolarise} nous montrons l'existence d'une m\'etrique naturelle associ\'ee \`a un syst\`eme dynamique polaris\'e.

%Dans la section~\ref{sec:generalitesfibres}, nous adaptons, sans difficult\'es, la d\'efinition provenant de la th\'eorie d'Arakelov, puis, dans la section~\ref{sec:fibrestopologie}, nous munissons l'ensemble des m\'etriques sur un fibr\'e donn\'e d'une topologie. Enfin, dans la section~\ref{sec:systemedynamiquepolarise} nous montrons l'existence d'une m\'etrique naturelle associ\'e \`a un syst\`eme dynamique polaris\'e.

Soit $\cA$ un anneau de Banach et soit $X$ un espace $\cA$-analytique.

\subsection{G\'en\'eralit\'es}\label{sec:generalitesfibres}

\begin{defi}
Soit $L$ un fibr\'e en droites sur~$X$. Une \emph{m\'etrique~$\nm$ sur~$L$} est une famille $(\nm_{x})_{x\in X}$ o\`u, pour tout $x\in X$, 
\[ \nm_{x} \colon L(x) \too \R_{\ge 0} \]
est une norme sur~$L(x)$ v\'erifiant la propri\'et\'e suivante~:
\[ \forall \lambda \in \cH(x), \forall s \in L(x),\ \norm{\lambda s}_{x} = \abs{\lambda} \, \norm{s}_{x}.\]

La m\'etrique~$\nm$ est dite \emph{continue} si, pour tout ouvert~$U$ de~$X$ et toute section globale $s\in L(U)$, la fonction
\[ x \in U \mapstoo \norm{s(x)}_{x} \in \R_{\ge0}\]
est continue.

On note $\Met(L)$ l'ensemble des m\'etriques continues sur~$L$.
\end{defi}

\begin{rema}
Soit $x\in X$. La condition d'homog\'en\'eit\'e entra\^ine que la connaissance de~$\norm{s}_{x}$ pour un seul $s \in L(x) \setminus \{0\}$ suffit \`a d\'eterminer la norme~$\nm_{x}$.  

Dans le cas du fibr\'e trivial $L = \cO_{X}$, on peut choisir la section $s=1$. On obtient ainsi une bijection entre $\Met(\cO_{X})$ et $\cC(X,\R_{> 0})$. Nous identifierons d\'esormais ces deux ensembles.
\end{rema}

\begin{exem}\label{ex:metriquestandard}
Soit $(k,\va)$ un corps valu\'e complet. Pla\c cons-nous sur la droite projective~$\EP{1}{k}$ munie des coordonn\'ees homog\`enes $T_{0},T_{1}$. Le fibr\'e~$\cO(1)$ peut \^etre muni de la \emph{m\'etrique standard}~$\nm_{\st}$ caract\'eris\'ee par la propri\'et\'e suivante~: pour toute $s\in \Gamma(\EP{1}{k},\cO(1))$, repr\'esent\'ee par $P_{s}(T_{0},T_{1}) \in k[T_{0},T_{1}]$ homog\`ene de degr\'e~1, on a
\[\forall x\in \EP{1}{k},\ \norm{s}_{\st,x} := \frac{\abs{P_{s}(T_{0},T_{1})(x)}}{\max(\abs{T_{0}(x)}, \abs{T_{1}(x)})}.\]
\end{exem}

On dispose d'op\'erations classiques sur les fibr\'es m\'etris\'es, que nous rappelons ici. Soit $L$ un fibr\'e en droites sur~$X$ muni d'une m\'etrique~$\nm$.

Pour tout $d\in \Z$, il existe une unique m\'etrique~$\nm^{\otimes d}$ sur~$L^{\otimes d}$ telle que
\[\forall x\in X, \forall s\in L(x),\ \norm{s^{\otimes d}}^{\otimes d}_{x} = \norm{s}^d_{x}.\]

Supposons qu'il existe~$e\in \Z^\ast$ et un fibr\'e en droites~$M$ sur~$X$ tel que $L\simeq M^{\otimes e}$. Il existe alors une unique m\'etrique~$\nm^{\otimes 1/e}$ sur~$M$ telle que
\[\forall x\in X, \forall t\in M(x),\ \norm{t}^{\otimes 1/e}_{x} = \norm{t^{\otimes e}}^{1/e}_{x}.\]

Soit $\varphi \colon X\to X$ un endomorphisme de~$X$. Il existe alors une unique m\'etrique~$\varphi^\ast \nm$ sur $\varphi^* L$ telle que 
\[\forall x\in X, \forall s\in L(x),\ (\varphi^\ast\nm)_{x}(\varphi^*(s)) = \norm{s}_{\varphi(x)}.\]

Soit $L'$ un fibr\'e en droites sur~$X$ muni d'une m\'etrique~$\nm'$. Il existe alors une unique m\'etrique~$\nm \otimes \nm'$ sur $L\otimes L'$ telle que
\[\forall x\in X, \forall s\in L(x), \forall s'\in L'(x),\ (\nm \otimes \nm')_{x}(s\otimes s') = \norm{s}_{x}\, \norm{s'}_{x}.\]

Toutes ces op\'erations pr\'eservent la continuit\'e des m\'etriques.

\subsection{Topologie}\label{sec:fibrestopologie}

L'objet de cette section est de munir l'ensemble des m\'etriques sur un fibr\'e en droites d'une structure topologique. Nous y parvenons en construisant une famille d'\'ecarts\footnote{Un \emph{\'ecart} est une distance g\'en\'eralis\'ee, \`a valeurs dans $\intff{0,+\infty}$, o\`u la condition de positivit\'e stricte  sur des couples de points distincts est rel\^ach\'ee, \cf~\cite[IX, \S 1, \no 1, D\'efinition~1]{BourbakiTG510}.}. 

%L'objet de cette section est de munir l'ensemble des m\'etriques sur un fibr\'e en droites d'une structure topologique. Nous y parvenons en construisant une famille d'\'ecarts, au sens de la d\'efinition suivante.
%
% \begin{defi}[\protect{\cite[IX, \S 1, \no 1, D\'efinition~1]{BourbakiTG510}}]
% Soit~$E$ un ensemble. Un \emph{\'ecart} sur~$E$ est une application $d \colon E\times E \to \intff{0,+\infty}$ satisfaisant aux conditions suivantes~:
% \begin{enumerate}[i)]
% \item $\forall x \in E$, $d(x,x)=0$ ;
% \item $\forall x,y\in X$, $d(x,y)=d(y,x)$ ;
% \item $\forall x,y,z \in X$, $d(x,y) \le d(x,z)+d(z,y)$. 
% \end{enumerate}
% \end{defi}

%Pr\'esentons maintenant la construction principale de cette section.

\begin{nota}
Soient $L$ un fibr\'e en droites sur~$X$ et~$\nm_{1}$ et~$\nm_{2}$ des m\'etriques continues sur~$L$. Alors $\nm_{1} \otimes \nm_{2}^{\otimes -1}$ est une m\'etrique continue sur~$\cO_{X}$ et, pour toute partie compacte~$K$ de~$X$, on pose
\[ d_{K}(\nm_{1},\nm_{2}) := \max_{x\in K} \big( |\log( (\nm_{1} \otimes \nm_{2}^{\otimes -1})_{x}(1))| \big) \in \R_{\ge 0}.\]
\end{nota}

\begin{lemm}
Soient $L$ un fibr\'e en droites sur~$X$. Pour toute partie compacte~$K$ de~$X$, $d_{K}$ est un \'ecart sur~$\Met(L)$.
\qed
\end{lemm}

\begin{lemm}\label{lem:MetC}
Soient $L$ un fibr\'e en droites sur~$X$ et $\nm_{0} \in \Met(L)$. L'application
\[ \nm \mapstoo \log( \nm \otimes \nm_{0}^{\otimes -1})\]
induit une bijection entre $\Met(L)$ et $\cC(X,\R)$. Pour tout partie compacte~$K$ de~$X$, elle envoie l'\'ecart~$d_{K}$ sur l'\'ecart induit par la semi-norme uniforme sur~$K$.
\qed 
\end{lemm}

Nous munirons d\'esormais $\Met(L)$ de la famille d'\'ecarts $(d_{K})_{K}$, o\`u $K$ d\'ecrit l'ensemble des parties compactes de~$X$, et de la structure uniforme correspondante (\cf~\cite[IX, \S 1, \no 2, D\'efinition~2]{BourbakiTG510}). 

\begin{prop}\label{prop:MetLcomplet}
L'espace uniforme~$\Met(L)$ est complet.
\end{prop}
\begin{proof}
Si~$\Met(L)$ est vide, l'\'enonc\'e est satisfait. Sinon, d'apr\`es le lemme~\ref{lem:MetC}, il suffit de d\'emontrer que l'espace~$\cC(X,\R)$, muni de la structure uniforme de la convergence compacte est complet. Puisque~$X$ est localement compact, le r\'esultat d\'ecoule de \cite[IX, \S 1, \no 6, cor.~3 du th.~2]{BourbakiTG510}.
\end{proof}

On peut d\'ecrire simplement la fa\c con dont les \'ecarts sont modifi\'es par certaines op\'erations usuelles sur les m\'etriques. 

\begin{lemm}\label{lem:variationecarts}
Soient $L$ un fibr\'e en droites sur~$X$ et~$\nm_{1}$ et~$\nm_{2}$ des m\'etriques continues sur~$L$. Soit~$K$ une partie compacte de~$X$. 

Pour tout $d\in \Z^*$, on a 
\[ d_{K}(\nm_{1}^{\otimes d}, \nm_{2}^{\otimes d}) = |d|\cdot d_{K}(\nm_{1}, \nm_{2}).\]
Pour tout endomorphisme surjectif $\varphi$ de~$X$, on a 
\[ d_{K}(\varphi^*\nm_{1}, \varphi^*\nm_{2}) = d_{K}(\nm_{1}, \nm_{2}).\]
\qed
\end{lemm}

\subsection{Cas d'un syst\`eme dynamique polaris\'e}\label{sec:systemedynamiquepolarise}

%Sous certaines conditions, on montre l'existence de m\'etriques continues naturelles.
%
%\begin{defi}\label{def:polarise}
%Un endomorphisme surjectif~$\varphi$ de~$X$ est dit \emph{polaris\'e} s'il existe un fibr\'e en droites~$L$ sur~$X$, des entiers $d,e \in \Z^*$ avec $\abs{e} < \abs{d}$ et un isomorphisme
%\[ \theta \colon (\varphi^* L)^{\otimes e} \simtoo L^{\otimes d}.\]
%\end{defi}

Nous nous pla\c cons d\'esormais dans la situation o\`u nous disposons 
\begin{itemize}
\item d'un fibr\'e en droites~$L$ sur~$X$ ;
\item d'un endomorphisme surjectif~$\varphi$ de~$X$
\end{itemize} 
soumis \`a la condition suivante~: il existe $d,e \in \Z^*$ avec $\abs{e} < \abs{d}$ et un isomorphisme
\[ \theta \colon (\varphi^* L)^{\otimes e} \simtoo L^{\otimes d}.\]
L'objet de cette section est de montrer qu'il existe une m\'etrique continue sur~$L$ naturellement associ\'ee \`a ces donn\'ees. Dans le cadre des sch\'emas sur un corps valu\'es, ce r\'esultat est d\^u \`a S.-W.~Zhang (\cf~\cite[theorem~2.2]{ZhangSmallPoints}). Nous suivons ici la preuve de \cite[theorem~9.5.4]{BombieriGubler}.

\medbreak

Consid\'erons l'application
\[\fonction{\Phi}{\Met(L)}{\Met(L)}{\nm}{\big( (\varphi^*\nm)^{\otimes e} \circ \theta^{-1} \big)^{\otimes 1/d}}.\]
D'apr\`es le lemme~\ref{lem:variationecarts}, pour toutes m\'etriques continues $\nm_{1}$ et~$\nm_{2}$ sur~$L$ et toute partie compacte~$K$ de~$X$, on a
\[ d_{K}(\Phi(\nm_{1}),\Phi(\nm_{2})) = \frac{\abs{e}}{\abs{d}} \, d_{K}(\nm_{1},\nm_{2}).\]

\begin{rema}\label{rem:Phiexplicite}
Soit $s_{1} \in L(\varphi(x))$. Alors, $(\varphi^*s_{1})^{\otimes e}$ induit un \'el\'ement de~$(\varphi^* L)^{\otimes e}(x)$, que nous noterons identiquement. En outre, si $s_{1} \ne 0$, alors $(\varphi^*s_{1})^{\otimes e} \ne 0$ et, pour tout $s \in (\varphi^* L)^{\otimes e}(x)$, $s/(\varphi^*s_{1})^{\otimes e}$ fait sens et s'identifie \`a un \'el\'ement de~$\cH(x)$. On peut donc \'ecrire $\abs{s/(\varphi^*s_{1})^{\otimes e}}$ sans ambigu\"it\'e.

Soit~$\nm\in\Met(L)$ et posons $^\Phi \nm := \Phi(\nm)$.  Avec les notations pr\'ec\'edentes, pour tout $s_{2} \in L(x)$, on a
\[^\Phi \norm{s_{2}}_{x}^d = \Big\lvert \frac{\theta^{-1}(s_{2}^{\otimes d})}{(\varphi^*s_{1})^{\otimes e}} \Big\rvert \, \norm{s_{1}}_{\varphi(x)}^e.\]
\end{rema}

\begin{theo}\label{th:Phinconverge}
Supposons que $\Met(L)$ ne soit pas vide. Alors, il existe une unique m\'etrique continue~$\nm$ sur~$L$ telle que 
\[ (\varphi^* \nm)^{\otimes e} = \nm^{\otimes d} \circ \theta \textrm{ dans } \Met(\varphi^*L). \]
Nous la noterons~$\nm_{\varphi}$.

En outre, pour tout $\nm_{0} \in \Met(L)$, la suite $(\Phi^n(\nm_{0}))_{n\in \N}$ converge vers~$\nm_{\varphi}$.
\end{theo}
\begin{proof}
Posons $c := \abs{e}/\abs{d} \in \intoo{0,1}$. Soit $\nm_{0} \in \Met(L)$. D'apr\`es le lemme~\ref{lem:variationecarts}, pour tout compact~$K$ de~$X$ et tout $n\in \N$, on a
\[ d_{K}(\Phi^{n+1}(\nm_{0}),\Phi^{n}(\nm_{0})) = c^n \, d_{K}(\Phi(\nm_{0}),\nm_{0}).\]
On en d\'eduit que la suite $(\Phi^n(\nm_{0}))_{n\in \N}$ est de Cauchy, et donc convergente, d'apr\`es la proposition~\ref{prop:MetLcomplet}. Notons~$\nm_{\varphi}$ sa limite.

Par construction, on a $\Phi(\nm_{\varphi}) = \nm_{\varphi}$, et $\nm_{\varphi}$ satisfait donc l'\'egalit\'e de l'\'enonc\'e. Nous avons d\'ej\`a d\'emontr\'e la seconde partie. La propri\'et\'e d'unicit\'e de la mesure en d\'ecoule.
\end{proof}

\section[Laplacien sur la droite projective]{Laplacien sur la droite projective sur un corps valu\'e complet}\label{sec:Laplacien}

Cette section est consacr\'ee \`a l'op\'erateur laplacien sur la droite projective sur un corps valu\'e complet. Elle est essentiellement constitu\'ee de rappels. Nous traitons d'abord le cas complexe classique, celui de $\P^1(\C)$ (section~\ref{sec:laplacienC}), puis d'un corps archim\'edien quelconque (section~\ref{sec:corpsarchimedienquelconque}), et enfin le cas ultram\'etrique (section~\ref{sec:laplacienultrametrique}). Nous optons ici pour une pr\'esentation rapide, mais invitons le lecteur int\'eress\'e par le d\'etail des preuves et constructions \`a lire la version ant\'erieure de ce texte \cite{DynamiqueIarXivv2}.

\begin{nota}
Soit $(k,\va)$ un corps valu\'e complet. Pour $c\in k$ et $r\in \R_{> 0}$, on note $D_{k}(c,r)$ (resp. $\oD_{k}(c,r)$) le disque ouvert (resp. ferm\'e) de centre~$c$ et de rayon~$r$ dans $\E{1}{k}$. 

On s'autorise \`a supprimer le corps~$k$ de la notation lorsque le contexte permet de l'identifier sans ambigu\"it\'e.
\end{nota}

\subsection{Le cas complexe usuel}\label{sec:laplacienC} 

Pour les d\'etails de la th\'eorie dans ce cadre classique, on renvoie \`a \cite{Ransford}.

Consid\'erons le corps~$\C$ muni de sa valeur absolue usuelle~$\va_{\infty}$. Fixons une coordonn\'ee~$z$ sur~$\C$. 

Soit $U$ un ouvert de~$\C$. \'Etant donn\'ee une fonction $f \in \cC^2(U,\R)$, on d\'efinit classiquement son laplacien par 
\[ \Delta f := \frac1{2\pi} \, \frac{\partial^2 f}{\partial z \partial \bar z} \in  \cC^0(U,\R).\]
Dans la suite, nous l'identifierons \`a la mesure associ\'ee $\Delta f \diff z \diff \bar z$. Consid\'erant d\'esormais le laplacien comme un op\'erateur \`a valeurs dans l'ensemble des mesures de Radon sur~$U$, on peut \`a l'ensemble des fonctions sous-harmoniques, not\'e $\SH(U)$, puis \`a celui des  fonctions qui sont localement diff\'erences de fonctions sous-harmoniques, not\'e $\DSH(U)$. Le laplacien, dans sa version mesure, \'etant ind\'ependant de la coordonn\'ee choisie, on peut \'etendre sa d\'efinition aux fonctions sur les ouverts de~$\P^1(\C)$ (ou toute surface de Riemann).

Les propri\'et\'es suivantes sont classiques et nous les utiliserons sans plus de pr\'ecautions dans le reste du texte :

\begin{itemize}
\item Pour tout $U$ ouvert de $\P^1(\C)$ et toutes $u,v \in \cC_{c}(U,\R) \cap \DSH(U)$, on a
\[ \int u \diff \Delta v =  \int v \diff \Delta u.\]
\item Pour tout morphisme analytique fini $\varphi \colon U' \to U$ entre ouverts de $\P^1(\C)$ et toute $u \in \DSH(U)$, on a $u \circ \varphi \in \DSH(U')$ et
\[ \Delta (u \circ \varphi) = \varphi^\ast \Delta u.\]
\end{itemize}

Rappelons \'egalement la formule de Poincar\'e-Lelong.

\begin{theo}\label{th:Poincare-Lelong}
Soient $U$ un ouvert de~$\P^1(\C)$ et $f$ une fonction analytique sur~$U$, identiquement nulle sur aucune composante connexe. Alors, la fonction $\log(\abs{f})$ est sous-harmonique sur~$U$ et on a 
\[ \Delta \log(\abs{f}) = \delta_{\div(f)}.\]
\qed
\end{theo}

Nous concluons cette section en \'enon\c cant un r\'esultat permettant de majorer l'int\'egrale de $\Delta u$ sur un disque en termes de la norme uniforme de~$u$ sur un disque plus grand. Ce r\'esultat (et une g\'en\'eralisation en dimension sup\'erieure) se trouve dans \cite[proposition~2.6]{BedfordTaylor}. Nous reproduisons ici la preuve \`a la fois pour la commodit\'e du lecteur et pour rendre explicite la constante apparaissant au membre de droite. %(avec une normalisation diff\'erente pour le laplacien).

\begin{prop}\label{prop:majorationSH}
Soient~$U$ un ouvert de~$\C$ et~$u \in \cC(U,\R) \cap \SH(U)$. Soient $t\in U$ et $r,R \in \R_{>0}$ avec $r<R$ tels que $\oD(t,R) \subset U$. Alors, pour toute $f\in \cC_{c}(D(t,r),\R)$, on a 
\begin{align*}
\int f \diff \Delta u &\le \frac{1}{\log(R/r)}\,\norm{f}_{\oD(t,r)}\, \big(\norm{u}_{\oD(t,R)}-u(t)\big)\\ 
&\le \frac{2}{\log(R/r)}\, \norm{f}_{\oD(t,r)} \norm{u}_{\oD(t,R)}.
\end{align*}
\end{prop}
\begin{proof}
Par un argument classique d'approximation, il suffit de prouver le r\'esultat lorsque $u \in \cC^\infty(U,\R)$. Dans ce cas, il suffit de montrer que 
\[\frac1{2\pi} \int_{D(t,r)}  \frac{\partial^2 u}{\partial z \partial \bar z} \diff \lambda  \le \frac{1}{\log(R/r)}\, \big(\norm{u}_{\oD(t,R)}-u(t)\big).\]

Pour $s \in \intoo{0,R}$, posons 
\[n(s) := \frac{1}{2\pi} \int_{D(t,s)} \frac{\partial^2 u}{\partial z \partial \bar z} \diff \lambda.\]
D'apr\`es la formule de Jensen, on a
\[\int_{0}^R \frac{n(s)}{s} \diff s = \frac1{2\pi} \int_{-\pi}^{\pi} u(t+R e^{i\theta}) \diff \theta - u(t) \le \norm{u}_{\oD(t,R)} - u(t).\] 
Pour $s \in \intoo{0,R}$, on a
\[\int_{0}^R \frac{n(s)}{s} \diff s \ge \int_{r}^R \frac{n(s)}{s} \diff s \ge n(r) \int_{r}^R \frac{1}{s} \diff s\ge n(r) \log\Big(\frac R r\Big).\]
Le r\'esultat s'ensuit.
\end{proof}

\subsection{Le cas d'un corps archim\'edien quelconque}\label{sec:corpsarchimedienquelconque}

Dans cette section, nous montrons que les r\'esultats de la section~\ref{sec:laplacienC} restent valables pour tout corps valu\'e complet archim\'edien, c'est-\`a-dire pour les corps~$\C_{\eps}$ et~$\R_{\eps}$, avec $\eps \in \intof{0,1}$. 

La t\^ache principale consiste \`a d\'efinir le laplacien dans ce nouveau contexte, ce dont nous nous acquittons en nous ramenant au cas de~$\C$. L'adaptation des r\'esultats se r\'esume alors \`a un jeu de d\'efinitions.

%\begin{defi}
%Soit $(k,\va)$ un corps valu\'e complet archim\'edien. Une \emph{courbe $k$-analytique lisse} est un espace localement annel\'e localement isomorphe \`a un ouvert de~$\E{1}{k}$.
%\end{defi}

\subsubsection{Le cas de $\C_{\eps}$.}\label{sec:laplacienCeps}

Soit $\eps \in \intof{0,1}$ et consid\'erons le corps~$\C_{\eps}$, c'est-\`a-dire le corps~$\C$ muni de la valeur absolue $\va_{\infty}^\eps$. Nous utiliserons les notations et r\'esultats de la section~\ref{sec:flotcorpsvalue}. Rappelons que l'on dispose d'un isomorphisme d'espaces localement annel\'es
\[ \Phi_{\eps} \colon \P^1(\C) = \EP{1}{\C} \simtoo \EP{1}{\C_{\eps}}.\]
%et que, pour tout ouvert~$U$ de~$\P^1(\C)$, on note $U_{\eps} := \Phi_{\eps}(U)$.
%est particuli\`erement simple. En effet, il est d\'efini au niveau des semi-normes multiplicatives comme l'\'el\'evation \`a la puissance~$\eps$. Cette op\'eration pr\'eservant les noyaux des semi-normes, $\Phi_{\eps}$ est l'identit\'e sur les espaces sous-jacents.

%Soit~$X$ une courbe $\C_{\eps}$-analytique lisse. On dispose d'une courbe~$X_{\eps^{-1}}$ qui est $\C$-analytique, autrement dit une surface de Riemann, et donc justiciable de la section~\ref{sec:laplacienC}.

\begin{nota}
Soit $U$ un ouvert de~$\EP{1}{\C_{\eps}}$. Pour toute fonction  $u \colon U \to \R \cup \{\pm\infty\}$, on d\'efinit une fonction~$u_{\eps^{-1}}$ sur~$ \Phi_{\eps}^{-1} (U)$ par 
\[u_{\eps^{-1}} := u \circ \Phi_{\eps} \colon  \Phi_{\eps}^{-1} (U) \too  \R \cup \{\pm\infty\}.\]
\end{nota}

On \'etend les d\'efinitions des propri\'et\'es des fonctions de~$\P^1(\C)$ \`a $\EP{1}{\C_{\eps}}$ en disant qu'une fonction $u$ sur un ouvert de~$\EP{1}{\C_{\eps}}$ est continue, sous-harmonique, etc. lorsque la fonction~$u_{\eps^{-1}}$ l'est. 

%\begin{defi}
%Soit $(P)$ une propri\'et\'e des fonctions sur une surface de Riemann. On dit qu'une fonction $u \colon X \to \R \cup \{\pm\infty\}$ satisfait~$(P)$ si la fonction~$u_{\eps^{-1}}$ satisfait~$(P)$. 
%\end{defi}

\begin{defi}\label{def:laplacienCeps}
Soit $U$ un ouvert de~$\EP{1}{\C_{\eps}}$ et soit $u\in \DSH(U)$. On d\'efinit le laplacien de~$u$ comme la mesure de Radon
\[\fonction{\Delta u}{\cC_{c}(U,\R)}{\R}{f}{\frac1\eps \, \disp \int  f_{{\eps}^{-1}} \diff \Delta u_{{\eps}^{-1}}}.\]
%\[\fonction{\Delta u}{\cC_{c}(U,\R)}{\R}{f}{\frac1\eps \,\Delta u_{{\eps}^{-1}} (f_{{\eps}^{-1}})}.\]
\end{defi}

Les r\'esultats de la section~\ref{sec:laplacienC} (sym\'etrie, comportement par r\'etrotirette, formule de Poincar\'e-Lelong) s'adaptent imm\'ediatement sur les ouverts $\EP{1}{\C_{\eps}}$. 

D\'etaillons le r\'esultat final de majoration (\cf~proposition~\ref{prop:majorationSH}). 

\begin{prop}\label{prop:majorationSHCeps}
Soient~$U$ un ouvert de~$\E{1}{\C_{\eps}}$ et $u \in \cC(U,\R) \cap \SH(U)$. Soient $t\in U$ et $r,R \in \R_{>0}$ avec $r<R$ tels que $\oD_{\C_{\eps}}(t,R) \subset U$. Alors, pour toute $f\in \cC_{c}(D(t,r),\R)$, on a 
\begin{align*}
\int f \diff \Delta u &\le \frac{1}{\log(R/r)}\,\norm{f}_{\oD(t,r)}\, \big(\norm{u}_{\oD(t,R)}-u(t)\big)\\ 
&\le \frac{2}{\log(R/r)}\, \norm{f}_{\oD(t,r)} \norm{u}_{\oD(t,R)}.
\end{align*}
\end{prop}
\begin{proof}
Par d\'efinition, pour tout $s\in \R_{>0}$, on a
\[\Phi_{\eps}^{-1}(D_{\C_{\eps}}(t,s)) = D_{\C}(t,s^{1/\eps})\]
et, pour toute fonction~$g$ d\'efinie au voisinage de~$\oD_{\C_{\eps}}(t,s)$, on a
\[\norm{g}_{\oD_{\C_{\eps}}(t,s)} = \norm{g_{\eps^{-1}}}_{\Phi_{\eps}^{-1}(\oD_{\C_{\eps}}(t,s))} = \norm{g_{\eps^{-1}}}_{\oD_{\C}(t,s^{1/\eps})}.\]
D'apr\`es la d\'efinition~\ref{def:laplacienCeps} et la proposition~\ref{prop:majorationSH}, on a
\begin{align*}
\int f \diff \Delta u & = \frac1\eps \, \int f_{\eps^{-1}} \diff \Delta u_{\eps^{-1}}\\
& \le \frac1\eps \, \frac{1}{\log(R^{1/\eps}/r^{1/\eps})}\, \norm{f_{\eps^{-1}}}_{\oD_{\C}(t,r^{1/\eps})} \,\big(\norm{u_{\eps^{-1}}}_{\oD_{\C}(t,R^{1/\eps})} - u_{\eps^{-1}}(t)\big)\\ 
&\le \frac{1}{\log(R/r)}\,\norm{f}_{\oD(t,r)}\, \big(\norm{u}_{\oD(t,R)}-u(t)\big).
\end{align*}
\end{proof}

%Consid\'erons la droite affine $\E{1}{\C_{\eps}}$ avec coordonn\'ee~$T$. Par analogie avec la notation~\ref{nota:hrR}, d\'efinissons
%\[\fonction{h_{r,R}}{\E{1}{\C_{\eps}}}{\R_{\ge 0}}{z}{h_{r,R}(\log(\abs{z}))}.\]
%
%\begin{lemm}\label{lem:majorationSHCeps}
%Soient $r, R \in \R_{>0}$ avec $r < R$. Pour tout ouvert~$U$ de~$\E{1}{\C_{\eps}}$ contenant $\oD_{\C_{\eps}}(R)$ et toute $u \in \cC(U,\R) \cap \SH(U)$, on a 
%\[\int h_{r,R} \diff \Delta u \le C_{r,R}\, \norm{u}_{\oD(R)},\]
%o\`u la constante~$C_{r,R}$ est celle qui appara\^it dans le lemme~\ref{lem:majorationSH}.
%\end{lemm}
%\begin{proof}
%Par d\'efinition, on a 
%\[\int h_{r,R} \diff \Delta u  = \frac1\eps \int (h_{r,R})_{\eps^{-1}} \diff \Delta u_{\eps^{-1}}.\]
%On a \'egalement
%\[ \forall z \in \C,\, (h_{r,R})_{\eps^{-1}}(z) = h_{r,R}(\eps \log(\abs{z})).\]
%En particulier, $\Delta (h_{r,R})_{\eps^{-1}}$ est nulle sur $\oD(r^{1/\eps})$ et hors de $\oD(R^{1/\eps})$ et 
%\[\forall z \in \C^\ast,\,  \frac{\partial^2 (h_{r,R})_{\eps^{-1}}}{\partial z \partial \bar z}(z) = \frac{\eps^2}{4\abs{z}^2} \, \alpha_{r,R}''(\eps \log(\abs{z})).\] 
%Le m\^eme raisonnement que dans la preuve du lemme~\ref{lem:majorationSH} fournit alors le r\'esultat annonc\'e.
%\end{proof}

\subsubsection{Le cas de $\R_{\eps}$.}\label{sec:laplacienReps}

Soit $\eps \in \intof{0,1}$ et consid\'erons le corps~$\R_{\eps}$, c'est-\`a-dire le corps~$\R$ muni de la valeur absolue $\va_{\infty}^\eps$.

Soit~$U$ un ouvert de $\EP{1}{\R_{\eps}}$. En effectuant le changement de base par le morphisme born\'e $\R_{\eps} \to \C_{\eps}$, on obtient un ouvert $U\ho{\R_{\eps}} \C_{\eps}$ de $\E{1}{\C_{\eps}}$, auquel on peut appliquer les r\'esultats de la section~\ref{sec:laplacienCeps}. Notons $\pr_{\C_{\eps},\R_{\eps}} \colon U\ho{\R_{\eps}} \C_{\eps} \to U$ le morphisme de changement de base.

%Soit~$X$ une courbe $\R_{\eps}$-analytique lisse. En effectuant le changement de base par le morphisme born\'e $\R_{\eps} \to \C_{\eps}$, on obtient une courbe $\C_{\eps}$-analytique lisse $X\ho{\R_{\eps}} \C_{\eps}$, \`a laquelle on peut appliquer les r\'esultats de la section~\ref{sec:laplacienCeps}. Notons $\pr_{\C_{\eps},\R_{\eps}} \colon X\ho{\R_{\eps}} \C_{\eps} \to X$ le morphisme de changement de base.

\begin{nota}
Soit~$U$ un ouvert de $\EP{1}{\R_{\eps}}$. Pour toute fonction  $u \colon U \to \R \cup \{\pm\infty\}$, on d\'efinit une fonction~$u_{\C_{\eps}}$ sur~$U \ho{\R_{\eps}} \C_{\eps}$ par 
\[u_{\C_{\eps}} := u \circ \pr_{\C_{\eps},\R_{\eps}} \colon  U\ho{\R_{\eps}} \C_{\eps}\too  \R \cup \{\pm\infty\}.\]
\end{nota}

On \'etend les d\'efinitions des propri\'et\'es des fonctions de~$\EP{1}{\C_{\eps}}$ \`a $\EP{1}{\R_{\eps}}$ en disant qu'une fonction $u$ sur un ouvert de~$\EP{1}{\R_{\eps}}$ est continue, sous-harmonique, etc. lorsque la fonction~$u_{\C_\eps}$ l'est.

%\begin{defi}
%Soit $(P)$ une propri\'et\'e des fonctions sur une courbe $\C_{\eps}$-analytique. On dit qu'une fonction $u \colon X \to \R \cup \{\pm\infty\}$ satisfait~$(P)$ si la fonction~$u_{\C_{\eps}}$ satisfait~$(P)$. 
%\end{defi}

\begin{defi}\label{def:laplacienReps}
Soit~$U$ un ouvert de $\EP{1}{\R_{\eps}}$ et soit $u\in \DSH(U)$. On d\'efinit le laplacien de~$u$ comme la mesure de Radon
\[ \Delta u :=  (\pr_{\C_{\eps},\R_{\eps}})_{*} \Delta u_{\C_{\eps}}.\]
\end{defi}

Tous les r\'esultats de la section~\ref{sec:laplacienC} (sym\'etrie, comportement par r\'etrotirette, formule de Poincar\'e-Lelong, majoration sur les disques) se transf\`erent ais\'ement sur les ouverts $\EP{1}{\R_{\eps}}$.

\subsection{Le cas d'un corps ultram\'etrique}\label{sec:laplacienultrametrique}

Soit $(k,\va)$ un corps valu\'e ultram\'etrique complet. Il existe, dans le cadre de la droite projective de Berkovich~$\EP{1}{k}$, une th\'eorie du potentiel analogue \`a la th\'eorie complexe. Dans ce texte, nous n'aurons pas besoin d'en conna\^itre les d\'etails, mais seulement de pouvoir utiliser quelques propri\'et\'es formelles, \'enonc\'ees ci-dessous. Nous renvoyons le lecteur int\'eress\'e aux r\'ef\'erences classiques \cite{ValuativeTree,BR,TheseThuillier}, cette derni\`ere r\'ef\'erence traitant aussi du cas des courbes. La version ant\'erieure de ce texte \cite{DynamiqueIarXivv2} propose \'egalement une pr\'esentation plus d\'etaille.

%Indiquons simplement que la th\'eorie du potentiel ultram\'etrique s'appuie de fa\c con essentielle sur la structure d'arbre r\'eel de~$\EP{1}{k}$. Par exemple, une fonction~$u$ est dite lisse si elle peut s'\'ecrire comme compos\'ee de la r\'etraction sur un sous-arbre fini~$\Gamma$ de~$\EP{1}{k}$ et d'une fonction continue lin\'eaire par morceaux~$u_{\Gamma}$ sur~$\Gamma$. Le laplacien de~$u$ est alors une somme de masses de Dirac support\'ees aux points o\`u $u_{\Gamma}$ change de pente, le coefficient correspondant \`a un tel point~$p$ \'etant donn\'e par la somme des pentes de la fonction sur les ar\^etes issues de~$p$.

%Les propri\'et\'es suivantes sont classiques et nous les utiliserons sans plus de pr\'ecautions dans le reste du texte.
\'Etant donn\'e un ouvert~$U$ de~$\EP{1}{k}$, on peut d\'efinir l'ensemble~$\SH(U)$ des fonctions sous-harmoniques sur~$U$, puis l'ensemble~$\DSH(U)$ des fonctions localement diff\'erences de fonctions sous-harmoniques. Pour toute telle fonction~$u$, on peut  d\'efinir un laplacien~$\Delta u$ qui est une mesure de Radon. Les propri\'et\'es suivantes sont satisfaites :

\begin{itemize}
\item Pour tout $U$ ouvert de $\EP{1}{k}$ et toutes $u,v \in \cC_{c}(U,\R) \cap \DSH(U)$, on a
\[ \int u \diff \Delta v =  \int v \diff \Delta u.\]
\item Pour tout morphisme analytique fini $\varphi \colon U' \to U$ entre ouverts de $\EP{1}{k}$ et toute $u \in \DSH(U)$, on a $u \circ \varphi \in \DSH(U')$ et
\[ \Delta (u \circ \varphi) = \varphi^\ast \Delta u.\]
\end{itemize}

La formule de Poincar\'e-Lelong reste \'egalement valable.

\begin{theo}
Soient $U$ un ouvert de~$\EP{1}{k}$ et $f$ une fonction analytique sur~$U$, identiquement nulle sur aucune composante connexe. Alors, la fonction $\log(\abs{f})$ est sous-harmonique sur~$U$ et on a 
\[ \Delta \log(\abs{f}) = \delta_{\div(f)}.\]
\qed
\end{theo}

Ajoutons  un r\'esultat de majoration analogue \`a celui de la proposition~\ref{prop:majorationSH}. Nous renvoyons \`a la section~\ref{sec:corpsum} pour la d\'efinition des points~$\eta_{z,s}$.

\begin{prop}\label{prop:majorationSHdisqueum}
Soient $U$ un ouvert de~$\E{1}{k}$ et $u \in \cC(U,\R)\cap \SH(U)$. Soient $t\in U(k)$ et $r,R \in \R_{>0}$ avec $r<R$ tels que $\oD(t,R) \subset U$. Alors, pour toute $f\in \cC_{c}(D(t,r),\R)$ on a 
\begin{align*}
\int f \diff \Delta u &\le \frac{1}{\log(R/r)}\,  \norm{f}_{D(t,r)}\, (u(\eta_{t,R}) -  u(\eta_{t,r}))\\ 
%\int f \diff \Delta u &\le \frac{1}{\log(R/r)}\,  \norm{f}_{D(t,r)}\, \Big(\norm{u}_{D(t,R)} -  \inf_{x\in \oD(t,r)}(u(x))\Big)\\ 
&\le \frac{2}{\log(R/r)}\, \norm{f}_{D(t,r)}\,\norm{u}_{D(t,R)}.
\end{align*}
\end{prop}
\begin{proof}
Par un argument d'approximation, il suffit de prouver le r\'esultat pour une fonction $u$ lisse sur~$U$. Dans ce cas, il suffit de montrer que 
\[\int \diff \Delta u_{\vert D(t,r)} \le \frac{1}{\log(R/r)}\, (u(\eta_{t,R}) - u(\eta_{t,r})).\]
%\[\int \diff \Delta u_{\vert D(t,r)} \le \frac{1}{\log(R/r)}\, \Big(\norm{u}_{D(t,R)} - \inf_{x\in \oD(t,r)}(u(x))\Big).\]

La fonction~$u$ \'etant lisse, il existe un sous-arbre fini non vide~$\Gamma$ de~$\oD(t,R)$ et une fonction continue lin\'eaire par morceaux $u_{\Gamma} \colon \Gamma \to \R$ telle que $u_{\vert \oD(t,R)} = u_{\Gamma} \circ r_{\Gamma}$, o\`u $r_{\Gamma} \colon \oD(t,R) \to \Gamma$ d\'esigne la r\'etraction sur~$\Gamma$. Quitte \`a agrandir~$\Gamma$, on peut supposer qu'il contient le segment $\intff{\eta_{t,r},\eta_{t,R}}$ reliant le bord de~$\oD(t,r)$ \`a celui de~$\oD(t,R)$.

Pour tout point~$x$ de~$\Gamma$, notons $\Gamma_{x}$ l'ensemble des germes d'ar\^etes de~$\Gamma$ d'origine~$x$ (autrement dit, l'ensemble des directions sortantes en~$x$). Pour tout $a \in A_{x}$, la pente de~$u$ le long de~$a$ en~$x$ est bien d\'efinie. Notons-la~$\partial_{a}  u$. Par d\'efinition, pour tout ouvert~$V$ de~$U$ contenu dans $\oD(t,R)$, on a (\cf~\cite[sections 1.2.1 et 3.2.4]{TheseThuillier})
\[ \Delta u_{\vert V} = \sum_{x\in \Gamma \cap V} \Big(\sum_{a \in \Gamma_{x}} \partial_{a} u\Big) \delta_{x},\]
o\`u~$\delta_{x}$ d\'esigne la masse de Dirac en~$x$. La fonction~$u$ \'etant lin\'eaire par morceaux, la quanti\'e $\sum_{a \in A_{x}} \partial_{a} u$ est nulle pour presque tout~$x$, et $\Delta u_{\vert V}$ est bien d\'efinie.

%\cite[\S 3.5]{BR}
%
%\cite[\S 1.2.1 et \S 3.2.4]{TheseThuillier}

Puisque $u$ est sous-harmonique, $\Delta u$ est une mesure positive et on a
\[ \int \Delta u_{\vert D(t,r)} \le \int \Delta u_{\vert \oD(t,r) \setminus \{\eta_{t,r}\} }.\]
Puisque $\oD(t,r)$ est compact, $\Gamma \cap \oD(t,r)$ est un arbre fini et on a
\[ \sum_{x\in \Gamma \cap \oD(t,r)} \Big(\sum_{a \in \Gamma_{x} \cap \oD(t,r)} \partial_{a} u\Big) = 0,\]
d'o\`u 
\[ \int \Delta u_{\vert \oD(t,r) \setminus \{\eta_{t,r}\} } = - \sum_{a \in \Gamma_{\beta_{r}} \cap \oD(t,r)} \partial_{a} u.\]
Parmi les germes d'ar\^etes d'origine~$\eta_{t,r}$, un seul n'est pas contenu dans $\oD(t,r)$, celui repr\'esent\'e par~$\intff{\eta_{t,r},\eta_{t,R}}$. Notons-le~$a_{r,R}$. Puisque~$u$ est sous-harmonique, on a
\[ \sum_{a \in \Gamma_{\eta_{t,r}} \cap \oD(t,r)} \partial_{a} u + \partial_{a_{r,R}} u \ge 0.\]
La sous-harmonicit\'e de~$u$ entra\^ine \'egalement que la restriction de~$u$ \`a $\intff{\eta_{t,r},\eta_{t,R}}$ est convexe, donc sa pente au point~$\eta_{t,r}$ est inf\'erieure \`a la pente totale~:
\[\partial_{a_{r,R}} u \le \frac{u(\eta_{t,R})-u(\eta_{t,r})}{\log(R/r)}.\]
En combinant les diff\'erentes in\'egalit\'es, on obtient le r\'esultat voulu.
\end{proof}

Examinons maintenant l'effet d'un changement de la valeur absolue sur les notions introduites. 

Soit $\eps \in \R_{>0}$. Nous reprenons les notations de la section~\ref{sec:flotcorpsvalue}. Soit~$U$ un ouvert de~$\EP{1}{k}$. On dispose d'un ouvert $U_{\eps}$ de~$\EP{1}{k_{\eps}}$ et d'un isomorphisme
\[ \Phi_{\eps} \colon U \simtoo U_{\eps}.\]

On v\'erifie que l'application~$\Phi_{\eps}$ divise les distances par~$\eps$, ce qui est la clef pour comparer les notions de laplacien sur~$U$ et~$U_{\eps}$.

\begin{nota}
Pour toute fonction  $u \colon U \to \R \cup \{\pm\infty\}$, on d\'efinit une fonction~$u_{\eps}$ sur~$U_{\eps}$ par 
\[u_{\eps} := u \circ \Phi_{\eps}^{-1} \colon U_{\eps} \too  \R \cup \{\pm\infty\}.\]
\end{nota}

%Afin de comparer les laplaciens, nous devons \'etudier la fa\c con dont sont reli\'ees les m\'etriques sur~$X$ et~$X_{\eps}$. La clef se trouve dans l'\'enonc\'e suivant, qui d\'ecoule directement des d\'efinitions.  
%
%\begin{lemm}\label{lem:deps}
%Soient $r,s\in \R_{>0}$ avec $r<s$. Reprenons les notations de l'exemple~\ref{ex:squelettecouronne}. On a 
%\[\Phi_{\eps}(C_{k}(r,s)) = C_{k_{\eps}}(r^{1/\eps},s^{1/\eps}) \textrm{ et }
%\Phi_{\eps}(\Gamma_{C_{k}(r,s)}) = \Gamma_{C_{k_{\eps}}(r^{1/\eps},s^{1/\eps})}.\]
%Pour tous $x,y \in \Gamma_{C_{k}(r,s)}$, on a
%\[d(\Phi_{\eps}(x),\Phi_{\eps}(y)) = \eps \, d(x,y).\]
%\qed
%\end{lemm}
%
%Les r\'esultats de comparaison s'obtiennent maintenant sans peine.

\begin{prop}\label{prop:flotDeltaum}
Soit $u \colon U \to \R \cup \{\pm\infty\}$. Alors $u$ est lisse, sous-harmonique, etc. si, et seulement si, $u_{\eps}$ l'est. 

En outre, pour tout $u \in \DSH(U)$, on a $u_{\eps} \in \DSH(U_{\eps})$ et, pour tout $f\in \cC_{c}(U,\R)$, 
\[ \Delta u (f) = \eps\, \Delta u_{\eps} (f_{\eps}).\]
\qed
\end{prop}

\section{Syst\`emes dynamiques sur une droite projective relative}\label{sec:droiteprojectiverelative}

Dans cette section, nous d\'emontrons les th\'eor\`emes~\ref{th:continuiteintro} et~\ref{th:ACLintro} de l'introduction. La section~\ref{sec:ACL} contient la preuve du th\'eor\`eme~\ref{th:ACLintro}~: continuit\'e de familles de mesures constitu\'ees de mesures de Haar sur le cercle unit\'e (dans le cadre archim\'edien) et de mesures de Dirac au point de Gau\ss{} (dans le cadre ultram\'etrique). Dans la section~\ref{sec:affable}, nous introduisons les fonctions \emph{macologues}, des fonctions continues de nature assez simple, mais denses dans l'ensemble des fonctions continues. Nous les utilisons de fa\c con essentielle dans la section finale~\ref{sec:thcontinuite} pour  d\'emontrer le th\'eor\`eme~\ref{th:continuiteintro} sur la continuit\'e de familles de mesures d'\'equilibre associ\'ees \`a une famille de syst\`emes dynamiques.

\medbreak

Fixons le cadre. Soit~$\cA$ un bon anneau de Banach. Nous fixons des coordonn\'ees homog\`enes $T_{0},T_{1}$ sur~$\EP{1}{\cA}$ et posons $T:= T_{0}/T_{1}$. Posons $\infty := [1 : 0]$ et identifions $\EP{1}{\cA} \setminus \{\infty\}$ \`a la droite affine~$\E{1}{\cA}$ (munie de la coordonn\'ee~$T$).

Soit $Y$ un espace $\cA$-analytique. Soit $X := \EP{1}{\cA} \times_{\cA} Y$ et notons $\pi \colon X \to Y$ la seconde projection. Pour toute partie~$V$ de~$Y$, posons $X_{V} := \pi^{-1}(V)$. Pour tout point~$y$ de~$Y$, posons $X_{y} := \pi^{-1}(y)$.

Soit $X' := \E{1}{\cA} \times_{\cA} Y \subset X$. Pour toute partie~$V$ de~$Y$, posons $X'_{V} := X_{V} \cap X'$. Pour tout point~$y$ de~$Y$, posons $X'_{y} := X_{y} \cap X'$.

Soit $X'' := (\EP{1}{\cA} \setminus \{0\}) \times_{\cA} Y \subset X$. Pour toute partie~$V$ de~$Y$, posons $X''_{V} := X_{V} \cap X''$. Pour tout point~$y$ de~$Y$, posons $X''_{y} := X_{y} \cap X''$.

Introduisons finalement des notations pour les disques relatifs. Pour toute partie~$V$ de~$Y$ et tout $r\in \R_{>0}$, posons
\begin{align*}
D_{V}(r) &:= \{ x \in X'_{V} : \abs{T(x)} < r\},\\
\oD_{V}(r) &:= \{ x \in X'_{V} : \abs{T(x)} \le r\},\\
D^\infty_{V}(r) &:= \{ x \in X''_{V} : \abs{T(x)} > r^{-1}\},\\
\oD^\infty_{V}(r) &:= \{ x \in X''_{V} : \abs{T(x)} \ge r^{-1}\}.
\end{align*}
Pour tout $y \in Y$ et tout $r\in \R_{>0}$, on d\'efinit de m\^eme $D_{y}(r)$, $\oD_{y}(r)$, $D^\infty_{y}(r)$ et $\oD^\infty_{y}(r)$.

\subsection{Gau\ss, Haar, Chambert--Loir}\label{sec:ACL}

Dans \cite{ACLMesures}, Antoine Chambert--Loir esquisse une analogie entre la mesure de Haar sur un cercle du plan complexe et la mesure de Dirac en un point de Gauss d'une droite de Berkovich, et en tire profit pour d\'emontrer des r\'esultats d'\'equidistribution. Nous proposons ici une incarnation de cette id\'ee sous la forme d'un r\'esultat de continuit\'e pour des familles m\^elant mesures de Gau\ss{} et mesures de Haar. 

\begin{nota}
Soit $(k,\va)$ un corps valu\'e complet. Pour $z \in k$ et $r\in \R_{> 0}$, on pose
\[C_{k}(z,r) := \{x\in \E{1}{k} : |(T-z)(x)| = r \}.\] 
Pour $r \in \R_{>0}$, on pose
\[C_{k}(\infty,r) := C_{k}(0,r^{-1}).\]
\end{nota}

\begin{nota}
$\bullet$ Soit $\eps \in \intof{0,1}$ et consid\'erons le corps valu\'e $\C_{\eps} := (\C,\va_{\infty}^\eps)$. Pour $z\in \P^{1}(\C)$ et $r\in \R_{>0}$, on note $\chi_{\C_{\eps},z,r}$ la mesure de Haar de masse totale~1 sur le cercle~$C_{\C_{\eps}}(z,r)$. 

On note $\chi_{\C_{\eps},z,0}$ la mesure de Dirac~$\delta_{z}$ sur $\EP{1}{\C_{\eps}}$.

$\bullet$ Soit $\eps \in \intof{0,1}$ et consid\'erons le corps valu\'e $\R_{\eps} := (\R,\va_{\infty}^\eps)$. Soient $z\in \P^{1}(\R)$ et $r\in \R_{>0}$. Soit $z'$ l'unique ant\'ec\'edent de~$z$ par le morphisme canonique $\pr_{\C,\R} \colon \EP{1}{\C_{\eps}} \to \EP{1}{\R_{\eps}}$. On pose $\mu_{\R_{\eps},z,r} := (\pr_{\C,\R})_{*} \mu_{\C_{\eps},z',r}$. C'est une mesure positive de masse totale~1 sur $C_{\R_{\eps}}(z,r)$.

On note $\chi_{\R_{\eps},z,0}$ la mesure de Dirac~$\delta_{z}$ sur $\EP{1}{\R_{\eps}}$.

$\bullet$ Soit $(k,\va)$ un corps valu\'e ultram\'etrique complet. Soient $z\in k$ et $r\in \R_{> 0}$. Rappelons que nous avons d\'efini \`a la section~\ref{sec:corpsum} un point~$\eta_{z,r}$ de~$\E{1}{k}$ (qui n'est autre que l'unique point du bord de Shilov de la couronne~$C_{k}(z,r)$).

%Rappelons que l'on note~$\eta_{z,r}$ l'unique point du bord de Shilov de la couronne~$C_{k}(z,r)$ (\cf~section~\ref{sec:corpsum}). 

On pose $\chi_{k,z,r} := \delta_{k,\eta_{z,r}}$. 

On pose \'egalement $\chi_{k,\infty,r} := \chi_{k,0,r^{-1}}$. 

On note $\chi_{k,z,0}$ (resp. $\chi_{k,\infty,0}$)  la mesure de Dirac~$\delta_{z}$ (resp. $\delta_{\infty}$) sur $\EP{1}{k}$.

\medbreak

On sous-entendra parfois le corps~$k$ si cela ne peut pr\^eter \`a confusion.
\end{nota}

\medbreak

Les deux r\'esultats suivants d\'ecoulent directement des d\'efinitions. 

\begin{lemm}\label{lem:extensionchi}
Soit $(k,\va)$ un corps valu\'e complet. Soit $(K,\va)$ une extension valu\'ee compl\`ete de~$(k,\va)$ et notons $\pr_{K,k} \colon \EP{1}{K} \to \EP{1}{k}$ le morphisme canonique. Pour tout $z \in \E{1}{K}$ et tout $r\in \R_{\ge 0}$, on a 
\[(\pr_{K,k})_{*}\chi_{K,z,r} = \chi_{k,\pr_{K,k}(z),r}.\]
\qed
\end{lemm}

\begin{lemm}\label{lem:Phichi}
Soit $(k,\va)$ un corps valu\'e complet. Soit $\eps \in \intof{0,1}$. Notons~$k_{\eps}$ le corps~$k$ muni de la valeur absolue~$\va^\eps$. Pour tous $z \in \P^1(k)$ et $r\in \R_{\ge 0}$, on a
\[(\Phi_{\eps})_{*} \chi_{k,z,r} = \chi_{k_{\eps},z,r^\eps},\]
avec les notations de la section~\ref{sec:flotcorpsvalue}.
\qed
\end{lemm}

Remarquons qu'\`a l'aide du laplacien, on peut d\'efinir les mesures $\chi_{k,z,r}$ de fa\c con uniforme.

\begin{lemm}
Soit $(k,\va)$ un corps valu\'e complet. Pour tout $z\in k$ et $r\in \R_{>0}$, on a
\[\Delta \max\big(-\log(\abs{T-z}), \log(r)\big)= \chi_{k,z,r} - \delta_{k,z} \textrm{ dans } \Mes(\EP{1}{k}).\]
\qed
\end{lemm}

Nous allons d\'emontrer la continuit\'e de familles de mesures de la forme~$\chi_{k,z,r}$.

\begin{nota}
Soient $c\in \cO(Y)$ et $\varrho \colon Y \to \R_{\ge 0}$ une fonction continue. On note $\chi_{c,\varrho}$ la $\pi$-famille de mesures d\'efinie par 
\[\forall y\in Y,\ \chi_{c,\varrho,y} := \chi_{\cH(y),c(y),\varrho(y)} \textrm{ dans } \Mes^1(X_{y}).\]
\end{nota}

%Le lemme technique qui suit nous sera utile.

\begin{lemm}\label{lem:intD1D2}
Soit $\eps \in \intof{0,1}$. Soient $z_{1},z_{2} \in \P^1(\C)$ et $r_{1},r_{2}\in \R_{\ge 0}$. Alors, pour toute $f\in \cC_{c}(\EP{1}{\C_{\eps}},\R)$ \`a support contenu dans $\oD(z_{1},r_{1})$, on a
\[ \int f \diff \chi_{\C_{\eps},z_{2},r_{2}} \le \Big(\frac{r_{1}}{r_{2}}\Big)^{1/\eps} \,\norm{f}_{\infty}.\]
\end{lemm}
\begin{proof}
Le lemme~\ref{lem:Phichi} permet de se ramener au cas o\`u $\eps = 1$. L'intersection du disque $\oD(z_{1},r_{1})$ et du cercle $C(z_{2},r_{2})$ est un arc de longueur inf\'erieure \`a $2\pi r_{1}$. On en d\'eduit que 
\[ \int f \diff \chi_{z_{2},r_{2}} \le \frac{2\pi r_{1}}{2\pi r_{2}} \, \norm{f}_{\infty} = \frac{r_{1}}{r_{2}} \, \norm{f}_{\infty}.\]
\end{proof}

\begin{theo}\label{th:chicontinue}
Soient $c\in \cO(Y)$ et $\varrho \colon Y \to \R_{\ge 0}$ une fonction continue. Alors la famille de mesures $\chi_{c,\varrho}$ est continue.
\end{theo}
\begin{proof}
Il suffit de montrer que la famille est continue au voisinage de tout point de~$Y$. Soit $y\in Y$. D'apr\`es le lemme~\ref{lem:localisationBanach}, on peut supposer que $Y = \cM(\cA)$ et que $c\in \cA$. Quitte \`a effectuer le changement de variable $T \mapsto T -c$, on peut supposer que, pour tout  $y\in Y$, on a $c(y)=0$.

\medbreak

$\bullet$ Supposons que $y$ est ultram\'etrique.

Soit $\kF$ un filtre sur~$Y$ convergeant vers~$y$. Il suffit de montrer que son image~$\kF_{\chi}$ par $z \mapsto \chi_{0,\varrho(z)}$ converge vers $\chi_{0,\varrho(y)}$. Par compacit\'e de l'espace des mesures de probabilit\'e, %(\cf~proposition~\ref{prop:mesuresprobacompact}), \cite[III, \S 1, \no 9, cor.~3 de prop.~15]{BourbakiI14}} 
il suffit de montrer que toute valeur d'adh\'erence de~$\kF_{\chi}$ co\"incide avec $ \chi_{0,\varrho(y)}$. Soit~$\mu$ une telle valeur d'adh\'erence.

Pour montrer que $\mu = \chi_{0,\varrho(y)}$, il suffit de montrer que son support est r\'eduit \`a~$\{\eta_{0,\varrho(y)}\}$. Puisque, pour tout $z \in Y$, le support de~$\chi_{0,\varrho(z)}$ est contenu dans~$X_{z}$, on a $\Supp(\mu) \subset X_{y}$. Soit $x \in X_{y} \setminus  \{\eta_{0,\varrho(y)}\}$. Nous allons distinguer deux cas.

$\bullet \bullet$ Supposons que $\abs{T(x)} > \varrho(y)$.

Soit $s \in \intoo{\varrho(y),\abs{T(x)}}$. Par continuit\'e de~$\varrho$, quitte \`a remplacer~$Y$ par un voisinage de~$y$, on peut supposer que, pour tout $z\in Y$, on $\varrho(z) < s$. Dans ce cas, $D^\infty_{Y}(s^{-1})$ est un voisinage de~$x$ et, pour tout $z\in Y$, on a
\[ \Supp(\chi_{0,\varrho(z),z}) \cap D^\infty_{Y}(s^{-1}) = \emptyset.\]
On en d\'eduit que, pour tout $f \in \cC(X,\R)$ \`a support dans $D^\infty_{Y}(s^{-1})$, on a
\[ \int f \diff \mu = 0,\]
et donc que $x \notin \Supp(\mu)$.

$\bullet \bullet$ Supposons que $\abs{T(x)} \le \varrho(y)$.

Soit $\alpha$ un point rigide de $X_{y} \simeq \EP{1}{\cH(y)}$ appartenant \`a la composante connexe de $X_{y} \setminus \{\eta_{0,\varrho(y)}\}$ contenant~$x$.\footnote{Si le point~$\alpha$ peut \^etre choisi de fa\c con \`a appartenir \`a l'image de~$\cO(Y)$ dans~$\cH(y)$, alors on peut se ramener au cas pr\'ec\'edent par un changement de variables. En g\'en\'eral, cependant, un tel choix n'est pas possible.} Remarquons que $\alpha \in \oD_{y}(0,\varrho(y))$. Notons $P  \in \cH(y)[T]$ le polyn\^ome minimal de~$\alpha$ et $d\in \N_{\ge 1}$ son degr\'e. On a alors
\[ \abs{P(x)} < \abs{P(\eta_{0,\varrho(y)})} = \varrho(y)^d.\]
\JP{Donner des pr\'ecisions ? lemme avant ?}

Gr\^ace aux propri\'et\'es ultram\'etriques du corps~$\cH(y)$, la relation pr\'ec\'edente reste satisfaite apr\`es une petite perturbation des coefficients de~$P$. On peut donc supposer que~$P$ appartient \`a~$\kappa(y)[T]$, et m\^eme \`a~$\cO(Y)[T]$, quitte \`a remplacer~$Y$ par un voisinage de~$y$.

Soient $s, t\in \intoo{\abs{P(x)},\varrho(y)^d}$ avec $s<t$. Quitte \`a restreindre~$Y$, on peut supposer que, pour tout $z\in Y$, on a $t \le \varrho(z)^d$. Posons 
\[ U := \{ u \in X : \abs{P(u)} < s\}.\]
C'est un voisinage ouvert de~$x$. Soit $f\in \cC_{c}(X,\R)$ \`a support dans~$U$. Soit $z\in Y$.

Supposons que $z$ est ultram\'etrique. Puisque~$P$ est unitaire de degr\'e~$d$, on a 
\[ \abs{P(\eta_{0,\varrho(z)})} \ge \varrho(z)^d \ge t > s,\]
donc $\eta_{0,\varrho(z)} \notin U$ et 
\[ \int f \diff \chi_{0,\varrho(z),z} = 0.\]

Supposons que $z$ est archim\'edien. Il existe $\eps(z) \in \intof{0,1}$ tel que le corps r\'esiduel~$\cH(z)$ soit isomorphe \`a~$\R$ ou~$\C$ muni de la valeur absolue~$\va_{\infty}^{\eps(z)}$. Notons $K_{z}$ le corps valu\'e~$\C$ muni de la valeur absolue~$\va_{\infty}^{\eps(z)}$. C'est une extension de~$\cH(z)$. Notons $\pr_{z} \colon \EP{1}{K_{z}} \to  \EP{1}{\cH(z)}$ le morphisme naturel. Notons $P_{z}(T)$ l'image de~$P(T)$ dans $K_{z}[T]$ et $\alpha_{1},\dotsc,\alpha_{d} \in K_{z}$ ses racines. 

Soit $u \in \pr_{z}^{-1}(U \cap X_{z})$. Soit $j_{u} \in \cn{1}{d}$ tel que 
$\abs{u-\alpha_{j_{u}}} = \min_{1\le i\le d} (  \abs{u-\alpha_{i}} )$. 
On a alors $\abs{u-\alpha_{j_{u}}}^d \le \abs{P(u)} < s$.  On en d\'eduit que 
\[ \pr_{z}^{-1}(U \cap X_{z}) \subset \bigcup_{1\le i\le d}  D_{K_{z}}(\alpha_{i},s^{1/d}).\]
D'apr\`es les lemmes~\ref{lem:extensionchi} et~\ref{lem:intD1D2}, on a donc 
\[ \int f \diff \chi_{0,\varrho(z),z} \le d \Big(\frac{s^{1/d}}{\varrho(z)}\Big)^{\frac1{\eps(z)}}\, \|f\|_{\infty} \le d \Big(\frac{s}{t}\Big)^{\frac1{d\eps(z)}}\, \|f\|_{\infty}.\]

D'apr\`es la remarque~\ref{rem:calculepsilonarc}, on a $\eps(z) = \frac{\log(\abs{2(z)})}{\log(2)} \in \intof{0,1}$. Puisque~$y$ est ultram\'etrique, cette quantit\'e tend vers~0 lorsque~$z$ tend vers~$y$ selon~$\kF$. On en d\'eduit que
\[ \int f \diff \mu = 0,\]
et donc que $x \notin \Supp(\mu)$.

$\bullet$ Supposons que $y$ est archim\'edien.

Quitte \`a remplacer~$Y$ par un voisinage de~$y$, on peut supposer que tout point de~$Y$ est archim\'edien. L'injection canonique $j_{\Z} \colon \Z \hookrightarrow \cO(Y)$ s'\'etend alors en une injection $j_{\Q} \colon\Q \hookrightarrow \cO(Y)$ puis $j_{\R} \colon\R \hookrightarrow \cO(Y)$. 

D'apr\`es \cite[lemme~5.2.2]{A1Z}, on peut munir l'anneau $\cA' := \cA[t]/(t^2+1)$ d'une norme de $\cA$-alg\`ebre de Banach. Consid\'erons le morphisme canonique $\pr_{\cA',\cA} \colon \cM(\cA') \to \cM(\cA)$ et posons $\varrho' := \varrho \circ \pr_{\cA',\cA}$. D'apr\`es le lemme~\ref{lem:extensionchi}, il suffit de montrer que la famille de mesures $\chi_{0,\varrho'}$ sur~$\EP{1}{\cA'}$ est continue en un point de~$\pr_{\cA',\cA}^{-1}(y)$. On peut donc remplacer~$\cA$ par~$\cA'$. D\'esormais, pour $y\in Y$, on a $\cH(y) \simeq \C$ et l'injection $j_{\R}$ se prolonge donc en une injection $j_{\C} \colon\C \hookrightarrow \cO(Y)$ envoyant~$i$ sur~$t$.

Soit $\kF$ un filtre sur~$Y$ convergeant vers~$y$. Par le m\^eme argument que dans le cas ultram\'etrique, il suffit de montrer que toute valeur d'adh\'erence de l'image~$\kF$ par $z \mapsto \chi_{0,\varrho(z)}$ co\"incide avec $ \chi_{0,\varrho(y)}$. Soit~$\mu$ une telle valeur d'adh\'erence.

Pour tout $z \in Y$, le support de~$\chi_{0,\varrho(z)}$ est contenu dans $C_{\cH(z)}(0,\varrho(z)) \subset X_{z}$. On en d\'eduit que $\Supp(\mu) \subset C_{\cH(y)}(0,\varrho(y)) \subset X_{y}$. En outre, $\mu$ est positive et de masse totale~1.

Pour tout $\alpha\in \C$, on note $m_{Y,\alpha} \colon \cO(Y) \to \cO(Y)$ l'application induite par la multiplication par~$\alpha$ \textit{via}~$j_{\C}$ et $m_{X,\alpha} \colon \cO(X) \to \cO(X)$ l'application induite par~$m_{X,\alpha}$ en tirant en arri\`ere par~$\pi$. 

Pour tout $z \in Y$, la mesure~$\chi_{0,\varrho(z)}$ est invariante par rotation, au sens o\`u, pour tout $\alpha\in \C$ avec $\abs{\alpha}=1$ et tout $f\in \cC_{c}(X,\R)$, on a
\[ \int f \circ m_{X,\alpha} \diff \chi_{0,\varrho(z)} =  \int f \diff \chi_{0,\varrho(z)}.\]
On en d\'eduit que~$\mu$ est \'egalement invariante par rotation. On conclut que $\mu = \chi_{0,\varrho(y)}$ en utilisant la caract\'erisation de la mesure de Haar sur le cercle $C_{\cH(y)}(0,\varrho(y))$. 
%(Pour un raisonnement plus d\'etaill\'e, on peut faire appel \`a la notion de mesure induite sur un compact et ses propri\'et\'es, \cf~\cite[IV, \S 5, \no 7, d\'ef.~4]{BourbakiI14} et~\cite[V, \S 4, \no 4, th.~2]{BourbakiI14}.)
\end{proof}

\subsection{Fonctions macologues}\label{sec:affable}

Afin de d\'emontrer la continuit\'e de famille de mesures, il est commode de disposer de familles de fonctions continues maniables qui soient denses dans l'ensemble des fonctions continues. Dans ce but, nous d\'efinissons ici les fonctions macologues, inspir\'ees des fonctions mod\`eles introduites par Charles Favre dans \cite[section~2]{FavreEndomorphisms}. Elles appartiennent \`a la classe des fonctions dites \og PL\fg{} dans la litt\'erature.

%Elles appartiennent \`a la classe des fonctions connues dans la litt\'erature sous la d\'enomination de \og PL\fg.

Signalons que L\'eonard Pille--Schneider propose une version plus aboutie de notre notion dans \cite[section~2.3]{PilleSchneiderPluripotentialHybrid}\footnote{dont la premi\`ere version est post\'erieure \`a celle de notre manuscrit} sous le nom de fonctions de Fubini-Study tropicales.

%Signalons que, dans \cite[section~2.3]{PilleSchneiderPluripotentialHybrid}, dont la premi\`ere version est post\'erieure \`a celle de notre manuscrit, L\'eonard Pille--Schneider introduit et \'etudie les fonctions de Fubini-Study tropicales (et leurs diff\'erences), obtenant ainsi une version plus aboutie de notre notion.

%Elles sont \'egalement apparent\'ees aux fonctions de Fubini-Study tropicales de L\'eonard Pille--Schneider, \cf~\cite[section~2.3]{PilleSchneiderPluripotentialHybrid}\footnote{dont la premi\`ere version est post\'erieure \`a celle de notre manuscrit} ou, plus g\'en\'eralement, aux fonctions dites~PL dans la litt\'erature.

\begin{defi}\label{def:affablebasique}
Soit $U$ un ouvert de~$X$. Une fonction $f \colon U \to \R$ est dite \emph{positivement macologue basique} s'il existe $n \in \N_{\ge 1}$,  $p_{1},q_{1},\dotsc,p_{n},q_{n} \in \Q$, avec $q_{1},\dotsc,q_{n}\ge 0$, et $g_{1},\dotsc,g_{n} \in \cO(U)$ tels que
%\[f = \max(p_{1} + q_{1}\log(\abs{g_{1}}), \dotsc, p_{n} + q_{n}\log(\abs{g_{n}})).\]
\[f = \max_{1\le i \le n}(p_{i} + q_{i}\log(\abs{g_{i}})).\]
Une fonction $f \colon U \to \R$ est dite \emph{macologue basique} si elle est diff\'erence de deux fonctions positivement macologues basiques.

On note~$\MCL_{pb}(U,\R)$ (resp. $\MCL_{b}(U,\R)$) l'ensemble des fonctions positivement macologues basiques (resp. macologues basiques) sur~$U$.
\end{defi}

%\begin{defi}\label{def:affablebasique}
%Soit $U$ un ouvert de~$X$. Une fonction $f \colon U \to \R$ est dite \emph{positivement affable basique} s'il existe un entier $n \in \N_{\ge 1}$, un \'el\'ement $q_{0} \in \Q\cup \{-\infty\}$, des nombres rationnels $q_{1},\dotsc,q_{n} \in \Q_{\ge 0}$ et des fonctions $g_{1},\dotsc,g_{n} \in \cO(U)$ tels que
%\[f = \max(q_{0}, q_{1}\log(\abs{g_{1}}), \dotsc, q_{n}\log(\abs{g_{n}})).\]
%Une fonction $f \colon U \to \R$ est dite \emph{affable basique} si elle est diff\'erence de deux fonctions positivement affables basiques.
%
%On note~$\cA_{pb}(U,\R)$ (resp. $\cA_{b}(U,\R)$) l'ensemble des fonctions positivement affables basiques (resp. affables basiques) sur~$U$.
%\end{defi}

\begin{lemm}\label{lem:proprietesAb}
Soient $U$ un ouvert de~$X$. 
\begin{enumerate}[i)]
\item Pour tout $f \in \MCL_{pb}(U,\R)$ et tout $y\in Y$, 
$f_{\vert U \cap X_{y}}$ est sous-harmonique.
\item Pour tout $f \in \MCL_{b}(U,\R)$ et tout $y\in Y$, 
$f_{\vert U \cap X_{y}}$ est diff\'erence locale de fonctions sous-harmoniques continues.
\item L'ensemble $\MCL_{pb}(U,\R)$ est stable par les op\'erations binaires~$+$ et~$\max$. 
\item L'ensemble $\MCL_{b}(U,\R)$ est un sous-$\Q$-espace vectoriel de~$\cC(U,\R)$ stable par les op\'erations binaires~$\min$ et~$\max$.
\end{enumerate}
\end{lemm}
\begin{proof}
Les points i) et ii) d\'ecoulent directement des d\'efinitions.

iii) Soient $p,q,p',q' \in \Q$, avec $q,q'\ge 0$, et $g,g' \in \cO(U)$. \'Ecrivons $q=\frac a b$ et $q' = \frac{a'}{b'}$, avec $a,a'\in \N$ et $b,b' \in \N_{\ge 1}$. On a 
\begin{align*} 
p+q\log(\abs{g}) + p'+q'\log(\abs{g'}) &= p+p' + \frac1b \log(\abs{g^a}) +  \frac1{b'} \log(\abs{{g'}^{a'}})\\ 
& = p+p' +  \frac1{bb'} \log(\abs{g^{ab'}{g'}^{a'b}}).
\end{align*}
La stabilit\'e de $\MCL_{pb}(U,\R)$ par~$+$ s'en d\'eduit. La stabilit\'e par~$\max$ est \'evidente.

iv) L'ensemble $\MCL_{b}(U,\R)$ est stable par multiplication par un \'el\'ement de~$\Q$. D'apr\`es~ii), il est stable par somme et on en d\'eduit qu'il forme un sous-$\Q$-espace vectoriel de~$\cC(U,\R)$. 

Montrons la stabilit\'e par~$\max$. Soient $u,v \in \MCL_{b}(U,\R)$. Il existe $u^+,u^-, v^+,v^-\in \MCL_{pb}(U,\R)$ telles que $u = u^+-u^-$ et $v=v^+-v^-$. On a
\begin{align*}
\max(u,v) &= \max(u^+ - u^-, v^+ - v^-)\\ 
& =  \max(u^+ + v^-, v^+ + u^-) - (u^- + v^-),
\end{align*}
et le r\'esultat d\'ecoule maintenant de~iii).

La stabilit\'e par~$\min$ se d\'eduit de la stabilit\'e par~$\max$ en passant \`a l'oppos\'e.
\end{proof}

\begin{defi}\label{def:affable}
Soit~$V$ un ouvert de~$Y$. Une fonction $f \colon X_{V} \to \R$ est dite \emph{macologue} si ses restrictions \`a~$X'_{V}$ et~$X''_{V}$ sont macologues basiques. On note $\MCL(X_{V},\R)$ l'ensemble des fonctions macologues sur~$X_{V}$.
\end{defi}

%\begin{defi}\label{def:affable}
%Soit~$V$ un ouvert de~$Y$. Une fonction $f \colon X_{V} \to \R$ est dite \emph{affable} si ses restrictions \`a~$X'_{V}$ et~$X''_{V}$ sont affables basiques. On note $\cA(X_{V},\R)$ l'ensemble des fonctions affables sur~$X_{V}$.
%\end{defi}

\begin{lemm}\label{lem:proprietesA}
Soit~$V$ un ouvert de~$Y$. L'ensemble $\MCL(X_{V},\R)$ est un sous-$\Q$-espace vectoriel de~$\cC(X_{V},\R)$ stable par les op\'erations binaires~$\min$ et~$\max$.
\end{lemm}
\begin{proof}
Cela d\'ecoule directement du point~iv) du lemme~\ref{lem:proprietesAb}.
\end{proof}

%\begin{lemm}\label{lem:proprietesA}
%Soit~$V$ un ouvert de~$Y$. L'ensemble $\cA(X_{V},\R)$ est un sous-$\Q$-espace vectoriel de~$\cC(X_{V},\R)$ stable par les op\'erations binaires~$\min$ et~$\max$.
%\end{lemm}
%\begin{proof}
%Cela d\'ecoule directement du point~iv) du lemme~\ref{lem:proprietesAb}.
%\end{proof}

\begin{prop}\label{prop:borneaffable}
Soient~$V$ un ouvert de~$Y$ et $f \colon X_{V}\to \R$ une fonction macologue. Alors, la fonction
\[ y \in V \mapstoo \int \diff\abs{\Delta f_{\vert y}} \in \R_{\ge 0}\]
est born\'ee sur tout compact.
\end{prop}
\begin{proof}
Par d\'efinition, il existe $f_{0}^+, f_{0}^- \in \MCL_{pb}(X'_{V},\R)$ et $f_{\infty}^+, f_{\infty}^- \in \MCL_{pb}(X''_{V},\R)$ tels que $f_{\vert X'_{V}} = f_{0}^+-f_{0}^-$ et $f_{\vert X''_{V}} = f_{\infty}^+-f_{\infty}^-$. 

Soit~$K$ une partie compacte de~$V$ et soit $y\in K$. Soient $h_{0},h_{\infty} \in \cC(X_{y},\intff{0,1})$ telles que $\Supp(h_{0}) \subset D_{y}(2)$, $\Supp(h_{\infty}) \subset D^\infty_{y}(2)$ et
\[\forall x\in X_{y},\ h_{0}(x) + h_{\infty}(x) = 1.\] 
D'apr\`es la proposition~\ref{prop:majorationSH}, avec les adaptations des sections~\ref{sec:laplacienCeps} et~\ref{sec:laplacienReps}, si $y$ est archim\'edien, ou la proposition~\ref{prop:majorationSHdisqueum}, si $y$ est ultram\'etrique, on a 
\begin{align*} 
\int h_{0} \diff \abs{\Delta f_{\vert y}} & = \int h_{0} \diff \abs{\Delta (f^+_{0,\vert y}-f^-_{0,\vert y})}\\ 
& \le \int h_{0} \diff\Delta f^+_{0,\vert y} + \int h_{0}  \diff \Delta f^-_{0,\vert y}\\
& \le \frac{2}{\log(2)}\, \big(\norm{f^+_{0}}_{\oD_{y}(4)} +  \norm{f^-_{0}}_{\oD_{y}(4)}\big)\\
& \le \frac{2}{\log(2)}\, \big(\norm{f^+_{0}}_{\oD_{K}(4)} + \norm{f^-_{0}}_{\oD_{K}(4)}\big).
\end{align*}
et, par un raisonnement similaire,
\[\int h_{\infty} \diff \abs{\Delta f_{\vert y}} \le  \frac{2}{\log(2)}\, \big(\norm{f^+_{\infty}}_{\oD^\infty_{K}(4)} +  \norm{f^-_{\infty}}_{\oD^\infty_{K}(4)}\big).\]
Le r\'esultat s'ensuit.
\end{proof}

D\'emontrons maintenant un r\'esultat de densit\'e, en suivant la strat\'egie de la preuve de~\cite[theorem~2.12]{FavreEndomorphisms} (\cf~\'egalement \cite[theorem~2.16]{PilleSchneiderPluripotentialHybrid}).

\begin{theo}\label{th:densiteaffable}
Soit~$V$ un ouvert de~$Y$ et supposons 
%qu'il existe une immersion de~$V$ dans un espace affine analytique sur~$\cA$.
que
\begin{enumerate}[i)]
\item il existe une immersion de~$V$ dans un espace affine analytique sur~$\cA$ ;
\item pour tout $y\in Y$, il existe $\alpha_{y} \in \cO(V)$ tel que $\abs{\alpha_{y}(y)}>1$.
\end{enumerate}
Alors, $\MCL(X_{V},\R)$ est dense dans~$\cC(X_{V},\R)$ pour la topologie de la convergence uniforme sur les compacts.
\end{theo}
\begin{proof}
Il suffit de d\'emontrer que, pour toute partie compacte~$K$ de~$X_{V}$, $\MCL(X_{V},\R)$ est dense dans~$\cC(K,\R)$. Supposons avoir d\'emontr\'e que $\MCL(X_{V},\R)$ s\'epare les points de~$X_{V}$. Soit $\varphi \in \cC(K,\R)$ et soient $x_{1} \ne x_{2} \in K$. Par hypoth\`ese, il existe $f \in \MCL(X_{V},\R)$ telle que $f(x_{1})\ne f(x_{2})$. Pour tout $\eps \in \R_{>0}$, on peut trouver $a,b \in \Q$ tels que 
\[ \abs{af(x_{1})+b - \varphi(x_{1})} < \eps \textrm{ et } \abs{af(x_{2})+b - \varphi(x_{2})} < \eps.\]
D'apr\`es le lemme~\ref{lem:proprietesA}, $af+b \in \MCL(X_{V},\R)$. La version treillis du th\'eor\`eme de Stone-Weierstra\ss{} permet alors de conclure que $\MCL(X_{V},\R)$ est dense dans~$\cC(K,\R)$.

\medbreak

Il suffit donc de d\'emontrer que $\MCL(X_{V},\R)$ s\'epare les points de~$X_{V}$. Soient $x_{1} \ne x_{2} \in X_{V}$. Distinguons deux cas.

$\bullet$ Supposons que $\pi(x_{1}) \ne \pi(x_{2})$.

Puisque $V$ peut se plonger dans un espace affine analytique sur~$\cA$, il existe $g \in \cO(V)$ tel que $\abs{g(\pi(x_{1}))} \ne \abs{g(\pi(x_{2}))}$. On peut supposer que $\abs{g(\pi(x_{1}))} < \abs{g(\pi(x_{2}))}$. Soit $p \in \Q \cap \intoo{\log(\abs{g(\pi(x_{1}))}),\log(\abs{g(\pi(x_{2}))})}$. Alors la fonction
\[ f := \max(p, \log(\abs{\pi^*g}))\]
est une fonction macologue sur~$X_{V}$ qui s\'epare~$x_{1}$ et~$x_{2}$.

$\bullet$ Supposons que $\pi(x_{1}) = \pi(x_{2})$. Notons $y$ ce point.

$\bullet \bullet$ Supposons que $x_{1}$ et $x_{2}$ appartiennent \`a~$\oD_{V}(1)$.

Puisque $V$ peut se plonger dans un espace affine analytique sur~$\cA$, il existe $g \in \cO(V)[T]$ tel que $\abs{g(x_{1})} \ne \abs{g(x_{2})}$. Notons~$d$ le degr\'e de~$g$. Quitte \`a \'echanger $x_{1}$ et~$x_{2}$, on peut supposer que $\abs{g(x_{1})} > \abs{g(x_{2})}$ et, quitte \`a multiplier~$g$ par une puissance suffisamment grande de~$\alpha_{y}$, on peut supposer que $\abs{g(x_{1})} > 1$. 

Posons $G := T_1^d g(T_{0}/T_{1}) \in \cO(V)[T_{0},T_{1}]$. C'est un polyn\^ome homog\`ene de degr\'e~$d$. Posons 
\[ F := \max( d \log(\abs{T_{0}}) , d \log(\abs{T_{1}}), \log(\abs{G})) - \max(d \log(\abs{T_{0}}), d \log(\abs{T_{1}})).\]
C'est une fonction bien d\'efinie sur~$X_{V}$, macologue par construction, et qui co\"incide avec $\max(0,\log(\abs{g}))$ sur~$\oD_{V}(1)$. Elle s\'epare donc les points $x_{1}$ et~$x_{2}$.

$\bullet \bullet$ Supposons que $x_{1}$ et $x_{2}$ appartiennent \`a~$\oD^\infty_{V}(1)$.

On se ram\`ene au cas pr\'ec\'edent en \'echangeant les coordonn\'ees projectives~$T_{0}$ et~$T_{1}$.

$\bullet \bullet$. Supposons que les points $x_{1}$ et $x_{2}$ appartiennent l'un \`a $D_{V}(1)$ et l'autre \`a $D^\infty_{V}(1)$.

La fonction 
\[ F := \max(\log(\abs{T_{0}}) , \log(\abs{T_{1}}), \log(\abs{T_{0}+\alpha_{y}T_{1}})) - \max(\log(\abs{T_{0}}), \log(\abs{T_{1}}))\]
est alors une fonction macologue sur~$X_{V}$ qui s\'epare~$x_{1}$ et~$x_{2}$.
\end{proof}

\subsection{Mesures d'\'equilibre}\label{sec:thcontinuite}

L'objet de cette section est d'\'etudier un syst\`eme dynamique polaris\'e (\cf~section~\ref{sec:systemedynamiquepolarise}) dans le cadre de la droite projective relative que nous utilisons d\'esormais, et de d\'emontrer la continuit\'e de la famille de mesures d'\'equilibre associ\'ees.

\begin{nota}
Soit $d\in\Z$. Le fibr\'e en droites~$\cO(d)$ sur~$\P^1_{\cA}$ induit un fibr\'e en droite sur~$\EP{1}{\cA}$ par analytification, puis un fibr\'e en droites sur~$X$ en tirant en arri\`ere par la projection $X \to \EP{1}{\cA}$. Notons~$\cO_{X}(d)$ ce fibr\'e.
\end{nota}

Introduisons maintenant un syst\`eme dynamique polaris\'e que nous utiliserons dans toute cette section. Nous supposerons que nous disposons d'un morphime d'espaces analytiques surjectif, fini et plat $\varphi \colon X \to X$ au-dessus de~$Y$, de degr\'e $d\ge 2$, et d'un isomorphisme 
\[\theta \colon \varphi^* \cO_{X}(1) \simtoo \cO_{X}(d).\]
En plus des hypoth\`eses formul\'ees au d\'ebut de la section~\ref{sec:droiteprojectiverelative}, nous supposerons que l'espace~$Y$ est s\'epar\'e (ce qui entra\^ine que l'espace~$X$ l'est \'egalement).

\begin{defi}
On appelle \emph{m\'etrique standard} sur le fibr\'e~$\cO_{X}(1)$ la m\'etrique dont la restriction \`a toute fibre $X_{y} \simeq \EP{1}{\cH(y)}$, pour $y\in Y$, est la m\'etrique standard sur~$\EP{1}{\cH(y)}$, comme d\'efinie dans  l'exemple~\ref{ex:metriquestandard}. On la note encore~$\nm_{\st}$. C'est une m\'etrique continue. 
\end{defi}

Remarquons que, pour tout $x\in X$, on a 
\[
- \log(\norm{T_{0}}_{\st,x}) = \max( - \log(|T(x)|), 0).
\]
Par cons\'equent, pour tout $y\in Y$, la fonction $x\mapsto - \log(\norm{T_{0}}_{\st,x})$ est sous-harmonique continue sur $X_{y} \setminus \{0\}$ et on a
\[\Delta (- \log(\norm{T_{0}}_{\st,\vert X_{y}})) = \chi_{0,1} - \delta_{0} \textrm{ dans } \Mes(X_{y}).\]

Plus g\'en\'eralement, on a le r\'esultat suivant.

\begin{lemm}\label{lem:Deltalogs}
Soient $y\in Y$ et $s \in  \Gamma(X_{y},\cO_{X_{y}}(1))$.
La fonction $x\mapsto - \log(\norm{s}_{\st,x})$ est sous-harmonique continue sur $X_{y} \setminus \div(s)$ et on a
\[\Delta( - \log(\norm{s}_{\st,\vert X_{y}})) =   \chi_{0,1} - \delta_{\div(s)} \textrm{ dans } \Mes(X_{y}).\]
\qed
\end{lemm}

\'Etudions maintenant le comportement de la norme standard vis-\`a-vis du flot. Nous utilisons ici la terminologie et les notations de la section~\ref{sec:flotanneau}. Notons que, si $V$ est un ouvert flottant de~$Y$, alors $X_{V}$ est un ouvert flottant de~$X$.

\begin{lemm}\label{lem:nmstflottante}
Supposons que $\cA$ est flottant. Soit~$V$ un ouvert flottant de~$Y$. Alors $X_{V}$ est un ouvert flottant de~$X$ et, pour toute $s\in \Gamma(X_{V},\cO_{X}(1))$, la fonction
\[x \in X_{V} \mapsto \norm{s}_{\st,x} \in \R\]
est flottante.
\end{lemm}
\begin{proof}
Le r\'esultat d\'ecoule directement de l'expression explicite de la m\'etrique standard pr\'esent\'ee dans l'exemple~\ref{ex:metriquestandard}.
\end{proof}

Par les arguments de la section~\ref{sec:systemedynamiquepolarise}, on peut d\'efinir une application
\[\Phi \colon \Met(\cO_{X}(1)) \too \Met(\cO_{X}(1))\]
et une m\'etrique~$\nm_{\varphi}$ sur~$\cO_{X}(1)$ telle que
\[\varphi^*\nm_{\varphi} = \nm_{\varphi}^{\otimes d} \circ \theta \textrm{ dans } \Met(\varphi^*\cO_{X}(1)).\]
Pour tout $n\in \N$, posons $\nm_{n} := \Phi^n(\nm_{\st})$. 

\begin{lemm}\label{lem:nmphiflottante}
Supposons que $\cA$ est flottant. Soit~$V$ un ouvert flottant de~$Y$ et supposons que l'endomorphisme $\varphi_{\vert X_{V}}$ de~$X_{V}$ est flottant. Alors, pour toute $s\in \Gamma(X_{V},\cO_{X}(1))$, les fonctions 
\[x \in X_{V} \mapstoo \norm{s}_{n,x} \in \R, \textrm{ pour } n\in\N,\]
et la fonction
\[x \in X_{V} \mapstoo \norm{s}_{\varphi,x} \in \R\]
sont flottantes.
\end{lemm}
\begin{proof}
Le r\'esultat pour $\nm_{n}$ d\'ecoule du lemme~\ref{lem:nmstflottante} et des d\'efinitions. Le r\'esultat pour~$\nm_{\varphi}$ s'en d\'eduit par passage \`a la limite.
\end{proof}

Notons~$\mu_{0}$ la $\pi$-famille de mesures d\'efinie comme suit~: pour tout $y \in Y$,
\[\mu_{0,y} := \chi_{0,1} \textrm{ dans } \Mes(X_{y}).\]
Elle est continue, d'apr\`es le th\'eor\`eme~\ref{th:chicontinue}.

Pour tout $n\in \N_{\ge 1}$, posons $\mu_{n} := \frac1{d^n}\, (\varphi^n)^* \mu_{0}$. D'apr\`es la remarque~\ref{rem:pullbackproba}, pour tout $y\in Y$, $\mu_{n,y}$ est une mesure de probabilit\'e sur~$X_{y}$. D'apr\`es le lemme~\ref{lem:imagemesurecontinue}, $\mu_{n}$ est encore une $\pi$-famille de mesures continue.

\begin{lemm}
Supposons que $\cA$ est flottant. Soit~$V$ un ouvert flottant de~$Y$ et supposons que l'endomorphisme $\varphi_{\vert X_{V}}$ de~$X_{V}$ est flottant. Alors la famille de mesures~$\mu_{n}$  est flottante.
\end{lemm}
\begin{proof}
Le r\'esultat d\'ecoule du lemme~\ref{lem:Phichi} et des d\'efinitions.
\end{proof}

La relation entre~$\nm_{\st}$ et~$\chi_{0,1}$ \'enonc\'ee au lemme~\ref{lem:Deltalogs} se transf\`ere \`a~$\nm_{n}$ et~$\mu_{n}$.

\begin{prop}\label{prop:Deltalogsn}
Soient $y\in Y$ et $n \in \N$. Pour toute $s \in \Gamma(X_{y},\cO_{X_{y}}(1))$, la fonction $x\mapsto -\log(\norm{s}_{n,x})$ est sous-harmonique continue sur $X_{y} \setminus \div(s)$ et on a 
\[\Delta (-\log(\norm{s}_{n,\vert X_{y}})) = \mu_{n} - \delta_{\div(s)} \textrm{ dans } \Mes(X_{y}).\]
\end{prop}
\begin{proof}
D\'emontrons le r\'esultat par r\'ecurrence sur $n\in \N$. Le cas $n=0$ n'est autre que le r\'esultat du lemme~\ref{lem:Deltalogs}.

Soit $n\in \N$ et supposons que le r\'esultat est v\'erifi\'e pour~$\nm_{n}$. Soit  $s \in \Gamma(X_{y},\cO_{X_{y}}(1))$. D'apr\`es la remarque~\ref{rem:Phiexplicite}, pour toute $s' \in \Gamma(X_{y},\cO_{X_{y}}(1))$ et tout $x\in X_{y} \setminus \varphi^{-1}(\div(s'))$, on a
\[\norm{s}_{n+1,x}^d = \Big\lvert \frac{\theta^{-1}(s^{\otimes d})(x)}{(\varphi^*s')(x)} \Big\rvert \, \norm{s'}_{n,\varphi(x)}.\]
Soit $x\in X_{y}$. Soit $s'\in \Gamma(X_{y},\cO_{X_{y}}(1))$ telle que $s'(\varphi(x)) \ne 0$. Soit~$U$ un voisinage ouvert de~$\varphi(x)$ dans~$X_{y}$ sur lequel~$s'$ ne s'annule pas et~$\cO_{X_{y}}(1)$ est trivial. Soit~$V$ un voisinage ouvert de~$x$ dans~$\varphi^{-1}(U)$ sur lequel~$\cO_{X_{y}}(1)$ est trivial. Dans ces conditions, on peut identifier~$\theta^{-1}(s^{\otimes d})$ et $\varphi^*s'$ \`a des fonctions analytiques sur~$V$, $\varphi^*s'$ \'etant inversible. On d\'eduit alors de l'hypoth\`ese de r\'ecurrence que la fonction $z\mapsto -\log(\norm{s}_{n+1,z})$ est sous-harmonique continue sur $V \setminus \div(s)$ et, en utilisant la formule de Poincar\'e-Lelong et la propri\'et\'e de r\'etrotirette, que 
\begin{align*}
\Delta (-\log(\norm{s}^d_{n+1,\vert V})) &= \Delta (-\log(\abs{\theta^{-1}(s^{\otimes d})})_{\vert V}) + \Delta (-\log(\norm{s'}_{n}\circ \varphi)_{\vert V})\\
&=- d \,\delta_{\div(s)\cap V} + \varphi^*\mu_{n,\vert V}\\
&= d\,(\mu_{n+1,\vert V} - \delta_{\div(s)\cap V}).
\end{align*}
Le r\'esultat s'en d\'eduit.
\end{proof}

\begin{coro}\label{cor:logsphish}
Soit $y\in Y$. Pour toute $s \in \Gamma(X_{y},\cO_{X_{y}}(1))$, la fonction $x\mapsto -\log(\norm{s}_{\varphi,x})$ est sous-harmonique continue sur $X_{y} \setminus \div(s)$. 
\end{coro}
\begin{proof}
D'apr\`es le th\'eor\`eme~\ref{th:Phinconverge}, pour toute $s \in \Gamma(X_{y},\cO_{X_{y}}(1))$, la suite de fonctions $(x\mapsto -\log(\norm{s}_{n,x}))_{n\in\N}$ converge vers la fonction $x\mapsto -\log(\norm{s}_{\varphi,x})$ uniform\'ement sur tout compact de $X_{y}\setminus \div(s)$. Le r\'esultat d\'ecoule alors de la proposition~\ref{prop:Deltalogsn}.
\end{proof}

Soit $n\in\N$. Pour tout $x\in X$,  la quantit\'e $\log(\|s\|_{n,x}) - \log(\|s\|_{\st,x})$ est ind\'ependante du choix de $s\in \cO_{Y}(1)(x) \setminus\{0\}$. On peut donc d\'efinir une fonction $\lambda_{n}\colon X\to\R$ en posant, pour tout $x\in X$, 
\[\lambda_{n}(x) := \log(\|s\|_{n,x}) - \log(\|s\|_{\st,x}),\]
pour $s\in \cO_{Y}(1)(x) \setminus\{0\}$.

De m\^eme, on peut d\'efinir une fonction $\lambda_{\varphi}\colon X\to\R$ en posant, pour tout $x\in X$, 
\[\lambda_{\varphi}(x) := \log(\|s\|_{\varphi,x}) - \log(\|s\|_{\st,x}),\]
pour $s\in \cO_{Y}(1)(x) \setminus\{0\}$.

On d\'eduit du th\'eor\`eme~\ref{th:Phinconverge} que la suite de fonctions $(\lambda_{n})_{n\in\N}$ converge vers~$\lambda_{\varphi}$ uniform\'ement sur tout compact de~$X$.

\begin{lemm}\label{lem:lambdaphi}
\ 
\begin{enumerate}[i)]
\item Les fonctions~$\lambda_{n}$, pour $n\in\N$, et $\lambda_{\varphi}$ sont continues sur~$X$. 

\item Pour tout $y\in Y$, les fonctions~$\lambda_{n,\vert X_{y}}$, pour $n\in\N$, et $\lambda_{\varphi,\vert X_{y}}$ sont diff\'erences locales de fonctions sous-harmoniques continues.

\item Supposons que $\cA$ est flottant. Soit~$V$ un ouvert flottant de~$Y$ et supposons que l'endomorphisme $\varphi_{\vert X_{V}}$ de~$X_{V}$ est flottant. Alors, les fonctions~$\lambda_{n}$, pour $n\in\N$, et $\lambda_{\varphi}$ sont $\log$-flottantes.
\end{enumerate}
\end{lemm}
\begin{proof}
i) Soit $n\in\N$. Soit~$x\in X$. Il existe un voisinage ouvert~$U$ de~$x$ dans~$X$ et une section $s\in\Gamma(U,\cO_{X}(1))$ ne s'annulant pas sur~$U$. On a alors, 
\[\forall z\in U,\ \lambda_{n}(z) := \log(\|s\|_{n,z}) - \log(\|s\|_{\st,z}).\]
La continuit\'e de~$\lambda_{n}$ en~$x$ se d\'eduit alors de celles de $z\mapsto \|s\|_{n,z}$ et $z\mapsto \|s\|_{\st,z}$.

Le r\'esultat pour~$\lambda_{\varphi}$ se d\'emontre de m\^eme.

ii) Soit~$y\in Y$. On peut recouvrir~$X_{y}$ par des ouverts de la forme $X_{y}\setminus \div(s)$, avec $s\in\Gamma(X_{y},\cO_{X_{y}}(1))$. Le r\'esultat d\'ecoule alors du lemme~\ref{lem:Deltalogs}, de la proposition~\ref{prop:Deltalogsn} et du corollaire~\ref{cor:logsphish}. 

iii) C'est une cons\'equence directe du lemme~\ref{lem:nmphiflottante}.
\end{proof}

Pour tout $y\in Y$, posons
\[\mu_{\varphi,y} := \chi_{0,1} - \Delta (\lambda_{\varphi,\vert X_{y}} ) \textrm{ dans } \Mes(X_{y}).\]
Consid\'erons la famille de mesures $\mu_{\varphi} := (\mu_{\varphi,y})_{y\in Y}$.

\begin{lemm}\label{lem:muphiflottante}
Supposons que $\cA$ est flottant. Soit~$V$ un ouvert flottant de~$Y$ et supposons que l'endomorphisme $\varphi_{\vert X_{V}}$ de~$X_{V}$ est flottant. Alors la famille de mesures~$\mu_{\varphi}$ est flottante.

\end{lemm}
\begin{proof}
La propri\'et\'e est v\'erifi\'ee pour~$\chi_{0,1}$ d'apr\`es le lemme~\ref{lem:Phichi}. Le fait qu'elle soit v\'erifi\'ee pour $\Delta (\lambda_{\varphi,\vert X_{y}})$ d\'ecoule du fait que $\lambda_{\varphi}$ est $\log$-flottante (\cf~lemme~\ref{lem:lambdaphi}) et des propri\'et\'es du laplacien. %(\cf~proposition~\ref{prop:epsDelta}).
\end{proof}

\begin{prop}\label{prop:convergenceXy}
Pour tout $y\in Y$, la suite $(\mu_{n,y})_{n\in \N}$ converge vaguement vers $\mu_{\varphi,y}$.  

En particulier, pour tout $y\in Y$, la mesure~$\mu_{\varphi,y}$ est positive et de masse totale~1 et on a
\[\varphi^* \mu_{\varphi,y} = d\, \mu_{\varphi,y}.\]
\end{prop}
\begin{proof}
Soit $y\in Y$. Soit $L$ une extension valu\'ee compl\`ete de~$\cH(y)$. L'endomorphisme~$\varphi$ de $X_{y} \simeq \E{1}{\cH(y)}$ induit, par extension des scalaires, un endomorphisme~$\varphi_{L}$ de~$\E{1}{L}$. En outre, le r\'esultat sur $\E{1}{L}$ implique celui sur~$X_{y}$. Quitte \`a remplacer~$\cH(y)$ par une extension~$L$ convenable, on peut donc supposer que~$\cH(y)$ n'est pas trivialement valu\'e.

Pour tout $n\in \N$, posons $\rho_{n} := \lambda_{\varphi,\vert X_{y}}-\lambda_{n,\vert X_{y}} \in \cC(X_{y},\R)$. La fonction~$\rho_{n}$ est localement diff\'erence de fonctions sous-harmoniques continues et on a
\[\Delta \rho_{n} = \mu_{n} - \mu_{\varphi}.\]
En outre, la suite de fonctions~$(\rho_{n})_{n\in\N}$ converge uniform\'ement vers~0 sur~$X_{y}$.

Soit $f\in \C(X_{y},\R)$. Soit $\eps\in \R_{>0}$. D'apr\`es le th\'eor\`eme~\ref{th:densiteaffable}, il existe $f_{\eps} \in \MCL(X_{y},\R)$ telle que $\norm{f-f_{\eps}}_{X_{y}}\le \eps$.

Soit $n\in \N$. Par sym\'etrie du laplacien, on a 
\[\Bigabs{ \int f_{\eps} \diff\mu_{\varphi,y} -  \int f_{\eps} \diff \mu_{n,y}}  = \Bigabs{\int \rho_{n} \diff \Delta f_{\eps}}  \le \norm{\rho_{n}}_{X_{y}} \int  \diff\abs{\Delta f_{\eps}},\]
%\begin{align*}
%\Bigabs{ \int f_{\eps} \diff\mu_{\varphi,y} -  \int f_{\eps} \diff \mu_{n,y}} & = \Bigabs{\int f_{\eps} \diff\Delta \rho_{n}} \\
%& = \Bigabs{\int \rho_{n} \diff \Delta f_{\eps}} \\
%& \le \norm{\rho_{n}}_{X_{y}} \int  \diff\abs{\Delta f_{\eps}}.
%\end{align*}
d'o\`u 
\begin{align*}
\Bigabs{ \int f \diff\mu_{\varphi,y} -  \int f \diff\mu_{n,y}} &\le \Bigabs{ \int (f-f_{\eps}) \diff \mu_{\varphi,y}} + \Bigabs{ \int f_{\eps} \diff\mu_{\varphi,y} -  \int f_{\eps} \diff \mu_{n,y}} \\
& \quad  +  \Bigabs{ \int (f_{\eps}-f) \diff \mu_{n,y}}\\
&\le \eps \int \diff \abs{\mu_{\varphi,y}} +  \norm{\rho_{n}}_{X_{y}} \int  \diff\abs{\Delta f_{\eps}} +\eps.
\end{align*}
%Par cons\'equent, il existe $N\in \N$ tel que, pour tout $n\ge N$, on ait
%\[\Bigabs{ \int f \diff\mu_{\varphi,y} -  \int f \diff\mu_{n,y}}  \le \eps \Big( \int \diff \abs{\mu_{\varphi,y}} +2 \Big).\]
On en d\'eduit que la suite $(\mu_{n,y})_{n\in \N}$ converge vaguement vers $\mu_{\varphi,y}$. Le reste de l'\'enonc\'e s'en d\'eduit en passant \`a la limite.

%Pour tout $n\in \N$, la mesure~$\mu_{n}$ est positive et de masse totale~1 et on a $\mu_{n+1,y} = \frac1d\, \varphi^* \mu_{n,y}$. La seconde partie de l'\'enonc\'e s'en d\'eduit en passant \`a la limite.

%Soit $n\in \N$. Par la propri\'et\'e de sym\'etrie du laplacien, on a
%%En utilisant la proposition~\ref{prop:symetrieDelta}, on montre que
%\begin{align*}
%\Bigabs{ \int f_{\eps} \diff\mu_{\varphi,y} -  \int f_{\eps} \diff \mu_{n,y}} & = \Bigabs{\int f_{\eps} \diff\Delta \rho_{n}} \\
%& = \Bigabs{\int \rho_{n} \diff \Delta f_{\eps}} \\
%& \le \norm{\rho_{n}}_{X_{y}} \int  \diff\abs{\Delta f_{\eps}}.
%\end{align*}
%
%On a donc 
%\begin{align*}
%\Bigabs{ \int f \diff\mu_{\varphi,y} -  \int f \diff\mu_{n,y}} &\le \Bigabs{ \int (f-f_{\eps}) \diff \mu_{\varphi,y}} + \Bigabs{ \int f_{\eps} \diff\mu_{\varphi,y} -  \int f_{\eps} \diff \mu_{n,y}} \\
%& \quad  +  \Bigabs{ \int (f_{\eps}-f) \diff \mu_{n,y}}\\
%&\le \eps \int \diff \abs{\mu_{\varphi,y}} +  \norm{\rho_{n}}_{X_{y}} \int  \diff\abs{\Delta f_{\eps}} +\eps.
%\end{align*}
%Par cons\'equent, il existe $N\in \N$ tel que, pour tout $n\ge N$, on ait
%\[\Bigabs{ \int f \diff\mu_{\varphi,y} -  \int f \diff\mu_{n,y}}  \le \eps \Big( \int \diff \abs{\mu_{\varphi,y}} +2 \Big).\]
%On en d\'eduit que la suite $(\mu_{n,y})_{n\in \N}$ converge vaguement vers $\mu_{\varphi,y}$.
%
%Pour tout $n\in \N$, la mesure~$\mu_{n}$ est positive et de masse totale~1 et on a $\mu_{n+1,y} = \frac1d\, \varphi^* \mu_{n,y}$. La seconde partie de l'\'enonc\'e s'en d\'eduit en passant \`a la limite.
\end{proof}

\begin{theo}\label{th:muphicontinue}
La famille de mesures $\mu_{\varphi}$ est continue.
\end{theo}
\begin{proof}
Soient $r,s \in \intoo{1,+\infty}$ avec $r<s$. Consid\'erons la droite analytique~$\E{1}{\cA}$ munie de la coordonn\'ee~$T_{0}$, posons 
\[C_{\cA}(r,s) := \{z\in \E{1}{\cA} : r < \abs{T_{0}(z)} < s\}\] 
et $Y' := Y\times_{\cA} C_{\cA}(r,s)$. Pour tout $y'\in Y'$, on a $\abs{T_{0}(y')}>1$.

La premi\`ere projection fournit un morphisme $\pr_{1} \colon Y' \to Y$ et toutes les donn\'ees de la section~: $X$, $\varphi$, etc. s'\'etendent par changement de base \`a~$Y'$. En outre, puisque $\pr_{1}$ est continue et surjective, il suffit de d\'emontrer le r\'esultat apr\`es changement de base \`a~$Y'$. Quitte \`a remplacer~$Y$ par~$Y'$, on peut donc supposer que, pour tout $y\in Y$, il existe un voisinage ouvert~$V$ de~$y$ et $\alpha \in \cO(V)$ tel que $\abs{\alpha(y)}>1$.

Soit $y\in Y$. D\'emontrons que la famille~$\mu_{\varphi}$ est continue en~$y$. Par hypoth\`ese, il existe un voisinage ouvert~$V$ de~$y$ et $\alpha \in \cO(V)$ tel que $\abs{\alpha(y)}>1$. Quitte \`a restreindre~$V$, on peut supposer qu'il se plonge dans un espace affine analytique sur~$\cA$ et que, pour tout $z\in V$, on a $\abs{\alpha(z)}>1$. 

Soit $f\in\cC_{c}(X_{V},\R)$. Nous souhaitons montrer que la fonction
\[\fonction{I_{f,\varphi}}{V}{\R}{y}{\disp \int f \diff \mu_{\varphi,y}}\]
est continue en~$y$. Or, pour tout $n\in \N$, la fonction 
\[\fonction{I_{f,n}}{V}{\R}{y}{ \disp\int f \diff \mu_{n,y}},\]
est continue et,  d'apr\`es la proposition~\ref{prop:convergenceXy}, la suite de fonctions $(I_{f,n})_{n\in \N}$ converge simplement vers~$I_{f,\varphi}$. Pour conclure, il suffit de montrer que la convergence est uniforme sur un voisinage de~$y$.

Soit~$K$ un voisinage compact de~$y$ dans~$V$. Pour tout $n\in \N$, posons $\rho_{n} := \lambda_{\varphi,\vert X_{V}}-\lambda_{n,\vert X_{V}} \in \cC(X_{V},\R)$. La suite de fonctions~$(\rho_{n})_{n\in\N}$ converge uniform\'ement vers~0 sur~$X_{K}$.

Soit $\eps\in \R_{>0}$. D'apr\`es le th\'eor\`eme~\ref{th:densiteaffable}, il existe $f_{\eps} \in \MCL(X_{V},\R)$ telle que $\norm{f-f_{\eps}}_{X_{K}}\le \eps$. D'apr\`es la proposition~\ref{prop:borneaffable}, il existe~$M_{\eps} \in \R$ tel que,
\[\forall z \in K,\, \int \diff\abs{\Delta f_{\eps,\vert z}} \le M_{\eps}.\]

Soient $n\in \N$ et $z\in K$. En utilisant le m\^eme raisonnement que dans la preuve de la proposition~\ref{prop:convergenceXy} ainsi que la positivit\'e de~$\mu_{\varphi,z}$, on montre que
\[\abs{I_{f,\varphi}(z)-I_{f,n}(z)} \le 2\eps + \norm{\rho_{n}}_{X_{K}} \, M_{\eps}. \]
Le r\'esultat s'ensuit.
\end{proof}

\begin{rema}
On peut \'egalement proposer une preuve plus directe du th\'eor\`eme~\ref{th:muphicontinue}, suivant celle de R.~Ma\~n\'e dans le cas complexe (\cf~\cite[theorem~B]{ManeHausdorffDimension}). Elle est bas\'ee sur la caract\'erisation des mesures d'\'equilibre rappel\'ee au d\'ebut du texte, comme mesures invariantes ne chargeant pas l'ensemble exceptionnel. Consid\'erons une suite (ou une suite g\'en\'eralis\'ee) $(y_{n})_{n\in \N}$ de points de~$Y$ convergeant vers un point~$y$. Comme dans la preuve du th\'eor\`eme~\ref{th:chicontinue}, il suffit de montrer que toute valeur d'adh\'erence de~$\mu_{\varphi_{y_{n}}}$ co\"incide avec~$\mu_{\varphi_{y}}$. Soit~$\mu$ une telle valeur d'adh\'erence. Chaque $\mu_{\varphi_{y_{n}}}$ \'etant invariante par~$\varphi_{y_{n}}$, $\mu$ est invariante par~$\varphi_{y}$. Supposons, par l'absurde, que $\mu$ charge l'ensemble exceptionnel de~$\varphi_{y}$, et donc un point fixe super-attractif~$x$. Il existe alors un voisinage~$U$ de~$x$ dans~$X$ tel que tout point~$z$ de~$U$ appartienne \`a l'ensemble de Fatou de~$\varphi_{\pi(z)}$. Par cons\'equent, $x$ n'appartient pas \`a l'adh\'erence du support des~$\mu_{\varphi_{y_{n}}}$, d'o\`u la contradiction d\'esir\'ee.

Cet argument ne permet pas en revanche de d\'emontrer la continuit\'e des potentiels (\cf~proposition~\ref{prop:potentielintro}), ni celle des \'energies mutuelles (\cf~corollaire~\ref{cor:energieintro}).
\end{rema}

\backmatter

\nocite{}
\bibliographystyle{alpha}
\bibliography{../../biblio}

\end{document}